\documentclass[a4paper,reqno]{amsart}

\usepackage{amsfonts,amssymb,amscd,amsmath,latexsym,amsbsy,enumitem,stmaryrd,a4wide,mathtools,verbatim,color,subfiles,multirow}
\usepackage{mwe}
\usepackage{hyperref}
\usepackage{tikz}
\usetikzlibrary{arrows,positioning,shapes.misc}
\theoremstyle{plain}
\newtheorem{theorem}{Theorem}[section]
\newtheorem{corollary}[theorem]{Corollary}

\newtheorem{lemma}[theorem]{Lemma}
\newtheorem{proposition}[theorem]{Proposition}
\newtheorem{Definition}[theorem]{Definition}
\theoremstyle{remark}
\newtheorem{remark}[theorem]{Remark}
\numberwithin{equation}{section}
\newcommand{\R}{\mathbb R}
\newcommand{\N}{\mathbb N}
\newcommand{\C}{\mathbb C}
\newcommand{\Z}{\mathbb Z}

\newcommand{\ASIP}[1]{\textnormal{ASIP}(#1)}
\newcommand{\ASIPno}{\textnormal{ASIP}}
\newcommand{\ASIPpar}{\textnormal{ASIP}(q,\vec{k})}
\newcommand{\BEP}[1]{\textnormal{BEP}(#1)}

\newcommand{\DBS}{D^{\mathrm{B}}_{\hspace{0.04cm}\mathrm{S}}}

\newcommand{\RASIP}[1]{\textnormal{ASIP}_{\mathsf{R}}(#1)}
\newcommand{\RASIPno}{\textnormal{ASIP}_{\mathsf{R}}}
\newcommand{\RASIPpar}{\textnormal{ASIP}_\mathsf{R}(q,\vec{k},\rho)}
\newcommand{\LASIP}[1]{\textnormal{ASIP}_{\mathsf{L}}(#1)}
\newcommand{\LABEP}[1]{\textnormal{ABEP}_{\mathsf{L}}(#1)}
\newcommand{\LABEPno}{\textnormal{ABEP}_{\mathsf{L}}}
\newcommand{\LABEPpar}{\textnormal{ABEP}_{\mathsf{L}}(\sigma,\vec{k},\lambda)}

\newcommand{\LASIPno}{\textnormal{ASIP}_{\mathsf{L}}}
\newcommand{\LASIPpar}{\textnormal{ASIP}_\mathsf{L}(q,\vec{k},\lambda)}

\newcommand{\SIPno}{\textnormal{SIP}}
\newcommand{\RSIP}[1]{\textnormal{SIP}_{\mathsf{R}}(#1)}
\newcommand{\LSIP}[1]{\textnormal{SIP}_{\mathsf{L}}(#1)}

\newcommand{\genASIP}{L^{\mathrm{ASIP}}_{q,\vec{k}}}
\newcommand{\genBEP}{L^{\mathrm{BEP}}_{\vec{k}}}
\newcommand{\genRASIP}{L_{q,\vec{k},\rho}^{\mathrm{ASIP}_\mathsf{R}}}
\newcommand{\genRSIP}{L_{\vec{k},\rho}^{\mathrm{SIP}_\mathsf{R}}}
\newcommand{\genLASIP}{L_{q,\vec{k},\lambda}^{\mathrm{ASIP}_\mathsf{L}}}
\newcommand{\genLSIP}{L_{\vec{k},\lambda}^{\mathrm{SIP}_\mathsf{L}}}
\newcommand{\genLABEP}{L^{\mathrm{ABEP}_\mathsf{L}}_{\sigma,\vec{k},\lambda}}

\newcommand{\De}{\Delta}

\newcommand{\half}{\frac{1}{2}}

\newcommand{\Om}{\Omega}
\newcommand{\pitensor}{\pi_{j,j+1}}
\newcommand{\pitensortwo}{\pi_{1,2}}
\newcommand{\tensor}{\otimes}
\newcommand{\rphis}[5]{\,_{#1}\varphi_{#2} \left( \genfrac{.}{.}{0pt}{}{#3}{#4}
\ ;#5 \right)}
\newcommand{\rphisempty}[2]{\,_{#1}\varphi_{#2}}
\newcommand{\rFs}[5]{\,_{#1}F_{#2} \left( \genfrac{.}{.}{0pt}{}{#3}{#4}	\ ;#5 \right)}

\newcommand{\su}{\mathfrak{su}}
\newcommand{\U}{\mathcal U}

\newcommand{\rev}{\textnormal{rev}}

\begin{document}
\title[Dynamic Generalizations of ASIP, ABEP and their Dualities]{Dynamic generalizations of the Asymmetric Inclusion Process, Asymmetric Brownian Energy Process and their Dualities}
\author{Carel Wagenaar}
\date{\today}
\begin{abstract}
	Two new interacting particle systems are introduced in this paper: dynamic versions of the asymmetric inclusion process (ASIP) and the asymmetric Brownian energy process (ABEP). Dualities and reversibility of these processes are proven, where the quantum algebra $\U_q(\mathfrak{su}(1,1))$ and the Al-Salam--Chihara polynomials play a crucial role. Two hierarchies of duality functions are found, where the Askey-Wilson polynomials and Jacobi polynomials sit on top.
\end{abstract}
\maketitle
\section{Introduction}
	The main objective of this paper is two introduce to new stochastic processes and prove their reversibility as well as several of their dualities. The first process is a generalization of the asymmetric inclusion process (ASIP \cite{CGRSsu11}) that we call the dynamic asymmetric inclusion process (Dynamic ASIP). This process, which has one extra parameter compared to ASIP, is the inclusion version of the recently introduced (generalized) dynamic asymmetric exclusion process (dynamic ASEP, \cite{Bo,BoCo,GroeneveltWagenaarDyn}). 
	The second process is a diffusion limit of dynamic ASIP that we call dynamic ABEP. This is a generalization of the asymmetric Brownian energy process (ABEP, \cite{CGRSsu11}). \\
	\indent For both processes we show a Markov duality, a useful tool in studying interacting particle systems. Duality allows us to study one Markov process by analysing another (often simpler) one. For example, duality can help considerably in proving hydrodynamic limits of interacting particle systems (e.g. \cite{PRV}) or computing correlation functions (e.g. \cite{Sch}). Crucial for duality is finding duality functions. 	In recent years, many duality functions were found that can be expressed as products of ($q$-)hypergeometric polynomials which are orthogonal with respect to the reversible measure of the process. The advantage of such orthogonal duality functions is that they form an orthogonal basis for the Hilbert space induced by the reversible measure of the process. This can greatly simplify the expansion of observables in terms of the duality function. For example, this was used in \cite{AyaCarRed1}, \cite{AyaCarRed2} to study Bolzmann-Gibbs principles and higher order fluctuation fields, and in \cite{FlReSau} in the study of $n$-point correlation functions in non-equilibrium steady states. \\ 
	\indent In the past 15 years, many examples of hypergeometric orthogonal polynomials have been found to be duality functions for symmetric processes. Franceschini and Giardin\`a \cite{FrGi2019} showed that Krawtchouk polynomials arise as self-duality functions for the generalized symmetric exclusion process (`generalized exclusion' is sometimes also referred to as `partial exclusion'). This process has a product of binomial distributions as reversible measure, which corresponds to the orthogonality of the Krawtchouk polynomials. Other examples are Meixner polynomials, Laguerre polynomials, and Hermite polynomials (\cite{CFGGR}, \cite{FraGiaGro}, \cite{Gr2019}, \cite{ReSau2018}, \cite{Zh}). Very recently, $q$-hypergeometric orthogonal polynomials were also to found as duality functions, this time for asymmetric particle processes. Because of the asymmetry, the products of these polynomials have a nested structure, which connects them to the multivariate orthogonal polynomials of Tratnik-type from \cite{GasRahMulti}. In \cite{CFG}, Carinci, Franceschini, and Groenevelt show that certain $q$-Krawtchouk and $q$-Meixner polynomials appear as duality functions for generalized ASEP and ASIP. The results from \cite{CFG} were extended to multi-species versions of generalized ASEP in \cite{BlBuKuLiUsZh},\cite{FraKuaZho}, where also nested products of $q$-Krawtchouk polynomials appear as duality functions. In \cite{GroeneveltWagenaarDyn} Groenevelt and the author show that $q$-Racah polynomials appear as duality functions for dynamic ASEP. These polynomials sit on top of a hierachy of explicit polynomials which are orthogonal on a finite set. Many other dualities related to (a)symmetric exclusion processes can be derived as a limit of these $q$-Racah polynomials. In this paper, we aim to do something similar to \cite{GroeneveltWagenaarDyn} for the inclusion processes, where now the Askey-Wilson polynomials sit on top of this hierarchy. \\
	\indent A valuable tool in the context of stochastic duality is the use of (Lie/quantum) algebras. For many interacting particle systems, the generator can be expressed as a representation of a special element of the algebra (often the coproduct of the Casimir). Therefore, tools of the algebras can be leveraged which can e.g. help in finding symmetries of the generator, duality functions and reversible measures. In this paper, we rely on the quantum algebra $\U_q(\mathfrak{su}(1,1))$ for our results. However, many results can be stated without the algebra. Therefore, for the readers convenience, the first part of this paper can be read without knowledge on (quantum) algebras. For more information on the algebraic approach to stochastic duality, see e.g. \cite{GiRed,StuSwaVol}.\\
	
	\indent Let us describe the processes and results of this paper in a bit more detail. Dynamic ASIP is an interacting particle system on the lattice $\{1,2,..,M\}$ with closed boundary conditions where each site can hold an infinite number of particles. The process depends on three parameters: the intensity parameter $q>0$, where $q\neq 1$, the parameter $\vec{k}=(k_1,k_2,...,k_M)$ with $k_j >0$ and the dynamic parameter which we will call $\lambda$ or $\rho$.  The parameter $\vec{k}$ determines the balance between the attraction of particles to each other and the (asymmetric) random walk character of the particles. When $\vec{k}$ is close to $0$, the attraction between particles is dominant, when $\vec{k}$ goes to infinity, the random walk of the particles dominates. As is the case for generalized dynamic ASEP, we introduce two versions of dynamic ASIP which can be obtained from each other by reversing the order of sites. The `left' version of dynamic ASIP, ASIP$_\mathsf{L}$, has a boundary value $\lambda$ on the left of site $1$. The `right' version, ASIP$_\mathsf{R}$, has a boundary value $\rho$ on the right of site $M$. Through a height function, this boundary value influences the rates of all particles in the system. The parameter $q$ cannot be seen as an asymmetry parameter, as dynamic ASIP is invariant under sending $q$ to $q^{-1}$ (as is the case for generalized dynamic ASEP). By letting the dynamic parameter go to $\pm\infty$, we get back ASIP with parameter $q$ or $q^{-1}$. \\
	\indent Dynamic ASIP is reversible and has an orthogonal Markov duality with dynamic ASIP on the reversed lattice. The orthogonal duality functions are Askey-Wilson polynomials. These functions sit on top of the Askey-Scheme \cite{KLS}: a hierachy of explicit orthogonal polynomials. By considering limiting cases of this duality, we obtain many other Markov dualities. The resulting duality can be ordered into ones with and without a free parameter, see Table \ref{tab:dualityfunctionsASIP}. The latter is a parameter, independent of both Markov processes, that appears in a non-trivial way in the duality function.  In the case of the self-duality of ASIP (and limits of this), the duality functions without a free parameter are often referred to as the `classical' or `triangular' duality functions and the ones with a free parameter as the `orthogonal' ones. These names are a bit confusing for dynamic ASIP, since the duality functions without a free parameter are orthogonal and do not have a triangular structure. Some of the dualities in Table \ref{tab:dualityfunctionsASIP} are well-known, such as self duality of ASIP, and some of which are new.
	\begin{table}[h]
		\begin{tabular}{|l|ll|}
			\hline
			\multirow{2}{*}{Duality} & \multicolumn{2}{l|}{Type of duality function}                         \\  \cline{2-3} 
			& \multicolumn{1}{l|}{Free parameter} &No free parameter      \\ \hline
			ASIP$_\mathsf{L}\leftrightarrow$ASIP$_\mathsf{R}$                     & \multicolumn{1}{l|}{Askey-Wilson polynomials}            &    Special case Askey-Wilson polynomials     \\ \hline
			ASIP$_\mathsf{L}\leftrightarrow$ASIP                    & \multicolumn{1}{l|}{Big $q$-Jacobi polynomials}            & Al-Salam--Chihara polynomials         \\ \hline
			ASIP$(q)\leftrightarrow$ASIP$(1/q)$                      & \multicolumn{1}{l|}{Big $q$-Laguerre polynomials}      & Triangular  \\ \hline
			Self-duality ASIP                       & \multicolumn{1}{l|}{$q$-Meixner polynomials}      & Triangular \\ \hline
		\end{tabular}\vspace{0.1cm}\\
		\caption{Dualities of (dynamic) ASIP.}\label{tab:dualityfunctionsASIP}.\vspace{-0.8cm}
	\end{table}	

	The other process, dynamic ABEP, is a diffusion process on the lattice $\{1,...,M\}$ where each site $j$ can hold an infinite amount of energy $x_j \geq 0$. It depends on three parameters, the intensity parameter $\sigma$, the parameter $\vec{k}$ and again a boundary value $\lambda$ or $\rho$. It generalizes ABEP in the sense that taking a limit in the dynamic parameter gives back ABEP. The process has similarities with dynamic ASIP. The process is invariant under sending $\sigma$ to $-\sigma$ and the boundary value influences the rates of the process through a height function (which is slightly different from the one of dynamic ASIP). Via a deterministic non-local transformation, this process can be turned into the Brownian energy process (BEP), as is the case for ABEP. Taking a limit in the dynamic parameter gives the known transformation between ABEP and BEP. This non-local transformation can be used to carry over several results of BEP to dynamic ABEP, including duality and reversibility. Therefore, dynamic ABEP is dual to any process to which BEP is dual. Previously known was a self-duality of BEP as well as a duality with SIP. We will generalize the latter by showing that BEP has a duality with dynamic SIP, the $q=1$ version of dynamic ASIP. An overview of the dualites related to BEP can be found in Table \ref{tab:dualityfunctionsBEP}.
	\begin{table}[h]
		\begin{tabular}{|l|ll|}
			\hline
			\multirow{2}{*}{Duality} & \multicolumn{2}{l|}{Type of duality function}                         \\  \cline{2-3} 
			& \multicolumn{1}{l|}{Free parameter} &No free parameter      \\ \hline
			BEP$\leftrightarrow$ dynamic SIP                    & \multicolumn{1}{l|}{Jacobi polynomials}            &    Special case Jacobi polynomials   \\ \hline
			BEP$\leftrightarrow$SIP                    & \multicolumn{1}{l|}{Laguerre polynomials}            & Monomials        \\ \hline
			Self-duality BEP                    & \multicolumn{1}{l|}{Bessel functions}            & Bessel functions        \\ \hline
		\end{tabular}\vspace{0.1cm}\\
		\caption{Dualities of BEP.}\label{tab:dualityfunctionsBEP}.\vspace{-0.8cm}
	\end{table}
\subsection{Outlook}
	It would be interesting to find hydrodynamic limits of dynamic ASIP and dynamic ABEP. In the near future, the author intends to report on the hydrodynamic limit of the latter. It is known that the hydrodynamic limit of BEP is the heat equation \cite{PRV}. Using the non-local transformation between BEP and dynamic ABEP, we can use this result to find a hydrodynamic limit of dynamic ABEP. The resulting PDE will involve a $\tanh$ term and generalizes Burger's equation. Furthermore, so far it has been difficult to introduce a multispecies version of ASIP. Known algebraic methods how to do this for ASEP seem to fail for ASIP, since the resulting operator is not a Markov generator. It would be interesting to see whether a multispecies version of ASIP exists which has an algebraic interpretation. 
\subsection{Outline of the paper}
	The organization of this paper is as follows. In Section \ref{sec:DynASIP} we introduce a slight generalization of ASIP by allowing the parameter $k$ to vary per site. Then we are ready to introduce dynamic ASIP, where we have to be careful how to choose our dynamic parameter. In contrast with the corresponding exclusion process, generalized dynamic ASEP, where the rates are nonnegative for each value of the height function, the rates of dynamic ASIP may become negative if one is not careful in choosing the proper dynamic parameter and state space. After this, we prove the duality between ASIP and dynamic ASIP with the Al-Salam--Chihara polynomials as duality functions. This in turn implies duality of dynamic ASIP on the reversed lattice with ASIP. Therefore, we can use the scalar-product method (see e.g. \cite{CFGGR}) to show that dynamic ASIP is dual to dynamic ASIP on the reversed lattice with Askey-Wilson polynomials as duality function. Crucial here is a summation formula between Al-Salam--Chihara polynomials in base $q$ and $q^{-1}$ to the Askey-Wilson polynomials. Next, we show reversibility of dynamic ASIP, where the reversible measure comes from the orthogonality measure of the $q^{-1}$-Al-Salam--Chihara polynomials, and introduce dynamic SIP, the $q=1$ version of dynamic ASIP. \\
	\indent In Section \ref{sec:AsymmetricDegenerations}, we will take limits of the duality of dynamic ASIP with the same process on the reversed lattice to obtain the other dualities of Table \ref{tab:dualityfunctionsASIP}. \\
	\indent Then we turn our attention to dynamic ABEP in Section \ref{sec:dynABEP}. We introduce the process, some special cases of the process and show that it arises as a diffusion limit of dynamic ASIP. We also prove a duality between dynamic ABEP and SIP and use that to derive that dynamic ABEP can be turned into BEP with a non-local transformation. Using this, we can show the reversibility of dynamic ABEP using the reversibility of BEP. We end the section by giving an overview of the dualities related to BEP (and thus dynamic ABEP), where the main novelty is that we find Jacobi polynomials as duality function between BEP and dynamic SIP. \\
	\indent In sections \ref{sec:Uq11} and \ref{sec:dynasipalgebra} we will discuss the algebraic background of dynamic ASIP, which we postponed until now for the readers` convenience. We first introduce the necessary knowledge about the quantum algebra $\U_q(\mathfrak{su}(1,1))$ and the Al-Salam--Chihara polynomials in Section \ref{sec:Uq11}, before we go to the construction of the generator of dynamic ASIP in Section \ref{sec:dynasipalgebra}. The method of constructing dynamic ASIP is similar to the construction of generalized dynamic ASEP \cite{GroeneveltWagenaarDyn}, where in essence we have to replace $N_j\in\N$, the maximum number of particles that can be at site $j$, by `$-k_j$'. The Al-Salam--Chihara polynomials now play the role the $q$-Krawtchouk polynomials played in \cite{GroeneveltWagenaarDyn}. In algebraic terms, this boils down to going from the compact quantized Lie algebra $\U_q(\mathfrak{su}(2))$ to the non-compact one $\U_q(\mathfrak{su}(1,1))$. This is no surprise, since the same happens when going from ASEP to ASIP. Lastly, in the Appendix we have put two calculations which have been removed from the main text for readability. \\

\subsection{Preliminaries and notations}
	\noindent Let us start with the definition of Markov duality. Let $\{\mathcal{X}_t\}_{t\geq0}$ and $\{\widehat{\mathcal{X}}_t\}_{t\geq0}$ be Markov processes with state spaces $\Omega$ and $\widehat{\Omega}$ and generators $L$ and $\widehat{L}$. We say that $\mathcal{X}_t$ and $\widehat{\mathcal{X}}_t$ are dual to each other with respect to a duality function $D:\Omega\times\widehat{\Omega}\to \C$ if
	\[		
	[L D(\cdot,\hat{\eta})](\eta)=[\widehat{L} D(\eta,\cdot)](\hat{\eta})
	\]  
	for all $\eta\in \Omega$ and $\hat{\eta}\in\widehat{\Omega}$. If $\{\widehat{\mathcal{X}}_t\}_{t\geq0}$ is a copy of $\{\mathcal{X}_t\}_{t\geq0}$, we say that the process $\{\mathcal{X}_t\}_{t\geq0}$ is self-dual with respect to the duality function $D$.\\
	
	For future reference, we also make the following remark.
	\begin{remark}\label{rem:InvTotPart}
		Let $L$ and $\widehat{L}$ be generators of interacting particle systems where the total number of particles is preserved (as all processes in this paper will be). Denote for a vector $\eta\in \R^M$ the sum of its components by $|\eta|$, i.e.
		\[
		|\eta| = \eta_1+\ldots+\eta_M.
		\] 
		Then we have the following two basic results. 
		\begin{itemize}
			\item Let $D(\eta,\hat{\eta})$ be a duality function between the two processes. 
			If $f$ is a function only depending on parameters of the processes and the total number of (dual) particles $|\eta|$ and $|\hat{\eta}|$, then $f(|\eta|,|\hat{\eta}|)D(\eta,\hat{\eta})$ is again a duality function since $f$ is invariant under the action of both generators.
			\item Let $\mu$ be a reversible measure for the process with generator $L$, i.e.
			\[
			\mu(\eta)L(\eta,\eta') = \mu(\eta')L(\eta',\eta),
			\] 
			where $L(\eta,\eta')$ is the jump rate from the state $\eta$ to $\eta'$. Note that both sides of the above equation become zero if $|\eta|\neq |\eta'|$ since in that case $L(\eta,\eta')=0$. If $g$ is a function only depending on parameters of the process and the total number of particles $|\eta|$, then $g(|\eta|)\mu(\eta)$ is again a reversible measure since we can just multiply above detailed balance condition on both sides by $g(|\eta|)=g(|\eta'|)$.
		\end{itemize} 
	\end{remark}
	Let us introduce some notations and conventions we use throughout the paper. For discrete states of particle systems we will use the greek letters $\zeta,\eta$ and $\xi$. For continuous states we use $\vec{x}$ and $\vec{y}$. We fix a scaling parameter $q>0$ with $q\neq 1$, where we will sometimes require $q\in(0,1)$. By $\N$ we denote all positive integers,
	\[
	\Z_{\geq 0}=\N \cup \{0\}, \qquad \R_{\geq0}=\{x\in \R : x\geq 0\},  \qquad \text{and} \qquad \R_{>0}=\{x\in \R : x>0\}.
	\]
	For $a\in\R$, let 
	\[
	[a]_q=\begin{cases}
		\begin{split}&\frac{q^a-q^{-a}}{q-q^{-1}} &\text{ for } q\neq 1,\\
			&a &\text{ for } q= 1,\end{split}
	\end{cases}
	\]
	which is justified by
	\begin{align*}
		\lim\limits_{q\to1}[a]_q=a . \label{eq:limshq}
	\end{align*}
	We use standard notation for $q$-shifted factorials and $q$-hypergeometric functions as in \cite{GR}. In particular, $q$-shifted factorials are given by
	\[
	(a;q)_n = (1-a)(1-aq) \cdots (1-aq^{n-1}), 
	\]
	with $n \in \Z_{\geq0} $ and the convention $(a;q)_0=1$. For $q \in (0,1)$, we define
	\[
	 	(a;q)_\infty = \lim_{n \to \infty} (a;q)_n.
	 \] 
	The $q$-hypergeometric series $_{r+1}\varphi_r$ is given by
	\[
	\rphis{r+1}{r}{a_1,\ldots,a_{r+1}}{b_1,\ldots,b_r}{q,z} = \sum_{n=0}^\infty \frac{(a_1;q)_n \cdots (a_{r+1};q)_n}{(b_1;q)_n \cdots (b_r;q)_n} \frac{z^n}{(q;q)_n}.
	\]
	If for some $k$ we have $a_k=q^{-N}$ with $N \in \Z_{\geq 0}$, the series terminates after $N+1$ terms, since $(q^{-N};q)_n=0$ for $n>N$.
	The shifted factorials are given by
	\[
	(a)_0=1, \qquad (a)_n = a(a+1)\cdots (a+n-1), \qquad n \in \N,  
	\]
	and the hypergeometric series $_{r+1}F_r$ is defined by
	\[
	\rFs{r+1}{r}{a_1,\ldots,a_{r+1}}{b_1,\ldots,b_r}{z} = \sum_{n=0}^\infty \frac{(a_1)_n \cdots (a_{r+1})_n}{(b_1)_n \cdots (b_r)_n}\frac{z^n}{n!},
	\]
	where the series terminates if $a_k \in -\Z_{\geq0}$ for some $k$. The following limit relations hold: the $q$-shifted factorials become shifted factorials,
	\[
	\lim_{q \to 1} \frac{ (a;q)_n }{(1-q)^n} = (a)_n, \qquad a \in \R, \ n \in \Z_{\geq 0},
	\]
	and for $a_1,\ldots,a_{r}, b_1,\ldots,b_r \in \R$ and $n \in \Z_{\geq 0}$,
	\[
	\lim_{q \to 1} \rphis{r+1}{r}{q^{-n},q^{a_1}, \ldots, q^{a_{r}}}{q^{b_1},\ldots,q^{b_r}}{q,z} = \rFs{r+1}{r}{-n,a_1,\ldots,a_{r}}{b_1,\ldots,b_r}{z}.
	\]
	For an ordered $M$-tuple $\eta=(\eta_1,\ldots,\eta_M)$ we denote by $\eta^\rev$ the reversed $M$-tuple,
	\[
	\eta^\rev = (\eta_M,\ldots,\eta_1). 
	\]
	Also, for a particle state $\eta=(\eta_1,...,\eta_M)$ with $\eta_j\in \Z_{\geq 0}$ we define $\eta^{j,m}$ to be the state where a particle moved from site $j$ to site $m$ if possible and $\eta^{j,m}=\eta$ otherwise.
	
	All interacting particle processes and functions in this paper will depend on certain parameters. To simplify notation, we suppress the dependence on the parameters in notations, but occasionally add one or more parameters in the notation to stress dependency on the included parameters. \\

\section{Dynamic ASIP}\label{sec:DynASIP}
	This section will introduce dynamic ASIP by specifying its generator. Before we do so, we will introduce $\ASIP{q,\vec{k}}$, a slight generalization of $\ASIP{q,k}$ from \cite{CGRSsu11} where the parameter $k$ may now vary per site. The generators of these processes can be constructed using quantum algebra techniques, which will be done in sections \ref{sec:Uq11} and \ref{sec:dynasipalgebra}. However, for defining the processes, the quantum algebra is not necessary. We will define two versions of dynamic ASIP: a `left' version $\LASIPno$ and and a `right' version $\RASIPno$. They can be obtained from each other by reversing the order of sites. We will show that both versions are dual to ASIP and are dual to each other with a product of orthogonal hypergeometric polynomials as duality functions. After this we will show that dynamic ASIP is reversible, where the reversible measure comes from the orthogonality measure of the duality function. We end the section by introducing dynamic SIP, the $q=1$ version of dynamic ASIP and showing some dualities as well as its reversibility. 
	\subsection{Definition}\label{subsec:defdynasip}
	As said, let us start with a slight generalization of $\ASIP{q,k}$ with closed boundary conditions introduced in \cite{CGRSsu11}, where the parameter $k$ can differ per site. For $j=1,\ldots,M$, let $k_j>0$ and denote $\vec{k}=(k_1,\ldots,k_M)$. The process $\ASIP{q,\vec{k}}$ is the continuous time Markov jump process on the state space $X_d=\Z_{\geq 0}^M$, where the subscript $d$ stands for discrete states. Besides $\vec{k}$, it depends on the parameter $q$, which corresponds to the asymmetry of the process. Given a state $\eta=(\eta_j)_{j=1}^M\in X_d$, a particle on site $j$ jumps to site $j+1$ at rate
	\begin{align}
		c_j^+ = q^{n_j+k_j-n_{j+1}-1}[n_j]_q[n_{j+1}+k_{j+1}]_q,\label{eq:ratesASIP+}
	\end{align}
	and to site $j-1$ at rate
	\begin{align}
		c_{j}^-&= q^{-(n_{j}+k_{j}-n_{j-1}-1)}[n_{j}]_q[n_{j-1}+k_{j-1}]_q. \label{eq:ratesASIP-}
	\end{align}
	The Markov generator of the process, acting on functions $f \colon X_d\to\R$ with finite support, is then given by
	\[
		[\genASIP f](\eta)=\sum_{j=1}^{M-1} c_j^+[f(\eta^{j,j+1})-f(\eta)]+c_{j+1}^-[f(\eta^{j+1,j})-f(\eta)].
	\]
	The process admits a reversible product measure 
	\begin{align}
		W(\eta)=W\big(\eta;\vec{k},q\big)=q^{u(\eta,\vec{k})}\prod_{j=1}^{M} w(\eta_j;k_j,q),\label{eq:revmeasASIP}
	\end{align}
	where
	\begin{align}
		w(n;k,q)= q^{-n(k-1)}\frac{(q^{2k};q^2)_n}{(q^2;q^2)_n}\label{eq:1siterevASIP}
	\end{align}
	and
	\begin{align}
		u\big(\eta,\vec{k}\big)=\sum_{j=1}^M \eta_j\bigg(k_j + 2\sum_{i=1}^{j-1}k_i\bigg).\label{eq:defu}
	\end{align}
	This can be proven by either checking the detailed balance condition or showing that the generator is symmetric with respect to this measure on the space of function with finite support. The latter will be proven in Section \ref{sec:Uq11}.
	\begin{remark}\* \label{rem:asip}
		\begin{itemize}
			\item When taking $q=1$, one obtains SIP$(\vec{k})$, a slight generalization of the symmetric inclusion process with parameter $k>0$, SIP$(k)$, where this parameter may vary per site. In this process, a particle jumps from site $j$ to $j+1$ at rate $\eta_j(\eta_{j+1}+k_j)$ and from site $j$ to $j-1$ at rate $\eta_j(\eta_{j-1}+k_{j-1})$.
			\item Since the process is reversible, it is self-dual. When $k_1=k_2=...=k_M$, non-trivial self-duality functions where already found, including triangular \cite{CGRSsu11} as well as orthogonal ones \cite{CFG}. In Section \ref{sec:AsymmetricDegenerations} we will generalize these functions to the case where $k_j$ may vary per site.
			\item Note that reversing the order of sites, i.e. interchanging site $j$ with site $M+1-j$ so that $M\leftrightarrow 1$, $M-1 \leftrightarrow 2,$ etc., is equivalent with sending $q$ to $q^{-1}$ in the rates.
			\item ASEP, the exclusion variant of $\ASIPno$, has a symmetry where interchanging `particles' with `free spaces' is similar to sending $q\to q^{-1}$. That is, if a site `$j$' can hold a maximum of `$N_j$' particles, we replace the number of particles `$\eta_j$' by `$N_j-\eta_j$'. Since $\ASIPpar$ can formally be obtained from ASEP by replacing $N_j$ by $-k_j$, one could expect a similar symmetry to be present for $\ASIPpar$. Indeed, one has the following symmetry in the rates
			\begin{align*}
				c_j^+(q,-\eta-\vec{k})&=c_{j+1}^-(q^{-1},\eta), \\
				c_{j+1}^-(q,-\eta-\vec{k})&=c_{j}^+(q^{-1},\eta).
			\end{align*}
			This can, for example, be used to obtain duality for $\ASIP{q,\vec{k}}\leftrightarrow\ASIP{q^{-1},\vec{k}}$ from the self-duality of $\ASIPpar$.
		\end{itemize}
	\end{remark}
	Let us now define dynamic ASIP, an interacting particle system on the statespace $X_d$ which generalizes ASIP. Constructing the generator is done using quantum algebra techniques in Section \ref{sec:dynasipalgebra}. However, the process and results can be stated without knowledge of quantum algebras. \\
	Dynamic ASIP is the inclusion variant of generalized dynamic ASEP \cite{GroeneveltWagenaarDyn} and has 3 parameters: $q$, $\vec{k}$ and one boundary value on either the left or right side. If the boundary value is on the left we will call it $\lambda$ and on the right $\rho$. Analogous to dynamic ASEP, the jump rates of dynamic ASIP will be a product of the ASIP rates and a factor containing a `height function'. Compared to dynamic ASEP, we have to replace `$N_j$' by `$-k_j$' in this height function. Let $\eta\in X_d$ be a particle configuration. There are two ways to define this height function, following either the rule
	\begin{align}
		h_j(\eta)=h_{j-1}+2\eta_j + k_j,\indent \text{(`left to right')} \label{eq:heightlefttoright}
	\end{align}
	or
	\begin{align}
		h_j(\eta)=h_{j+1}+2\eta_j + k_j.\indent \text{(`right to left')}. \label{eq:heightrighttoleft}
	\end{align}
	For \eqref{eq:heightlefttoright}, we introduce a left boundary value $\lambda$ at the virtual site $0$,
	\[
		 h_0^-=\lambda,
	\] 
	and work inductively from the left-most site towards site $j$ to obtain
	\[
		h_j^-(\eta)=\lambda + \sum_{m=1}^j 2\eta_m + k_m.
	\]
	The superscript `$-$' indicates we sum over sites lower than and including $j$. Note that
	\begin{align}
		h_0^-(\eta) < h_1^-(\eta) < ... < h_M^-(\eta).\label{eq:height-monotone}
	\end{align}
	\noindent For \eqref{eq:heightrighttoleft}, we work in the opposite direction. We introduce the right boundary value $\rho$ at the virtual site $M+1$,
	\[
		h_{M+1}^+=\rho,
	\]
	and go from the right-most site towards site $j$ to obtain
	\[
		h_j^+(\eta)=\rho + \sum_{i=j}^M 2\eta_i + k_i.
	\]
	This height function is monotone in the other direction,
	\[
			h_{M+1}^+(\eta)<	h_M^+(\eta) < ... < 	h_1^+(\eta).
	\]
	We will define a `left' and `right' version of dynamic ASIP. These two versions are the same process, but where the order of sites is reversed. Note that both height functions $h^-$ and $h^+$ can also be obtained from each other by reversing the order of sites. To make sure the rates of dynamic ASIP are non-negative, we have to put restrictions on the parameters of the process. The most natural restriction is to require that $\lambda,\rho >-1$, see also the remark after the definition. 
	\begin{Definition}[Dynamic ASIP for $\lambda,\rho > -1$]\*
		\begin{enumerate}
			\item Let $\lambda > -1$. ASIP$_\mathsf{L}(q,\vec{k},\lambda)$, the left version of dynamic ASIP, is the Markov process on the statespace $X_d$ with parameters $q$, $\vec{k}\in \R_{>0}^M$ and left boundary value $\lambda$. Its generator, acting on functions $f\colon X_d\to \R$ with finite support, is given by
			\[
				\Big[\genLASIP f\Big](\zeta) = \sum_{j=1}^{M-1}C_j^{\mathsf{L},+}\big[f(\zeta^{j,j+1})-f(\zeta)\big]+C_{j+1}^{\mathsf{L},-}\big[f(\zeta^{j+1,j})-f(\zeta)\big],
			\]
			where
			\begin{align}
				\begin{split}C_j^{\mathsf{L},+}&=C_j^{\mathsf{L},+}(\zeta,\vec{k},\lambda,q)=c_{j}^+(\zeta) \frac{\big(1-q^{2h^-_{j-1}+2\zeta_{j}}\big)\big(1-q^{2h^-_j+2\zeta_{j+1}}\big)}{\big(1-q^{2h^-_j}\big)\big(1-q^{2h^-_j-2}\big)},\\
				C_j^{\mathsf{L},-}&=	C_j^{\mathsf{L},-}(\zeta,\vec{k},\lambda,q)=c_j^-(\zeta) \frac{\big(1-q^{2h^-_{j-1}-2\zeta_{j-1}}\big)\big(1-q^{2h^-_{j}-2\zeta_{j}}\big)}{\big(1-q^{2h^-_{j-1}}\big)\big(1-q^{2h^-_{j-1}+2}\big)},\end{split} \label{eq:ratesDynASIPL}
			\end{align}
			where $c_j^+$ and $c_j^-$ are the rates of $\ASIP{q,\vec{k}}$ from \eqref{eq:ratesASIP+} and \eqref{eq:ratesASIP-}.
			\item Let $\rho>-1$. ASIP$_\mathsf{R}(q,\vec{k},\rho)$, the right version of dynamic ASIP, is the Markov process on the statespace $X_d$ with parameters $q$, $\vec{k}\in \R_{>0}^M$ and right boundary value $\rho$. Its generator, acting on functions $f\colon X_d\to \R$ with finite support, is given by
			\[
				\Big[\genRASIP f\Big](\xi)=\sum_{j=1}^{M-1}C_j^{\mathsf{R},+}\big[f(\xi^{j,j+1})-f(\xi)\big]+C_{j+1}^{\mathsf{R},-}\big[f(\xi^{j+1,j})-f(\xi)\big],
			\]
			where 
			\begin{align}
				\begin{split}C_j^{\mathsf{R},+}&=c^+_j(\xi) \frac{\left(1-q^{2\xi_j-2h^+_{j}}\right)\Big(1-q^{2\xi_{j+1}-2h^{+}_{j+1}}\Big)}{\Big(1-q^{-2h^{+}_{j+1}}\Big)\Big(1-q^{-2h^+_{j+1}-2}\Big)},\\
				C^{\mathsf R,-}_{j}(\xi) &= c^-_{j}(\xi) \frac{\Big(1-q^{-2\xi_{j-1}-2h^+_j}\Big)\Big(1-q^{-2\xi_{j}-2h^+_{j+1}}\Big)}{\Big(1-q^{-2h^+_j}\Big)\Big(1-q^{-2h^+_{j}+2}\Big)},\end{split} \label{eq:ratesDynASIPR}
			\end{align}
			where $c_j^+$ and $c_j^-$ are the rates of $\ASIP{q,\vec{k}}$ from \eqref{eq:ratesASIP+} and \eqref{eq:ratesASIP-}.
		\end{enumerate}
	\end{Definition}
	\begin{remark}	\*
		\begin{itemize}
			\item The left and right version of dynamic ASIP can be obtained from each other by reversing the order of sites. That is, by interchanging sites $1\leftrightarrow M$, $2\leftrightarrow M-1$, etc. 
			\item 	One can rewrite the rates for ASIP$_\mathsf{L}(q,\vec{k},\lambda)$ into
			\begin{align}
				\begin{split}
				C_j^{\mathsf{L},+}&=[\zeta_j]_q[\zeta_{j+1}+k_{j+1}]_q\frac{[h^-_{j-1}+\zeta_{j}]_q[h^-_{j}+\zeta_{j+1}]_q}{[h^-_j]_q[h^-_j-1]_q},\\
				C_j^{\mathsf{L},-}&=[\zeta_{j}]_q[\zeta_{j-1}+k_{j-1}]_q\frac{[h^-_{j-1}-\zeta_{j-1}]_q[h^-_j-\zeta_{j}]_q}{[h^-_{j-1}]_q[h^-_{j-1}+1]_q}.
				\end{split}\label{eq:ratesdynASIPrewritten}
			\end{align}
			From this we can see that the rates are non-negative if $\lambda>-1$ and that the process is invariant under sending $q\to q^{-1}$. Similarly for ASIP$_\mathsf{R}(q,\vec{k},\rho)$.
		\end{itemize}
	\end{remark}
	The question arises whether we can also define a version of dynamic ASIP when we pick the boundary value $\lambda,\rho\leq-1$. In that case, the rates might become negative. We can solve this problem by restricting our statespace $X_d$ to configurations where the rates are non-negative. For $\LASIPno$, if we only allow states $\zeta$ such that 
	\[
		 2|\zeta|+|\vec{k}|+\lambda < 0 ,
	\]
	then $h^-_j(\zeta) < 0$ for all $j$ by \eqref{eq:height-monotone}. One can then see from the rewritten rates $\eqref{eq:ratesdynASIPrewritten}$ that $C_j^{\mathsf{L},+}(\zeta)$ and $C_j^{\mathsf{L},-}(\zeta)$ will be nonnegative. An entirely similar reasoning shows that the rates $C_j^{\mathsf{R},+}(\xi)$ and $C_j^{\mathsf{R},-}(\xi)$ are nonnegative if
	\[
		2|\xi|+|\vec{k}|+\rho<0.
	\] 
	Therefore we can define $\LASIP{q,\vec{k},\lambda}$ and $\RASIP{q,\vec{k},\rho}$ when $\lambda,\rho \leq -1$ on the statespaces $X_{d,\lambda}$ and $X_{d,\rho}$ respectively, where
	\[X_{d,a}=\{\eta\in X_d :  2|\eta|+ |\vec{k}| + a<0\}.\]
	\begin{Definition}[Dynamic ASIP for $\lambda,\rho \leq -1$]\*
		\begin{enumerate}
			\item Let $\lambda < -2$, then ASIP$_\mathsf{L}(q,\vec{k},\lambda)$ is the Markov process on the statespace $X_{d,\lambda}$ where in the state $\zeta$, a particle jumps from site $j$ to $j+1$ with rate $C_j^{\mathsf{L},+}(\zeta)$ and from site $j$ to $j-1$ with rate $C_j^{\mathsf{L},-}(\zeta)$ from \eqref{eq:ratesDynASIPL}.
			\item Let $\rho \leq -1$, then ASIP$_\mathsf{R}(q,\vec{k},\rho)$ is the Markov process on the statespace $X_{d,\rho}$ where in the state $\xi$, a particle jumps from site $j$ to $j+1$ with rate $C_j^{\mathsf{R},+}(\xi)$ and from site $j$ to $j-1$ with rate $C_j^{\mathsf{R},-}(\xi)$ from \eqref{eq:ratesDynASIPR}.
		\end{enumerate} 
	\end{Definition} 
	\begin{remark}\*
		\begin{itemize}
			\item Note that when $|\vec{k}| + a \geq -2$, the set $X_{d,a}$ is either empty or only contains the state with zero particles. Therefore, we only have a non-trivial process if $\lambda,\rho< -2$ and $\vec{k}$ is chosen such that $X_{d,a}$ contains at least the states with $1$ particle. 
			\item If $0<q<1$, the limit $\lambda\to\infty$ of $\LASIP{q,\vec{k},\lambda}$ is $\ASIP{q,\vec{k}}$ which can be seen by looking at the rates. Similarly, the limit $\lambda\to -\infty$ of $\LASIPpar$ is $\ASIP{q^{-1},\vec{k}}$ and
			\begin{align*}
				&\lim\limits_{\rho\to\infty} \RASIP{q,\vec{k},\rho} = \ASIP{q^{-1},\vec{k}}, \\
				&\lim\limits_{\rho\to -\infty} \RASIP{q,\vec{k},\rho} = \ASIP{q,\vec{k}}.
			\end{align*}
			If $q>1$, we can use that dynamic ASIP is invariant under sending $q$ to $q^{-1}$ and above limits to show that
			\begin{align*}
				&\lim\limits_{\lambda\to\sigma\infty} \LASIP{q,\vec{k},\lambda} = \ASIP{q^{-\sigma},\vec{k}}, \\
				&\lim\limits_{\rho\to \sigma\infty} \RASIP{q,\vec{k},\rho} = \ASIP{q^{\sigma },\vec{k}},
			\end{align*}
			where $\sigma\in\{-1.1\}$.
		\end{itemize}
	\end{remark}
	\subsection{Duality dynamic ASIP with ASIP}
			Next we the duality between ASIP and dynamic ASIP. The proof will be done in Section \ref{sec:dynasipalgebra}. The duality function is a product of 1-site duality functions given in terms of the Al-Salam--Chihara polynomials \cite[\S14.8]{KLS}. We define a renormalized version of these polynomials by a $\rphisempty{3}{2}$,
		\begin{align}
			p_{\mathrm{A}}^{}(n,x;\rho;k;q) = q^{-\half n(2\rho+k+1)} \rphis{3}{2}{q^{-2n},q^{-2x},q^{2x+2\rho+2k}}{q^{2k}, 0}{q^2;q^2}. \label{eq:1siteASCq}
		\end{align}
		Throughout this paper, the first two entries of $p_\mathrm{A}^{}$ will be $\Z_{\geq 0}$-valued. Consequently, the $\rphisempty{3}{2}$ terminates either after $n$ or after $x$ terms, whichever is lower. The $\rphisempty{3}{2}$ part of $p_\mathrm{A}^{}$ has a `dual structure' in the following sense. It can be seen as a polynomial of degree $x$ in the variable $q^{-2n}$  since the $q$-shifted factorial
		\[
		(q^{-2n};q^2)_j
		\] 
		is a polynomial of degree $j$ in the variable $q^{-2n}$. However, it can also be seen as a polynomial of degree $n$ in the variable $q^{-2x}+q^{2x+2\rho+2k}$ since 
		\[
		(q^{-2x};q^2)_j(q^{2x+2\rho+2k};q^2)_j
		\]
		is a polynomial of degree $j$ in the variable $q^{-2x}+q^{2x+2\rho+2k}$. A nested product of $p^{}_{\mathrm{A}}$ is a duality function between ASIP and dynamic ASIP. 
		\begin{theorem}\label{thm:ASCduality}
			Define 
			\begin{align*}
				P^{}_\mathsf{R}(\eta,\xi)= q^{-\half u(\eta,\vec{k})}\prod_{j=1}^M  p_{\mathrm{A}}^{}(\eta_j,\xi_j;h^+_{j+1};k_j;q),
			\end{align*}
			where the factor $u\big(\eta,\vec{k}\big)$ can be found in \eqref{eq:defu}. Then $P^{}_\mathsf{R}(\eta,\xi)$ is a duality function between $\ASIP{q,\vec{k}}$ and $\RASIP{q,\vec{k},\rho}$, i.e.
			\[
			\Big[\genASIP P^{}_\mathsf{R}(\cdot,\xi)\Big](\eta)=\Big[\genRASIP P^{}_\mathsf{R}(\eta,\cdot)\Big](\xi).
			\]
		\end{theorem}
		\begin{proof}
			See Section \ref{sec:dynasipalgebra}.
		\end{proof}
		\begin{remark}
			The factor containing $u\big(\eta,\vec{k}\big)$ can also be put into the 1-site duality function since
			\[
				q^{-\half u(\eta,\vec{k})} = \prod_{j=1}^M q^{-\half \eta_j (k_j+2\sum_{i=1}^{j-1}k_i)}.
			\]
			Together with the height function $h_{j+1}^+$, which depends on $\xi_{j+1},...,\xi_M$, this makes $P^{}_\mathsf{R}$ a nested product of 1-site duality functions. Note that in the limit $q\to 1$ this nested structure disappears. 
		\end{remark}
		From the duality between ASIP and $\RASIPno$ we can also find duality functions between ASIP and ASIP$_\mathsf{L}$, using the following observations.
		\begin{itemize}
			\item $\LASIP{q,\vec{k},\lambda}$ is the same as $\RASIP{q,\vec{k}^\rev,\lambda}$ where the order of sites is reversed.
			\item If we reverse the order of sites of $\ASIP{q,\vec{k}}$, we obtain $\ASIP{q^{-1},\vec{k}^\rev}$.
			\item The rates of dynamic ASIP are invariant under sending $q$ to $q^{-1}$.
		\end{itemize}
		Therefore,  if we reverse the order of sites and replace $q$ by $q^{-1}$ in Theorem \ref{thm:ASCduality}, we obtain duality functions between ASIP and ASIP$_\mathsf{L}$.
		\begin{corollary}\label{cor:dualityLASIPASIP}
			Define 
			\begin{align*}
				P^{}_\mathsf{L}(\eta,\zeta)= q^{\half u(\eta,\vec{k})}\prod_{j=1}^M  p_{\mathrm{A}}^{}(\eta_j,\zeta_j;h^-_{j-1};k_j;q^{-1}).
			\end{align*}
			Then $P^{}_\mathsf{L}(\eta,\zeta)$ are duality functions between $\ASIP{q,\vec{k}}$ and $\LASIP{q,\vec{k},\lambda}$, i.e.
			\[
			\Big[\genASIP P^{}_\mathsf{L}(\cdot,\zeta)\Big](\eta)=\Big[\genLASIP P^{}_\mathsf{L}(\eta,\cdot)\Big](\zeta).
			\]
		\end{corollary}
	\subsection{Duality $\LASIPno$ and $\RASIPno$}
		Since both $\LASIPno$ and $\RASIPno$ are dual to the reversible process $\ASIPno$, they are dual to each other. Via the scalar-product method \cite[Proposition 4.1]{CFGGR}, duality functions can be obtained. The resulting duality functions are given by a nested product of Askey-Wilson polynomials, which can be defined by a $\rphisempty{4}{3}$. We give a sketch of the proof here, the details can be found in Section \ref{sec:dynasipalgebra}.

	\begin{theorem}\label{thm:AWduality}
		Define for $v\neq 0$ the 1-site duality function
		\begin{align}
			\begin{split}p^{}_{\mathrm{AW}}(y,x;\lambda,\rho,v,k;q) =   (v&q^{\rho+\lambda+k+1};q^2)_x(vq^{-\rho-\lambda-k-1};q^{-2})_y\\ &\times \rphis{4}{3}{q^{-2y},q^{2y+2\lambda+2k},q^{-2x},q^{2x+2\rho+2k}}{q^{2k},vq^{\rho+\lambda+k+1},v^{-1}q^{\rho+\lambda+k+1}}{q^2,q^2}.\end{split}\label{eq:1sitedualityAW}
		\end{align}
		Then the function
		\begin{align}
			P^v_{\mathrm{AW}}(\zeta,\xi)=P^v_{\mathrm{AW}}(\zeta,\xi;\lambda,\rho,\vec{k};q)=\prod_{j=1}^{M}p^{}_{\mathrm{AW}}(\zeta_j,\xi_j;h_{j-1}^-(\zeta),h_{j+1}^+(\xi),v,k_j;q)\label{eq:dualityAW}
		\end{align}
		is a duality function between $\LASIP{q,\vec{k},\lambda}$ and $\RASIP{q,\vec{k},\rho}$, i.e.
		\begin{align}
			[\genLASIP P^v_{\mathrm{AW}}(\cdot,\xi)](\zeta)=[\genRASIP P^v_{\mathrm{AW}}(\zeta,\cdot)](\xi).\label{eq:duality equation dynamic ASIP}
		\end{align}
	\end{theorem}
	\begin{proof}[Sketch of proof]
		Let $L_x,L_y,L_z$ be generators of Markov processes $\mathcal{X}_t,\mathcal{Y}_t$ and $\mathcal{Z}_t$ on a countable statespace. If $\mathcal{Y}_t$ is reversible w.r.t.~ the reversible measure $w_y$, $D_1(x,y)$ is a duality function between $\mathcal{X}_t$ and $\mathcal{Y}_t$, and $D_2(y,z)$ a duality function between $\mathcal{Y}_t$ and $\mathcal{Z}_t$, then the scalar-product method says that $\mathcal{X}_t$ is dual to $\mathcal{Z}_t$ with duality function given by
		\[
			D(x,z)=\sum_y D_1(x,y)D_2(y,z)w_y(y),
		\]
		if the sum converges. Applying this in our setting tells us that
		\begin{align*}
			D(\zeta,\xi)&=\sum_{\eta\in X_d} P^{}_\mathsf{R}(\eta,\xi)P^{}_\mathsf{L}(\eta,\zeta)W(\eta)
		\end{align*}
		is a duality function between $\LASIPno$ and $\RASIPno$, where $w$ is the reversible measure of ASIP from \eqref{eq:revmeasASIP}. In Section \ref{sec:dynasipalgebra} we explicitly compute this summation to prove the theorem.
	\end{proof}
	\begin{remark}\*
		\begin{itemize}
			\item Askey-Wilson polynomials are on top of the so-called `Askey-scheme': a hierarchy of explicit orthogonal polynomials. By taking appropriate limits in these functions, one can obtain many other orthogonal polynomials, such as the Al-Salam--Chihara polynomials and Jacobi polynomials. We will exploit this in sections \ref{sec:AsymmetricDegenerations} and \ref{sec:dynABEP}.
			\item Since $\LASIPno$ and $\RASIPno$ can be obtained from each other by reversing the order of sites, one would expect a similar symmetry in the duality function $P^v_{\mathrm{AW}}$. This is indeed the case. From the definition \eqref{eq:1sitedualityAW} we obtain the symmetry
			\[
				p^{}_{\mathrm{AW}}(x,y;\lambda,\rho,v,k;q) = p^{}_{\mathrm{AW}}(y,x;\rho,\lambda,v,k;q^{-1}) 
			\]
			 for the 1-site duality function. Then, using the definition \eqref{eq:dualityAW}, we obtain\footnote{The case when there are only 2 sites gives some intuition for the computation behind this symmetry.}
			\begin{align}
				P^v_{\mathrm{AW}}(\zeta,\xi;\lambda,\rho,\vec{k};q) = P^v_{\mathrm{AW}}(\xi^\rev,\zeta^\rev;\rho,\lambda,\vec{k}^\rev;q^{-1}).\label{eq:symmetryPaw}
			\end{align}
		\end{itemize}
	\end{remark}
	\subsection{Reversibility dynamic ASIP}
	We will now show that $\LASIPpar$ and $\RASIPpar$ admit reversible measures for both $\lambda,\rho>-1$ and $\lambda,\rho\leq -1$. Throughout this section we will assume that $q<1$. This is because we will use infinite shifted factorials $(a;q)_\infty$ which are only defined for $q<1$. Since the rates of dynamic ASIP are invariant under sending $q\to q^{-1}$, we can then obtain results for dynamic ASIP for $q>1$ as well. \\
	
	The reversibility follows from the orthogonality relations of the 1-site duality functions $p^{}_\mathrm{A}$ which is the content of the following proposition. Let us first define the relevant measure,
	\begin{align*}
		w^{}_\mathsf{dyn}(z;a,k,q)= \begin{dcases}
			\frac{1-q^{4z+2a+2k}}{1-q^{2z+2a+2k}}\frac{(q^{2k};q^2)_z}{(q^2;q^2)_z}\frac{(q^{2z+2a +2};q^2)_\infty}{(q^{2z+2a+2k +2};q^2)_\infty}q^{2z(z+a)} &\text{ if }a>-1, \\
			\frac{1-q^{4z+2a+2k}}{1-q^{2z+2a+2k}}\frac{(q^{2k};q^2)_z}{(q^2;q^2)_z}\frac{(q^{-2z-2a-2k};q^2)_\infty}{(q^{-2z-2a};q^2)_\infty}q^{-2z(z+a+k)} &\text{ if }a\leq -1. \end{dcases}
	\end{align*}
		\begin{proposition}				
			We have the following orthogonality relations for the Al-Salam--Chihara polynomials $p_\mathsf{A}^{}$.
			\begin{enumerate}[label=(\roman*)]
				\item If $\lambda > -1$, we have
				\begin{align}
					&\sum_{y=0}^\infty p^{}_\mathsf{A}(n,y;\lambda,k;q^{-1}) p^{}_\mathsf{A}(n',y;\lambda,k;q^{-1}) w^{}_\mathsf{dyn}(y;\lambda,k,q) = \frac{\delta_{n,n'}}{w(n;k,q)}\label{eq:orthASC1/q}\\
					&\sum_{n=0}^\infty p^{}_\mathsf{A}(n,y;\lambda,k;q^{-1}) p^{}_\mathsf{A}(n,y';\lambda,k;q^{-1}) w(n;k,q) = \frac{\delta_{y,y'}}{w_\mathsf{dyn}^{}(y;\lambda,k,q)} \label{eq:dualorthASC1/q}
				\end{align}
				for all $n,n',y,y'\in\Z_{\geq 0}$.
				\item If $\rho \leq -1$, we have
				\begin{align}
					&\sum_{n=0}^\infty p^{}_\mathsf{A}(n,x;\rho,k;q) p^{}_\mathsf{A}(n,x';\rho,k;q) w(n;k,q) = \frac{\delta_{x,x'}}{w_\mathsf{dyn}(x;\rho,k,q)}\label{eq:dualorthASCq}
				\end{align}
				for all $x,x'\in \{x\in \Z_{\geq 0}: x+\rho + k < 0\}$.
			\end{enumerate}
		\end{proposition}	
		\begin{proof}
			The first orthogonality relation \eqref{eq:orthASC1/q} can be found in \cite{AI} and \eqref{eq:dualorthASC1/q} is the dual\footnote{For a brief explanation and proof of the dual orthogonality, see appendix \ref{app:dualorth}.} orthogonality relation of \eqref{eq:orthASC1/q}. The relation \eqref{eq:dualorthASCq} is the dual orthogonality relation of \cite[(14.8.3)]{KLS}, with $a=q^{k+\rho}$ and $b=q^{k-\rho}$.
		\end{proof}
		\begin{remark}
			If $\lambda \leq -1$, the $q^{-1}$-Al-Salam--Chihara polynomials are still orthogonal. However, the support of the measure will not be a subset of the nonnegative integers. Therefore, it will not be a reversible measure for dynamic ASIP. The same holds for the $q$-Al-Salam--Chihara polynomials when $\rho > -1$.
		\end{remark} 
		Using the previous proposition, we can show that the duality functions $P^{}_\mathsf{L}$ and $P^{}_\mathsf{R}$ are orthogonal with respect to the reversible measure $W$ for $\ASIPpar$ from \eqref{eq:revmeasASIP} and the measures 
		\begin{align}
			&W_\mathsf{L}(\zeta)=W_\mathsf{L}(\zeta;\lambda,\vec{k},q)= \prod_{j=1}^M w_\mathsf{dyn}(\zeta_j;h^-_{j-1}(\zeta),k_j,q),\\
			&W_\mathsf{R}(\xi)=W_\mathsf{R}(\xi;\rho,\vec{k},q)=\prod_{j=1}^M w_\mathsf{dyn}(\xi_j;h^+_{j+1}(\xi),k_j,q).
		\end{align}
		\begin{corollary}\* \label{cor:orthASC}
			\begin{enumerate}[label=(\roman*)]
				\item If $\lambda > -1$, the duality functions $P_\mathsf{L}^{}$ are orthogonal with respect to the measures $W_\mathsf{L}$ and $W$,
				\begin{align}
					&\sum_{\eta\in X_d} P_\mathsf{L}^{}(\eta,\zeta) P_\mathsf{L}^{}(\eta,\zeta') W(\eta) = \frac{\delta_{\zeta,\zeta'}}{W_\mathsf{L}(\zeta)} \qquad \text{for all } \zeta,\zeta'\in X_d,\label{eq:orthoP_L1}\\
					&\sum_{\zeta\in X_d} P^{}_\mathsf{L}(\eta,\zeta) P^{}_\mathsf{L}(\eta',\zeta) W_\mathsf{L}(\zeta) = \frac{\delta_{\eta,\eta'}}{W(\eta)} \qquad \text{for all } \eta,\eta'\in X_d.\label{eq:orthoP_L2}
				\end{align}
				\item If $\rho \leq -1$, the duality functions $P^{}_\mathsf{R}$ are orthogonal with respect to the measure $W$,
				\begin{align}
					&\sum_{\eta\in X_d} P^{}_\mathsf{R}(\eta,\xi) P^{}_\mathsf{R}(\eta,\xi') W(\eta) = \frac{\delta_{\xi,\xi'}}{W_\mathsf{R}(\xi)} \qquad \text{for all }\xi\in X_{d,\rho}.\label{eq:orthoP_R}
				\end{align}
			\end{enumerate}
		\end{corollary}
		\begin{remark}
			Since $\LASIPno$ and $\RASIPno$ are the same process where the order of sites is reversed, above corollary might at first glance seem odd since this symmetry between the processes is not immediately clear. This is because $\ASIPpar$ is an asymmetric process where the direction of the asymmetry depends on whether $q\in(0,1)$ or $q\in(1,\infty)$. By reversing the order of sites in \eqref{eq:orthoP_L1} and \eqref{eq:orthoP_L2}, one can show that $P^{}_\mathsf{R}(\eta,\xi;q^{-1})$ is an orthogonal duality function for $\rho > -1$ between $\RASIPpar$ and $\ASIP{q^{-1},\vec{k}}$. Similarly, one can show that $P^{}_\mathsf{L}(\eta,\zeta;q^{-1})$ is an orthogonal duality function between $\LASIPpar$ and $\ASIP{q^{-1},\vec{k}}$ if $\lambda\leq -1$.
		\end{remark}
		Since $W(\eta)$ is a reversible measure for $\ASIPno$, its generator is symmetric with respect to functions with finite support. This does not have to hold for arbitrary functions since $\genASIP$ is an unbounded operator. However, as a consequence of the orthogonal duality, we get that $\genASIP$ is symmetric with respect to the duality functions $P^{}_\mathsf{L}$ and $P^{}_\mathsf{R}$ in the weighted $L^2$ space with respect to the reversible measure $W(\eta)$. 
		\begin{proposition}\label{prop:genASIPsym}
			Let $H$ be the Hilbert space of square integrable functions on $X_d$ with respect to $W(\eta)$. That is, it has inner product
			\[
				\langle f,g\rangle_H = \sum_{\eta\in X_d} f(\eta)g(\eta)W(\eta).
			\]
			Then $\genASIP$ is symmetric with respect to the duality functions $P^{}_\mathsf{L}$ and $P^{}_\mathsf{R}$, i.e.
			\begin{align*}
				&\langle \genASIP f,g\rangle_H = \langle f,\genASIP g\rangle_H,
			\end{align*}
			for all $f,g\in \{P^{}_\mathsf{L}(\cdot,\zeta)\}_{\zeta\in X_d}\cup \{P^{}_\mathsf{R}(\cdot,\xi)\}_{\xi\in X_d}$.
		\end{proposition}
		\begin{proof}
			We will prove this for $f=P_\mathsf{L}^{}(\eta,\zeta)$ and $g=P_\mathsf{L}^{}(\eta,\zeta')$, other cases are proven entirely similar. Define for $N\in \Z_{\geq 0}$ the restriction of $P^{}_\mathsf{L}$ to states where the total number of particles is smaller or equal $N$,
			\[
				P_\mathsf{L}^N(\eta,\xi)= P_\mathsf{L}^{}(\eta,\xi)\mathbf{1}_{|\eta| \leq N}.
			\]
			Since $W(\eta)$ is a reversible measure and $P_\mathsf{L}^N(\cdot,\zeta)$ has finite support, we have 
			\[
				\langle \genASIP P^{N}_\mathsf{L}(\cdot,\zeta),P^{N}_\mathsf{L}(\cdot,\zeta')\rangle_H = \langle P^{N}_\mathsf{L}(\cdot,\zeta),\genASIP P^{N}_\mathsf{L}(\cdot,\zeta')\rangle_H.
			\]
			We take the limit $N\to\infty$ on both sides and use the dominated convergence theorem to justify taking the limit inside the infinite sum. For finding a dominating function, we will use the duality relation for $P_\mathsf{L}^{}$. First note that
			\[
				\big[\genASIP P_\mathsf{L}^N(\cdot,\zeta)\big](\eta) = \begin{cases} \big[\genASIP P_\mathsf{L}(\cdot,\zeta)\big](\eta) \qquad &\text{for } |\eta|\leq N, \\ 0 \qquad &\text{for }|\eta| > N. \end{cases}
			\]
			Therefore, we can use the duality relation from Corollary \ref{cor:dualityLASIPASIP},
			\begin{align*}
				\big|\big[\genASIP P_\mathsf{L}^N(\cdot,\zeta)\big](\eta)\big| &\leq \big|\genASIP\big[ P_\mathsf{L}(\cdot,\zeta)\big](\eta)\big| \\
				&= \big|\big[\genLASIP P_\mathsf{L}(\eta,\cdot)\big](\zeta)\big| \\
				&\leq  C(\zeta)\sum_{j=1}^{M-1} |P_\mathsf{L}(\eta,\zeta^{j,j+1})| + |P_\mathsf{L}(\eta,\zeta^{j+1,j})| + 2|P_\mathsf{L}(\eta,\zeta)|,
			\end{align*}
			where $C(\zeta)$ is the maximum over $j$ of the absolute value of the rates $C^{\mathsf{L},\pm}_j(\zeta)$ of $\LASIPno$. Since we also have
			\[
				|P_\mathsf{L}^N(\eta,\xi)| \leq |P_\mathsf{L}^{}(\eta,\xi)|,
			\]
			we obtain a dominating function
			\[
				\big|\big[\genASIP P_\mathsf{L}^N(\cdot,\zeta)\big](\eta) P_\mathsf{L}^N(\eta,\zeta')\big| \leq  |P_\mathsf{L}^{}(\eta,\zeta')|C(\zeta)\sum_{j=1}^{M-1} |P_\mathsf{L}(\eta,\zeta^{j,j+1})| + |P_\mathsf{L}(\eta,\zeta^{j+1,j})| + 2|P_\mathsf{L}(\eta,\zeta)|.
			\]
			To see that the right-hand side is integrable, use Cauchy-Schwarz and 
			\[
			\big|\big| P_\mathsf{L}^{}(\cdot,\zeta)\big|\big|_H = 1/W_\mathsf{L}(\zeta)
			\]
			by Corollary \ref{cor:orthASC}.
		\end{proof}
		Next, we proceed to showing reversibility of dynamic ASIP, which follows quite directly from the orthogonality relations of the duality functions $P^{}_\mathsf{L}$ and $P^{}_\mathsf{R}$ and the previous proposition.
		\begin{theorem}[Reversibility]\*\label{thm:DynASIPRev}
			\begin{enumerate}[label=(\roman*)]
				\item The measure $W_\mathsf{L}$ is reversible for $\LASIP{q,\vec{k},\lambda}$ on $X_d$ if $\lambda >-1$ and on $X_{d,\lambda}$ if $\lambda\leq -1$. 
				\item The measure $W_\mathsf{R}$ is reversible for $\RASIP{q,\vec{k},\rho}$ on $X_d$ if $\rho >-1$ and on $X_{d,\rho}$ if $\rho\leq -1$. 
			\end{enumerate}
		\end{theorem}
		\begin{proof}
			Let us first prove (i) for $\lambda > -1$. Fix $\zeta_1,\zeta_2\in X_d$ such that $\zeta_1\neq\zeta_2$ and define
			\[
			\delta_{\zeta_1}(\zeta)=\delta_{\zeta_1,\zeta}, \qquad \zeta\in X_d.
			\]
			Note that for reversibility, it is enough to show that
			\begin{align}
				\big\langle \genLASIP \delta_{\zeta_1}, \delta_{\zeta_2} \big\rangle_{H_\mathsf{L}} =\big\langle \delta_{\zeta_1},\genLASIP \delta_{\zeta_2} \big\rangle_{H_\mathsf{L}} \label{eq:revdelta}.
			\end{align}
			Indeed, we have
			\begin{align*}
				\big\langle \genLASIP \delta_{\zeta_1}, \delta_{\zeta_2} \big\rangle_{H_\mathsf{L}} &= \sum_{\zeta\in X_d}\big[\genLASIP\delta_{\zeta_1}\big](\zeta)\delta_{\zeta_2}(\zeta)W_\mathsf{L}(\zeta) \\
				&= W_\mathsf{L}(\zeta_2)\big[\genLASIP\delta_{\zeta_1}\big](\zeta_2) \\
				&= W_\mathsf{L}(\zeta_2) C_\mathsf{L}(\zeta_1,\zeta_2),
			\end{align*}
			where $C_\mathsf{L}(\zeta_1,\zeta_2)$ is the rate in which $\LASIPpar$ goes from state $\zeta_1$ to state $\zeta_2$. Therefore, the detailed balance condition follows from \eqref{eq:revdelta}. \\
			\\
			By \eqref{eq:orthoP_L1}, we have
			\begin{align*}
				\delta_{\zeta_1}(\zeta_2)&=  W_\mathsf{L}(\zeta_1)\sum_{\eta\in X_d} P^{}_\mathsf{L}(\eta,\zeta_1) P^{}_\mathsf{L}(\eta,\zeta_2) W(\eta),\\
				\delta_{\zeta_2}(\zeta_1)&=  W_\mathsf{L}(\zeta_2)\sum_{\eta\in X_d} P^{}_\mathsf{L}(\eta,\zeta_1) P^{}_\mathsf{L}(\eta,\zeta_2) W(\eta).
			\end{align*}
			Therefore, using that $P^{}_\mathsf{L}$ is a duality function between $\ASIPpar$ and $\LASIPpar$, and Proposition \ref{prop:genASIPsym},
			\begin{align*}
				\big[\genLASIP \delta_{\zeta_1}\big](\zeta_2) &=  W_\mathsf{L}(\zeta_1)\sum_{\eta\in X_d} P^{}_\mathsf{L}(\eta,\zeta_1) \big[\genLASIP P^{}_\mathsf{L}(\eta,\cdot)\big](\zeta_2) W(\eta) \\
				&=  W_\mathsf{L}(\zeta_1)\sum_{\eta\in X_d} P^{}_\mathsf{L}(\eta,\zeta_1) \big[\genASIP P^{}_\mathsf{L}(\cdot,\zeta_2)\big](\eta) W(\eta) \\
				&=  W_\mathsf{L}(\zeta_1)\sum_{\eta\in X_d} \big[\genASIP P^{}_\mathsf{L}(\cdot,\zeta_1) \big](\eta) P^{}_\mathsf{L}(\eta,\zeta_2) W(\eta)\\
				&=  W_\mathsf{L}(\zeta_1)\sum_{\eta\in X_d} \big[\genLASIP P^{}_\mathsf{L}(\eta,\cdot) \big](\zeta_1) P^{}_\mathsf{L}(\eta,\zeta_2) W(\eta).
			\end{align*}
			Hence
			\begin{align*}
				\big\langle \genLASIP \delta_{\zeta_1}, \delta_{\zeta_2} \big\rangle_{H_\mathsf{L}} &= W_\mathsf{L}(\zeta_2)\big[\genLASIP\delta_{\zeta_1}\big](\zeta_2)\\
				&= W_\mathsf{L}(\zeta_1)W_\mathsf{L}(\zeta_2) \sum_{\eta\in X_d} \big[\genLASIP P^{}_\mathsf{L}(\eta,\cdot) \big](\zeta_1) P^{}_\mathsf{L}(\eta,\zeta_2) W(\eta) \\
				&=  W_\mathsf{L}(\zeta_1) \big[\genLASIP \delta_{\zeta_2}\big](\zeta_1)\\
				&=\big\langle \delta_{\zeta_1}, \genLASIP \delta_{\zeta_2} \big\rangle_{H_\mathsf{L}},
			\end{align*}
			proving \eqref{eq:revdelta}. Since ASIP$_\mathsf{R}$ is just ASIP$_\mathsf{L}$ with the order of sites reversed, we also get (ii) for $\rho>-1$. The statement (i) in the case $\lambda \leq -1$ (and therefore (ii) with $\rho\leq -1$) is proven similarly. 
		\end{proof}
		\begin{remark}\*\label{rem:revdynasip}
			\begin{itemize}
				\item By taking a limit in the dynamic parameter $\lambda$ or $\rho$, dynamic ASIP becomes ASIP. One would expect that we can take this limit in the reversible measure of dynamic ASIP as well and obtain a reversible measure of ASIP. This is indeed true.  For example, we have
				\begin{align*}
					\lim\limits_{\lambda\to\infty}q^{2\lambda|\eta|}W_\mathsf{L}(\eta)=q^{|\eta|(2|\eta|-1)} W(\eta),
				\end{align*}
				which follows by a straighforward calculation.
				\item Since dynamic ASIP is invariant under sending $q$ to $q^{-1}$, the measures $W_\mathsf{L}$ and $W_\mathsf{R}$ are reversible for dynamic ASIP with $q>1$ as well. One just has to plug in $q^{-1}$ in the measures in that case.
			\end{itemize}
		\end{remark}
	The duality functions $P^{v}_\mathrm{\!AW}$ between $\LASIPno$ and $\RASIPno$ are (bi)orthogonal  with respect to the reversible measure $W_\mathsf{L}$ if $\lambda >-1$ and $\rho\leq -1$ and with respect to $W_\mathsf{R}$ if $\rho >-1$ and $\lambda \leq 1$. They are biorthogonal in the sense that $P^{v}_\mathrm{\!AW}$ is orthogonal to $P^{v^{-1}}_\mathrm{\!AW}$. The conditions on $\lambda$ and $\rho$ come from the orthogonality of $P_\mathsf{L}^{}$ and $P_\mathsf{R}^{}$ given in corollary \ref{cor:orthASC}.
	\begin{theorem}\label{thm:AWdualityOrth}\* Define the factor
		\begin{align*}
			&\omega^{v}_{\mathrm{AW}}(y,x,\lambda,\rho,q)= \frac{ (vq^{2y+|\vec{k}|+\lambda-\rho+1};q^2)_\infty(v^{-1}q^{2y+|\vec{k}|+\lambda-\rho+1};q^2)_\infty}{(vq^{-2x-|\vec{k}|+\lambda-\rho+1};q^2)_\infty(v^{-1}q^{-2x-|\vec{k}|+\lambda-\rho+1};q^2)_\infty}.
		\end{align*}
		If $\lambda >-1$ and $\rho \leq -1$, we have for all $\xi,\xi'\in X_{d,\rho}$ the orthogonality relations
		\begin{align}
			\sum_{\zeta\in X_d} P^v_{\mathrm{AW}}(\zeta,\xi;q) P^{v^{-1}}_{\mathrm{AW}}(\zeta,\xi';q) W_\mathsf{L}(\zeta,q) \omega^{v}_{\mathrm{AW}}(|\zeta|,|\xi|,\lambda,\rho,q)&= \frac{\delta_{\xi,\xi'}}{W_\mathsf{R}(\xi,q)}. \label{eq:aworthozetaq}
		\end{align}
		\ \\
		If $\lambda \leq -1$ and $\rho > -1$, they are orthogonal with respect to the reversible measure $W_\mathsf{R}$,
		\begin{align}
			\sum_{\xi\in X_d} P^v_{\mathrm{AW}}(\zeta',\xi;q^{-1}) P^{v^{-1}}_{\mathrm{AW}}(\zeta,\xi;q^{-1}) W_\mathsf{R}(\xi,q) \omega^{v}_{\mathrm{AW}}(|\xi|,|\zeta|,\rho,\lambda,q)&= \frac{\delta_{\zeta,\zeta'}}{W_\mathsf{L}(\zeta,q)}.\label{eq:aworthoxiq}
		\end{align}
	\end{theorem}
	\begin{proof}
		In Section \ref{sec:dynasipalgebra} we will show that
		\begin{align*}
			P^v_\mathrm{AW}(\zeta,\xi)= \frac{(vq^{-2|\xi|-|\vec{k}|+\lambda-\rho+1};q^2)_\infty}{(vq^{2|\zeta|+|\vec{k}|+\lambda-\rho+1};q^2)_\infty}\sum_{\eta\in X_d} v^{|\eta|}P_\mathsf{L}^{}(\eta,\zeta)P_\mathsf{R}^{}(\eta,\xi)W(\eta).
		\end{align*}
		Using this, we have that
		\begin{align*}
			\sum_{\zeta\in X_d} P^v_{\mathrm{AW}}(\zeta,\xi;q) P^{v^{-1}}_{\mathrm{AW}}(\zeta,\xi';q) W_\mathsf{L}(\zeta,q) \omega^{v}_{\mathrm{AW}}(|\zeta|,|\xi|,\lambda,\rho,q) 
		\end{align*}
		is equal to 
		\begin{align}
			\sum_{\zeta,\eta,\eta'\in X_d}v^{|\eta|-|\eta'|}P_\mathsf{L}^{}(\eta,\zeta)P_\mathsf{R}^{}(\eta,\xi)W(\eta)P_\mathsf{L}^{}(\eta',\zeta)P_\mathsf{R}^{}(\eta',\xi')W(\eta')W_\mathsf{L}(\zeta,q). \label{eq:sumzetaetaeta'}
		\end{align}
		If we first sum over $\zeta$ and use the orthogonality relation \eqref{eq:orthoP_L2} for $P^{}_\mathsf{L}$, we get that \eqref{eq:sumzetaetaeta'} is equal to
		\begin{align*}
			\sum_{\eta\in X_d}P_\mathsf{R}^{}(\eta,\xi)P_\mathsf{R}^{}(\eta',\xi')W(\eta)
		\end{align*}
		for $\lambda > -1$. Now \eqref{eq:aworthozetaq} follows from the orthogonality relation \eqref{eq:orthoP_R} for $P^{}_\mathsf{R}$, which is valid for $\rho \leq -1$.
		The second equations \eqref{eq:aworthoxiq} follows from the previous one since $\LASIPno$ and $\RASIPno$ are the same processes but with the order of sites reversed. That is, use the symmetry \eqref{eq:symmetryPaw} of $P^v_{\mathsf{AW}}$ and 
		\[
			W_\mathsf{L}(\xi^\rev,\rho)=W_\mathsf{R}(\xi,\rho). \qedhere
		\]	
	\end{proof}
	\begin{remark}\*
		\begin{itemize}
			\item The factor $\omega_{\mathrm{AW}}^v$ only depends on parameters that are kept invariant by the generators of $\LASIPno$ and $\RASIPno$. Therefore one could also choose to put this factor in the duality function $P^v_\mathrm{AW}$ or the reversible measures $W_\mathsf{L}^{}$ and $W_\mathsf{R}^{}$. However, we choose to put this factor in the orthogonality relation like this to keep the duality functions and reversible measures simpler.
			\item Since dynamic ASIP is invariant under sending $q$ to $q^{-1}$, one would expect that the requirement $q<1$ is not necessary for the orthogonality relations of the previous theorem. This is indeed true. One can show that $P^v_\mathrm{AW}$ is almost invariant under sending $(q,v)$ to $(q^{-1},v^{-1})$, 
			\begin{align*}
				P^{v^{-1}}_{\mathrm{AW}}(\zeta,\xi;q^{-1})=(-v)^{|\zeta|+|\xi|}q^{-\alpha(|\xi|)-\beta(|\zeta|)}P^{v}_\mathrm{AW}(\zeta,\xi;q),
			\end{align*}
			where
			\begin{align*}
				&\alpha(x)=-x(x-|\vec{k}|-\lambda-\rho),\\
				&\beta(y)=-y(y+|\vec{k}|+\lambda+\rho).
			\end{align*}
			Applying this symmetry to \eqref{eq:aworthozetaq}, we have for $\lambda >-1$ and $\rho\leq -1$ that
			\begin{align*}
				\hspace{1cm}\sum_{\zeta\in X_d} P^v_{\mathrm{AW}}(\zeta,\xi;q^{-1}) P^{v^{-1}}_{\mathrm{AW}}(\zeta,\xi';q^{-1}) W_\mathsf{L}(\zeta,q) \omega^{v}_{\mathrm{AW}}(|\zeta|,|\xi|,\lambda,\rho,q)q^{2\beta(|\zeta|)}&= \frac{\delta_{\xi,\xi'}}{W_\mathsf{R}(\xi,q)q^{2\alpha(|\xi|)}}.
			\end{align*}
			Similarly, from \eqref{eq:aworthoxiq} we obtain for $\lambda\leq -1$ and $\rho>-1$ that
			\begin{align*}
							\hspace{0.2cm}\sum_{\xi\in X_d} P^v_{\mathrm{AW}}(\zeta,\xi;q) P^{v^{-1}}_{\mathrm{AW}}(\zeta',\xi;q) W_\mathsf{R}(\xi,q) \omega^{v}_{\mathrm{AW}}(|\xi|,|\zeta|,\rho,\lambda,q)q^{2\beta(|\xi|)}&= \frac{\delta_{\zeta,\zeta'}}{W_\mathsf{L}(\zeta,q)q^{2\alpha(|\zeta|)}}.
			\end{align*}
			\item Only looking at the $\rphisempty{4}{3}$ part of $P^v_\mathrm{AW}$, one can see that it is invariant under sending $v$ to $v^{-1}$. Even more is true, one can show that 
					\[
						P^{v^{-1}}_{\mathrm{AW}}(\zeta,\xi)= \gamma P^v_\mathrm{AW}(\zeta,\xi),
					\]
					where $\gamma$ is a factor depending on the total number of particles of both processes. Thus we can obtain orthogonality instead of biorthogonality by putting this factor $\gamma$ into $\omega_{\mathrm{AW}}^{v}$. However, the factor $\omega_\mathrm{AW}^{v}$ would contain four more infinite shifted factorials. Therefore, we choose to keep the biorthogonality.
			\item These orthogonality relations do not correspond to the usual measure of the Askey-Wilson polynomials, which consists of an integral over the unit circle and, if the parameters satisfy certain conditions, a sum over discrete mass points. The orthogonality relations of this theorem correspond to the dual orthogonality of the Askey-Wilson polynomials when the measure has discrete mass points, see Appendix \ref{app:dualorth} for a brief explanation and proof of dual orthogonality.
		\end{itemize}
	\end{remark}
	\subsection{Dynamic SIP}
		By letting $q\to1$ in the rates of dynamic ASIP, we obtain a dynamic version of SIP$(\vec{k})$ from Remark \ref{rem:asip}. Similar to dynamic ASIP, we have to be careful with choosing the parameters in order to have non-negative rates. The height functions $h_j^+$ and $h_j^-$ are the same as the ones defined in Section \ref{subsec:defdynasip}.
		\begin{Definition}[Dynamic SIP for $\lambda,\rho > -1$] \*
			\begin{enumerate}
				\item Let $\lambda > -1$. SIP$_\mathsf{L}(q,\vec{k},\lambda)$, the left version of dynamic SIP, is the Markov process on the statespace $X_d$ with parameters $\vec{k}\in \R_+^M$ and left boundary value $\lambda$. A particle jumps from site $j$ to $j+1$ at rate
				\begin{align}
					 \frac{\zeta_j(\zeta_{j+1}+k_{j+1})(h_{j-1}^{-}(\zeta)+\zeta_j)(h_j^-(\zeta)+\zeta_{j+1})}{h_j^-(\zeta)(h_j^-(\zeta)-1)},\label{eq:LDynSIPjj+1}
				\end{align}
				and from site $j$ to site $j-1$ at rate
				\begin{align}
					\frac{\zeta_j(\zeta_{j-1}+k_{j-1})(h_{j-1}^{-}(\zeta)-\zeta_{j-1})(h_j^-(\zeta)-\zeta_{j})}{h_{j-1}^-(\zeta)(h_{j-1}^-(\zeta)+1)}. \label{eq:LDynSIPjj-1}
				\end{align}
				\item Let $\rho>-1$. SIP$_\mathsf{R}(\vec{k},\rho)$, the right version of dynamic SIP, is the Markov process on the statespace $X_d$ with parameters $\vec{k}$ and right boundary value $\rho$. A particle jumps from site $j$ to site $j+1$ at rate
				\begin{align}
					\frac{\xi_j(\xi_{j+1}+k_{j+1})(h_{j}^{+}(\xi)-\xi_j)(h_{j+1}^+(\xi)-\xi_{j+1})}{h_{j+1}^+(\xi)(h_{j+1}^+(\xi)+1)}, \label{eq:RDynSIPjj+1}
				\end{align}
				and from site $j$ to site $j-1$ at rate
				\begin{align}
					\frac{\xi_j(\xi_{j-1}+k_{j-1})(h_{j}^{+}(\xi)+\xi_{j-1})(h_{j+1}^+(\xi)+\xi_{j})}{h_j^+(\xi)(h_j^+(\xi)-1)}. \label{eq:RDynSIPjj-1}
				\end{align}
			\end{enumerate}
		\end{Definition}
		Just as before, when we want $\lambda,\rho \leq -1$, we have the same rates, but on a different statespace.
		\begin{Definition}[Dynamic SIP for $\lambda,\rho \leq -1$]\*
			\begin{enumerate}
				\item Let $\lambda \leq -1$, then SIP$_\mathsf{L}(\vec{k},\lambda)$ is the Markov process on the statespace $X_{d,\lambda}$ where in the state $\zeta$, a particle jumps from site $j$ to $j+1$ with rate given by \eqref{eq:LDynSIPjj+1} and from $j$ to $j-1$ with rate \eqref{eq:LDynSIPjj-1}.
				\item Let $\rho \leq -1$, then SIP$_\mathsf{R}(\vec{k},\rho)$ is the Markov process on the statespace $X_{d,\rho}$ where in the state $\xi$, a particle jumps from site $j$ to $j+1$ with rate given by \eqref{eq:RDynSIPjj+1} and from site $j$ to $j-1$ with rate given by \eqref{eq:RDynSIPjj-1}.
			\end{enumerate} 
		\end{Definition} 
		\begin{remark}
			When letting $\lambda\to\pm\infty$ in SIP$_\mathsf{L}(\vec{k},\lambda)$, one obtains SIP$(\vec{k})$ from Remark \ref{rem:asip}. Similarly, when letting  $\rho\to\pm\infty$ in SIP$_\mathsf{R}(\vec{k},\rho)$, one also obtains SIP$(\vec{k})$.
		\end{remark}
		By replacing $v$ by $q^{2v}$ and letting $q\to 1$ in an appropriate way in the duality relation \eqref{eq:duality equation dynamic ASIP}, we obtain duality functions between $\LSIP{\vec{k},\lambda}$ and $\RSIP{\vec{k},\rho}$. The corresponding 1-site duality functions are given in terms of Wilson polynomials, which are the $q=1$ variant of the Askey-Wilson polynomials. We have put a hat on the duality functions and measures without a parameter $q$ to distinguish between them and their counterparts that do depend on $q$.
		\begin{proposition}
			Define the 1-site duality functions
			\begin{align*}
				\hat{p}^{}_{\mathrm{W}}(y,x;\lambda,\rho;v,k) =   \big(\tfrac12\alpha+v\big)_x\big(\tfrac12\alpha-v\big)_y\rFs{4}{3}{-y,y+\lambda+k,-x,x+\rho+k}{k,\half\alpha+v,\half\alpha-v}{1},
			\end{align*}
			where $\alpha = \rho+\lambda+k+1$. Then
			\begin{align}
				\hat{P}^v_W(\zeta,\xi) = \prod_{j=1}^M \hat{p}^{}_{\mathrm{W}}(\zeta_j,\xi_j;h_{j-1}^-(\zeta),h_{j+1}^+(\xi);v,\vec{k})
			\end{align}
			is a duality function between $\LSIP{\vec{k},\lambda}$ and $\RSIP{\vec{k},\rho}$, i.e.
			\begin{align}
				[\genLSIP \hat{P}^v_{\mathrm{W}}(\cdot,\xi)](\zeta)=[\genRSIP \hat{P}^v_{\mathrm{W}}(\zeta,\cdot)](\xi).\label{eq:dualityequationdynamicSIP}
			\end{align}
		\end{proposition}
		\begin{proof}
			We take the limit $q\to 1$ in the duality relation \eqref{eq:duality equation dynamic ASIP} in an appropriate way. Since 
			\[
				\lim\limits_{q\to 1} \genLASIP = \genLSIP
			\]
			and 
			\[
			 \lim\limits_{q\to 1} \genRASIP=\genRSIP,
			\]
			we only have to compute
			\begin{align}
				\lim\limits_{q\to 1} P^v_\mathrm{AW}(\zeta,\xi)=\lim\limits_{q\to1} \prod_{j=1}^M p^{}_{\mathrm{AW}}(\zeta_j,\xi_j;h_{j-1}^-(\zeta),h_{j+1}^+(\xi),v,k;q),\label{eq:AWqto1}
			\end{align}
			where we recall that
			\begin{align*}
				\begin{split}p^{}_{\mathrm{AW}}(y,x;\lambda,\rho,v,k;q) =   (v&q^{\rho+\lambda+k+1};q^2)_x(vq^{-\rho-\lambda-k-1};q^{-2})_y\\ &\times \rphis{4}{3}{q^{-2y},q^{2y+2\lambda+2k},q^{-2x},q^{2x+2\rho+2k}}{q^{2k},vq^{\rho+\lambda+k+1},v^{-1}q^{\rho+\lambda+k+1}}{q^2,q^2}.\end{split}
			\end{align*}
			Since
			\begin{align}
				\lim\limits_{q\to1}\frac{1-q^a}{1-q^b} = \frac{a}{b},\label{eq:fracqto1}
			\end{align}
			we have
			\[
				\lim\limits_{q\to 1} \frac{(q^{2a};q^2)_j}{(q^{2b};q^2)_j}=\frac{(a)_j}{(b)_j}.
			\]
			Therefore, we have the following limit between a basic hypergeometric series and a regular hypergeometric series,
			\[
				\lim\limits_{q\to 1} \rphis{r+1}{r}{q^{-2y},q^{2a_1},...,q^{2a_{r}}}{q^{2b_1},...,q^{2b_r}}{q^2;q^2} = \rFs{r+1}{r}{-y,a_1,...,a_{r}}{b_1,...,b_r}{1}.
			\]
			In order to get a non-zero limit of \eqref{eq:AWqto1}, we replace $v$ by $q^{2v}$ and divide the duality function $P^v_{AW}$ by $(1-q^2)^{|\xi|}(1-q^{-2})^{|\zeta|}$. The latter is done since the factors in front of the $\rphisempty{4}{3}$ of the $1$-site duality function $p^{}_{\mathrm{AW}}$ from \eqref{eq:1sitedualityAW} go to $0$ when $q\to1$, Since this factor is kept invariant by both generators, we still have a duality function. Then taking the limit $q\to 1$ gives
			\begin{align*}
				\lim\limits_{q\to 1} \frac{p^{}_{\mathrm{AW}}(x,y;\lambda,\rho;q^{2v},\vec{k},q)}{(1-q^2)^{|\xi|}(1-q^{-2})^{|\zeta|}} = \hat{p}^{}_{\mathrm{W}}(y,x;\lambda,\rho;v,k),
			\end{align*}
			proving the proposition.
		\end{proof}
		From Theorem \ref{thm:DynASIPRev} and letting $q\to 1$ in the reversible measures $W_\mathsf{R}$ and $W_\mathsf{L}$ we obtain reversible measures for dynamic SIP. We define the $1$-site measure $\hat{w}^{}_\mathsf{dyn}$. This is, up to a factor depending on $|\vec{k}|$, the $q\to1$ limit of $w_\mathsf{dyn}$,
		\begin{align*}
				\hat{w}^{}_\mathsf{dyn}(z;a,k)= \begin{dcases}
				\frac{(2z+a+k)}{(z+a+k)}\frac{(k)_z}{z!}\frac{\Gamma(z+a+k+1)}{\Gamma(z+a+1)} &\text{ if }a>-1, \\
				\frac{(2z+a+k)}{(z+a+k)}\frac{(k)_z}{z!}\frac{\Gamma(-z-a)}{(\Gamma(-z-a-k)} &\text{ if }a\leq -1. \end{dcases}
		\end{align*} 
		Its product we denote by
		\begin{align}
			&\hat{W}_\mathsf{L}(\zeta)=\hat{W}_\mathsf{L}(\zeta;\lambda,\vec{k})= \prod_{j=1}^M \hat{w}_\mathsf{dyn}(\zeta_j;h^-_{j-1}(\zeta),k_j),\\
			&\hat{W}_\mathsf{R}(\xi)=\hat{W}_\mathsf{R}(\xi;\rho,\vec{k})=\prod_{j=1}^M \hat{w}_\mathsf{dyn}(\xi_j;h^+_{j+1}(\xi),k_j).
		\end{align}
		Then we have the following result.
		\begin{proposition}\*
			\begin{enumerate}[label=(\roman*)]
				\item The measure $\hat{W}_\mathsf{L}$ is reversible for $\LSIP{\vec{k},\lambda}$ on $X_d$ if $\lambda >-1$ and on $X_{d,\lambda}$ if $\lambda\leq -1$. 
				\item The measure $\hat{W}_\mathsf{R}$ is reversible for $\RSIP{\vec{k},\rho}$ on $X_d$ if $\rho >-1$ and on $X_{d,\rho}$ if $\rho\leq -1$. 
			\end{enumerate}
		\end{proposition}
		\begin{proof}
			We will show that
			\begin{align}
				\lim\limits_{q\to 1}\frac{w_\mathsf{dyn}(z;a,k,q)}{(1-q^2)^{k}} = \hat{w}_\mathsf{dyn}(z;a,k),\label{eq:limitrevmeas}
			\end{align}
			after which the result follows from Theorem \ref{thm:DynASIPRev}, the observation that
			\[
				\prod_{j=1}^M (1-q^2)^{k_j}=(1-q^2)^{|\vec{k}|}
			\]
			is kept invariant by the generator and the fact that the generator of dynamic ASIP becomes the generator of dynamic SIP in the limit $q\to 1$. \\
			\indent The only difficulty in taking the limit of $q\to1$ in $w_\mathsf{dyn}$ lies in the infinite shifted factorials. We can make sense of this limit by rewriting the infinite shifted factorials in $w_\mathsf{dyn}$ in terms of the $q$-Gamma function, which is defined by \cite[(1.10.1)]{GR},
			\[
				\Gamma_{q}(\alpha)=\frac{(q;q)_\infty}{(q^{\alpha};q)_\infty}(1-q)^{1-\alpha}, \qquad 0<q<1.
			\]
			It becomes the regular Gamma-function in the limit $q\to 1$,
			\[
				\lim\limits_{q\to1}\Gamma_{q}(\alpha)=\Gamma(\alpha).
			\] 
			Rewriting the definition of the $q$-Gamma function gives
			\[
				(q^{2\alpha};q^2)_\infty=\frac{(q^2;q^2)_\infty}{\Gamma_{q^2}(\alpha)}(1-q^2)^{1-\alpha},
			\]
			and thus
			\[
				\frac{(q^{2\alpha};q^2)_\infty}{(q^{2\beta};q^2)_\infty} = \frac{\Gamma_{q^2}(\beta)}{\Gamma_{q^2}(\alpha)}(1-q^2)^{\beta-\alpha}.
			\]
			This gives for $a>-1$,
			\[
				w_\mathsf{dyn}(z;a,k,q)= \frac{1-q^{4z+2a+2k}}{1-q^{2z+2a+2k}}\frac{(q^{2k};q^2)_z}{(q^2;q^2)_z}\frac{\Gamma_{q^2}(z+a+k+1)}{\Gamma_{q^2}(z+a+1)}(1-q^2)^{k}q^{2z(z+a)}
			\]
			Thus \eqref{eq:limitrevmeas} for $a>-1$ follows from \eqref{eq:fracqto1}. The case $a\leq -1$ is proven similarly.
		\end{proof}

\section{Limits of (orthogonal) dualities for discrete asymmetric processes.}\label{sec:AsymmetricDegenerations}
\subsection{Duality functions which follow from Theorem \ref{thm:AWduality}}
In this section, we will take limits in the duality relation between $\LASIPpar$ and $\RASIPpar$ given in Theorem \ref{thm:AWduality} to obtain other dualities. We will assume throughout this section without loss of generality that $0<q<1$. Recall from Section \ref{sec:DynASIP} the limits from dynamic ASIP to ASIP when $0<q<1$,
\begin{align}
\lim_{\rho \to \pm \infty} \genRASIP = L_{q^{\mp 1},\vec{k}}^{\mathrm{ASIP}} \qquad \text{and} \qquad \lim_{\lambda \to \pm \infty} \genLASIP = L_{q^{\pm 1},\vec{k}}^{\mathrm{ASIP}}. \label{eq:limitsgendynasip}
\end{align}
This gives us three `different\footnote{Different in the sense that they cannot be obtained from reversing the order of sites or sending $q\to q^{-1}$.}' limit cases  of the duality between ASIP$_\mathsf{L}\leftrightarrow$ASIP$_\mathsf{R}$, which are listed in Table \ref{tab:limits_duality_dynamic_ASIP}. When $q>1$, similar limits can be taken which yield the same results in terms of duality functions.
\begin{table}[h] 
	\begin{tabular}{|l|ll|l|}
		\hline
		\ &\multicolumn{2}{|l|}{\hspace{0.85cm} Duality between} & Corresponding limit \\ \hline
		(i)&\multicolumn{1}{|l|}{$\ASIP{q,\vec{k}}$}       &  $\RASIP{q,\vec{k},\rho}$      &       $\lambda\to\infty$              \\ \hline
		(ii)&\multicolumn{1}{|l|}{$\ASIP{q,\vec{k}}$ }      &  $\ASIP{q,\vec{k}}$       &  $\lambda\to\infty$, $\rho\to-\infty$                   \\ \hline
		(iii)&\multicolumn{1}{|l|}{$\ASIP{q,\vec{k}}$}       &  $\ASIP{q^{-1},\vec{k}}$       &     $\lambda,\rho\to\infty$                \\ \hline
	\end{tabular}
	\vspace{0.1cm}
	\caption{Three limits of the duality between $\LASIP{q,\vec{N},\lambda}$ and $\RASIP{q,\vec{N},\rho}$.} \label{tab:limits_duality_dynamic_ASIP}\vspace{-0.7cm}
\end{table}\\
\indent The resulting duality functions will again be products of orthogonal polynomials. We need the following three families of orthogonal polynomials, which are special cases of Askey-Wilson polynomials.
\begin{enumerate}[label=(\roman*)]
	\item The Big $q$-Jacobi polynomials \cite[\S14.5]{KLS} are defined by 
	\[
	J_n(x;a,b,c;q) = \rphis{3}{2}{q^{-n},ab q^{n+1},x}{a q, cq}{q,q}.
	\]
	We define the $1$-site duality function 
	\begin{align*}
		p^{}_{\mathrm{J}}(n,x;\lambda,\rho,v,k;q) &=  (vq^{-\lambda-\rho-k-1};q^{-2})_n J_x(q^{-2n};q^{2(k-1)},q^{2\rho},v^{-1}q^{\rho+\lambda+k-1};q^2)\\
		&=(vq^{-\lambda-\rho-k-1};q^{-2})_n \rphis{3}{2}{q^{-2x},q^{2x+2\rho+2k},q^{-2n}}{q^{2k},v^{-1}q^{\rho+\lambda+k+1}}{q^2,q^2}
	\end{align*}
	\item The $q$-Meixner polynomials \cite[\S14.13]{KLS} are defined by
	\begin{align}
	M_n(q^{-x};b,c;q)= \rphis{2}{1}{q^{-n}, q^{-x}}{bq}{q, -\frac{q^{n+1}}{c}}.
	\end{align}
	We define the $1$-site duality function
	\begin{align*}
		p^{}_{\mathrm M}(n,x;\lambda,\rho,v,k;q)&=M_{x}(q^{-2n};q^{2(k-1)},-v^{-1}q^{-\rho+\lambda-k+1};q^2)\\
		&=\rphis{2}{1}{q^{-2x},q^{-2n}}{q^{2k}}{q^2,vq^{2x+\rho-\lambda+k+1}}.
	\end{align*}
	\item The Big $q$-Laguerre polynomials \cite[\S14.11]{KLS} are defined by 
	\[
	P^{\mathrm{L}}_n(x;a,b;q) = \rphis{3}{2}{q^{-n},0,x}{aq,bq}{q,q}.
	\]
	We define the $1$-site duality function
	\begin{align*}
		p^{}_\mathrm{L}(n,x;\lambda,\rho,v,k;q)&= (v^{-1}q^{-\lambda-\rho-k-1};q^{-2})_nP^{\mathrm{ L}}_{x}(q^{-2n};q^{2(k-1)},vq^{\rho+\lambda+k-1};q^2)\\
		&= (v^{-1}q^{-\lambda-\rho-k-1};q^{-2})_n \rphis{3}{2}{ q^{-2x},0,q^{-2n}}{q^{2k},vq^{\rho+\lambda+k+1}}{q^2,q^2}.
	\end{align*}
\end{enumerate}
\ \\
As we take appropriate limits in the duality relation \eqref{eq:duality equation dynamic ASIP}, we obtain the following duality functions which correspond to the dualities given in Table \ref{tab:limits_duality_dynamic_ASIP}. Parameters of these functions will sometimes contain a height function without the dynamic parameter $\lambda$ or $\rho$. Therefore, we define
\begin{align*}
	h_{j,0}^-(\eta) =& \sum_{i=1}^j 2\eta_j + k_j,\\
	h_{j,0}^+(\eta) =& \sum_{i=j}^M 2\eta_j + k_j.
\end{align*}
Now we have the following result.
\begin{proposition}\* \label{prop:degenerate duality}
	\begin{enumerate}[label=(\roman*)]
		\item Define the function $P_{\mathrm J}^v:X_d \times X_d \to \R$ by
		\[
		P^v_{\mathrm J}(\eta,\xi) = \lim_{\lambda \to \infty} P_{\mathrm{AW}}^{vq^{\lambda}}(\eta,\xi).
		\]
		Then $P^v_{\mathrm J}$ is a duality function between $\ASIP{q,\vec{k}}\leftrightarrow\RASIP{q,\vec{k},\rho}$ and
		\[
		P^v_{\hspace{-0.035cm}\mathrm J}(\eta,\xi) =\prod_{j=1}^N p\strut_{\hspace{-0.03cm}\mathrm J}(\eta_j,\xi_j;h^{-}_{j-1,0}(\eta),h^{+}_{j+1}(\xi);v,k_j).
		\]
		\item Define the function $P^v_\mathrm{M}:X_d \times X_d \to \R$ by 
		\[
		P^v_\mathrm{M}(\eta,\xi) = 	\lim_{\rho \to -\infty} P^{vq^{-\rho}}_{\mathrm J}(\eta,\xi),
		\]
		Then $P^v_\mathrm{M}$ is a self-duality function for $\ASIP{q,\vec{k}}$ and
		\[ 
		P^v_\mathrm{M}(\eta,\xi)	= \prod_{j=1}^N p^{}_{\mathrm M}(\eta_j,\xi_j;h^{-}_{j-1,0}(\eta),h^{+}_{j+1,0}(\xi);v,k_j).
		\]
		\item Define the function $P^v_\mathrm{L}:X_d \times X_d \to \R$ by 
		\[
		P^v_\mathrm{L}(\eta,\xi) = \lim_{\rho \to \infty} P^{v^{-1} q^{\rho}}_{\mathsf R}(\eta,\xi).
		\]
		Then $P^v_\mathrm{L}$ is a duality function between $\ASIP{q,\vec{k}}\leftrightarrow \ASIP{q^{-1},\vec{k}}$ and
		\[
		P^v_\mathrm{L}(\eta,\xi) =  \prod_{j=1}^N p_\mathrm{L}^{}(\eta_j,\xi_j;h^{-}_{j-1,0}(\eta),h^{+}_{j+1,0}(\xi);v,k_j).
		\]
	\end{enumerate}
\end{proposition}
\begin{proof}
	All computation are straightforward limits which follow from
	\[
		\lim\limits_{\alpha\to-\infty} \frac{1-\beta q^{\alpha}}{1-\gamma q^{\alpha}}=\frac{\beta}{\gamma}. \qedhere 
	\]	
\end{proof}
As is the case for other situations (see e.g. \cite{GroeneveltWagenaarDyn}), one can go from `orthogonal' duality functions to `classical' (sometimes also called `triangular') duality functions by taking appropriate limits involving the free parameter `$v$'. For example, one can show that
\[
	\lim\limits_{v\to\infty}v^{-|\eta|}P_{\mathrm{J}}^v(\eta,\xi) = \gamma P^{}_\mathsf{R}(\eta,\xi),
\]
where $\gamma$ is a factor depending on preserved quantities of the process. That is, for the duality between $\ASIPno$ and $\RASIPno$, one can go from the Big $q$-Jacobi duality functions `$P_\mathrm{J}^v$' to the Al-Salam--Chihara ones $P^{}_\mathsf{R}$ from Section \ref{sec:DynASIP}. At this point, the names `orthogonal' and `triangular' duality function are confusing, since both duality functions are orthogonal and $P^{}_\mathsf{R}$ is not triangular. However, in some sense the Al-Salam--Chihara polynomials $P^{}_\mathsf{R}$ can be considered as the `triangular' duality functions between ASIP and dynamic ASIP: it does not involve a free parameter and when one would take an appropriate limit involving the dynamic parameter, one obtains the triangular self-duality function of ASIP. \\

One can also take a limit in the free parameter of the Big $q$-Laguerre polynomials (duality $\ASIPpar\leftrightarrow\ASIP{q^{-1},\vec{k}}$) and the $q$-Meixner polynomials (self-duality of $\ASIPpar$). One then obtains classical duality functions. The first is new, the latter gives a similar duality function as in \cite[Theorem 5.1]{CGRSsu11}, where now the parameter $k$ can vary per site.
\begin{proposition}[`Triangular' duality functions]\*
	\begin{enumerate}[label=(\roman*)]
		\item The function $D:X_d\times X_d\to \R$ defined by
		\begin{align*}
			D(\eta,\xi)=\prod_{j=1}^M \frac{(q^{2\xi_j+2k_j};q^2)_{\eta_j}}{(q^{2k_j};q^2)_{\eta_j}}q^{-\eta_j(\eta_j+2\xi_j+h_{j-1,0}^-(\eta)+h_{j+1,0}^+(\xi)+k_j)},
		\end{align*}
		is a duality function between $\ASIPpar$ and $\ASIP{q^{-1},\vec{k}}$.
		\item The function $D^\mathrm{tr}:X_d\times X_d\to \R$ defined by
		\begin{align*}
			D^\mathrm{tr}(\eta,\xi)= \prod_{j=1}^{M} \frac{(q^{-2\xi_j};q^2)_{\eta_j}}{(q^{2k_j};q^2)_{\eta_j}}q^{-\eta_j(\eta_j-2\xi_j+h_{j-1,0}^-(\eta)-h_{j+1,0}^+(\xi)-k_j)}1_{\eta_j\leq \xi_j}
		\end{align*}
		is a self-duality function for $\ASIPpar$.
	\end{enumerate}
\end{proposition}
\begin{proof}
	For (i), we take the limit $v\to 0$ in the Big $q$-Laguerre duality function $P_\mathrm{L}^v$ between $\ASIP{q,\vec{k}}$ and $\ASIP{q^{-1},\vec{k}}$. To make the limit convergent, we multiply by $(-v)^{|\eta|}$,
	\begin{align*}
		\lim\limits_{v\to 0} (-v)^{|\eta|}P_\mathrm{L}^v(\eta,\xi)= 	\lim\limits_{v\to 0} \prod_{j=1}^{M} (-v)^{\eta_j}&(v^{-1}q^{-h_{j-1,0}^-(\eta)-h_{j+1,0}^+(\xi)-k_j-1};q^{-2})_{\eta_j} \\
		&\times \rphis{3}{2}{q^{-2\xi_j},0,q^{-2\eta_j}}{q^{2k_j},vq^{h_{j-1,0}^-(\eta)+h_{j+1,0}^+(\xi)+k_j+1}}{q^2,q^2}.
	\end{align*}
	Since 
	\begin{align*}
		\lim\limits_{v\to 0 } (-v)^n (v^{-1}q^{-\lambda-\rho-k-1};q^{-2})_n \rphis{3}{2}{ q^{-2x},0,q^{-2n}}{q^{2k},vq^{\rho+\lambda+k+1}}{q^2,q^2}& \\
		= q^{-n(n+\lambda+\rho+k)}\rphis{2}{1}{q^{-2x},q^{-2n}}{q^{2k}}{q^2,q^2}&,
	\end{align*}
	we can use the $q$-Chu-Vandermonde summation formula \cite[(II.6)]{GR} for a $\rphisempty{2}{1}$ to prove (i).\\
	\indent The triangular duality function in (ii) follows from taking the limit $v\to\infty$ in the $q$-Meixner self-duality function $P_{\mathrm{M}}^v$ for $\ASIPpar$. In order for the limit to converge, we require $\eta \leq \xi$ and multiply by $v^{-|\eta|}$,
	\begin{align*}
		\lim\limits_{v\to\infty} \frac{P_\mathrm{M}^v(\eta,\xi)1_{\eta\leq\xi}}{v^{|\eta|}} = \lim\limits_{v\to\infty} \prod_{j=1}^M v^{-\eta_j} \rphis{2}{1}{q^{-2\xi_j},q^{-2\eta_j}}{q^{2k_j}}{q^2,vq^{2\xi_j+h_{j+1,0}^+(\xi)-h_{j-1,0}^-(\eta)+k_j+1}}1_{\eta_j\leq\xi_j}.
	\end{align*}
	For the one-site duality functions we have
	\begin{align*}
		\lim\limits_{v\to\infty} v^{-n} \rphis{2}{1}{q^{-2x},q^{-2n}}{q^{2k}}{q^2,vq^{2x+\rho-\lambda+k+1}}1_{n\leq x} = \frac{(q^{-2x},q^{-2n};q^2)_n}{(q^{2k},q^2;q^2)_n}q^{n(2x+\rho-\lambda+k+1)}1_{n\leq x}.
	\end{align*}
	Note that $1_{n\leq x}$ could be left out since $(q^{-2x};q^2)_n=0$ if $n > x$. However, we leave it in to emphasize the triangular structure of the duality function. Now (ii) follows from the observation that
	\[
		\frac{(q^{-2n};q^2)_n}{(q^2;q^2)_n} = (-1)^nq^{-n(n+1)}.\qedhere
	\] 
\end{proof}
\begin{remark}
	A short computation shows that
	\[
		D^\mathrm{tr}(\eta,-\xi-\vec{k})=D(\eta,\xi).
	\]
	This is exactly what was pointed out in Remark \ref{rem:asip}: duality between $\ASIPpar$ and $\ASIP{q^{-1},\vec{k}}$ can be obtained from the self-duality of $\ASIPpar$ by replacing `$\xi$' by `$-\xi-\vec{k}$'.\\
\end{remark}
\subsection{Orthogonality relations which follow from Theorem \ref{thm:AWdualityOrth}}
Let us now turn to the orthogonality of the duality functions found in the previous subsection. In terms of functions, the duality between $\ASIP{q^{-1},\vec{k}}$ and $\RASIPpar$ can be obtained from the duality $\ASIPpar \leftrightarrow \RASIPpar$ by sending $q\to q^{-1}$ (since dynamic ASIP is invariant under that transformation). However, this cannot be done for their orthogonality relations due to convergence problems. Therefore, we will now distinguish between the dualities $\ASIP{q^{-1},\vec{k}}\leftrightarrow \RASIPpar$ and $\ASIPpar \leftrightarrow \RASIPpar$. \\
\indent We will prove the orthogonality relations by taking limits from the ones on top in hierarchy: the duality $\LASIPpar\leftrightarrow\RASIPpar$, which has the Askey-Wilson polynomials $P_{\mathrm{AW}}^v$ as duality function. The (bi)orthogonality with respect to the reversible measures can be found in Theorem \ref{thm:AWdualityOrth}. There is a (bi)orthogonality relation with respect to the reversible measures if either (i) $\lambda >-1$ and $\rho\leq -1$ or (ii) $\lambda\leq -1$ and $\rho > -1$. Therefore, the (bi)orthogonality can only survive the following limits:
	\begin{itemize}
		\item $\lambda\to\pm\infty$,
		\item $\rho\to\pm\infty$,
		\item $\lambda\to\infty$ and $\rho\to-\infty$, 
		\item $\lambda\to-\infty$ and $\rho\to\infty$.
	\end{itemize}
	Combining this with the limits \eqref{eq:limitsgendynasip} between dynamic ASIP and ASIP, we obtain three dualities which have orthogonal duality functions, where we filtered out the ones which can be obtained from one another by reversing the order of sites, see Table \ref{tab:limits_orth_duality_dynamic_ASIP}.
	\begin{table}[h] 
		\begin{tabular}{|l|ll|l|}
			\hline
			\ &\multicolumn{2}{|l|}{\hspace{0.85cm} Duality between} & Corresponding limit \\ \hline
			(i)&\multicolumn{1}{|l|}{$\ASIP{q,\vec{k}}$}       &  $\RASIP{q,\vec{k},\rho}$      &       $\lambda\to\infty$              \\ \hline
			(ii)&\multicolumn{1}{|l|}{$\ASIP{q^{-1},\vec{k}}$ }      &  $\RASIPpar$      &  $\lambda\to-\infty$                 \\ \hline
			(iii)&\multicolumn{1}{|l|}{$\ASIPpar$}       &  $ \ASIPpar$       &     $\lambda\to\infty,\rho\to-\infty$           \\ \hline
		\end{tabular}
		\vspace{0.1cm}
		\caption{Three limits of the duality between $\LASIP{q,\vec{N},\lambda}$ and $\RASIP{q,\vec{N},\rho}$.} \label{tab:limits_orth_duality_dynamic_ASIP}\vspace{-0.7cm}
	\end{table}\\
	Note that no orthogonal duality functions between $\ASIPpar$ and $\ASIP{q^{-1},\vec{k}}$ can be found in this manner. \\
	
	Let us start with the limit $\lambda\to\infty$ in the (bi)orthogonality relation \eqref{eq:aworthozetaq}. We obtain an orthogonality relation for the Big $q$-Jacobi duality function. 
	\begin{proposition}[Orthogonality Big $q$-Jacobi]\* \label{prop:degenerate orthogonality qjac}
			If $\rho\leq-1$ we have the following orthogonality relation for the Big $q$-Jacobi duality function $P^v_{\mathrm{J}}(\eta,\xi)$ between $\ASIPpar$ and $\RASIPpar$, 
			\begin{align*}
				\sum_{\eta\in X_d} P^v_{\mathrm{J}}(\eta,\xi) P^v_{\mathrm{J}}(\eta,\xi') W(\eta)\omega_{\mathrm{J}}^{v}(|\eta|,|\xi|,\rho) = \frac{\delta_{\xi,\xi'}}{W_\mathsf{R}(\xi)} \qquad \text{for all }\xi,\xi'\in X_{d,\rho},
			\end{align*}
			where
			\begin{align*}
				\omega_{\mathrm{J}}^{v}(|\eta|,|\xi|,\rho) = v^{-2|\eta|}q^{|\eta|(2|\eta|-1)}\frac{(v^{-1}q^{|\vec{k}|+\rho+1};q^2)_{|\xi|}}{(v^{-1}q^{|\vec{k}|+\rho+1};q^{2})_{|\eta|}}\frac{(v^{-1}q^{2|\eta|+|\vec{k}|-\rho+1};q^2)_\infty}{(v^{-1}q^{-2|\xi|-|\vec{k}|-\rho+1};q^2)_\infty}.
			\end{align*}
	\end{proposition}
	\begin{proof}
		We start with the orthogonality relation \eqref{eq:aworthozetaq} with $\zeta$ replaced $\eta$,
		\[
			\sum_{\eta\in X_d} P^v_{\mathrm{AW}}(\eta,\xi;q) P^{v^{-1}}_{\mathrm{AW}}(\eta,\xi';q) W_\mathsf{L}(\eta,q) \omega^{v}_{\mathrm{AW}}(|\eta|,|\xi|,\lambda,\rho,q)= \frac{\delta_{\xi,\xi'}}{W_\mathsf{R}(\xi,q)}.
		\]
		Let us replace $v$ by $vq^\lambda$ and take the limit $\lambda\to\infty$. Justifying interchanging the infinite sum and the limit $\lambda\to\infty$ is done in Appendix \ref{app:interchangelimitsum}. Per term, we have the following limits,
		\begin{align}
				\lim\limits_{\lambda\to\infty}P^{vq^{\lambda}}_{\mathrm{AW}}(\eta,\xi)&=P_\mathrm{J}^v(\eta,\xi), \label{eq:limorthjac1}\\
				\lim\limits_{\lambda\to\infty} q^{2\lambda|\eta|} P^{v^{-1}q^{-\lambda}}_{\mathrm{AW}}(\eta,\xi) &=v^{-2|\eta|}\frac{(v^{-1}q^{|\vec{k}|+\rho+1};q^2)_{|\xi|}}{(v^{-1}q^{|\vec{k}|+\rho+1};q^{2})_{|\eta|}}P_\mathrm{J}^v(\eta,\xi),\label{eq:limorthjac2}\\
				\lim\limits_{\lambda\to\infty} q^{-2\lambda|\eta|}W_\mathsf{L}(\eta) &= q^{|\eta|(2|\eta|-1)} W(\eta)\label{eq:limorthjac3},\\
				\lim\limits_{\lambda\to\infty} \omega^{vq^\lambda}_{\mathrm{AW}}(|\eta|,|\xi|,\lambda,\rho,q) &= \frac{(v^{-1}q^{2|\eta|+|\vec{k}|-\rho+1};q^2)_\infty}{(v^{-1}q^{-2|\xi|-|\vec{k}|-\rho+1};q^2)_\infty},\label{eq:limorthjac4}
		\end{align}
		from which the proposition easily follows. The first one \eqref{eq:limorthjac1} follows by definition (see Proposition \ref{prop:degenerate duality}). The third \eqref{eq:limorthjac3} and fourth \eqref{eq:limorthjac4} are straightforward\footnote{For the third, use that $\sum_{j=1}^M \eta_j (\eta_j+2\sum_{i=1}^{j-1}\eta_i)=|\eta|^2$.} computations. \\
		\indent So we are left with proving \eqref{eq:limorthjac2}. Note that we actually obtain an orthogonality relation (whilst \eqref{eq:aworthozetaq} is biorthogonal), since we obtain some factor times $P_\mathrm{J}^v$ instead of $P_\mathrm{J}^{v^{-1}}$ in \eqref{eq:limorthjac2}. On the one-site duality function level we have
		\begin{align*}
			\lim\limits_{\lambda\to\infty}q^{2\lambda n}p^{}_{\mathrm{AW}}(n,x;\lambda,\rho,v^{-1}q^{-\lambda},k;q) = \frac{(v^{-1}q^{\rho+k+1};q^2)_x}{(-v)^nq^{n(\rho+k+n)}} \rphis{3}{2}{q^{-2n},q^{-2x},q^{2x+2\rho+2k}}{q^{2k},v^{-1}q^{\rho+k+1}}{q^2,q^2},
		\end{align*}
		where we used that
		\[
			(a;q^{-2})_n=(-a)^nq^{-n(n-1)}(a^{-1};q^2)_n.
		\]
		Therefore,
		\begin{align*}
			\lim\limits_{\lambda\to\infty} q^{2\lambda|\eta|} P^{v^{-1}q^{-\lambda}}_{\mathrm{AW}}(\eta,\xi) = v^{-2|\eta|}P_\mathrm{J}^v(\eta,\xi) \prod_{j=1}^{M}\frac{(v^{-1}q^{h_{j-1,0}^-(\eta)+h_{j+1}^+(\xi)+k_j+1};q^2)_{\xi_j}}{(v^{-1}q^{h_{j-1,0}^-(\eta)+h_{j+1}^+(\xi)+k_j+1};q^{2})_{\eta_j}}.
		\end{align*}
		At this point, it is not at all clear that 
		\[
		\prod_{j=1}^{M}\frac{(v^{-1}q^{h_{j-1,0}^-(\eta)+h_{j+1}^+(\xi)+k_j+1};q^2)_{\xi_j}}{(v^{-1}q^{h_{j-1,0}^-(\eta)+h_{j+1}^+(\xi)+k_j+1};q^{2})_{\eta_j}}
		\]
		is a factor depending on the total number of particles. However, since for $m\in\Z_{\geq 0}$ we have
		\[
			(a;q^2)_m=\frac{(a;q^2)_\infty}{(aq^{2m};q^2)_\infty},
		\]
		we obtain
		\begin{align*}
			\prod_{j=1}^{M}\frac{(v^{-1}q^{h_{j-1,0}^-(\eta)+h_{j+1}^+(\xi)+k_j+1};q^2)_{\xi_j}}{(v^{-1}q^{h_{j-1,0}^-(\eta)+h_{j+1}^+(\xi)+k_j+1};q^{2})_{\eta_j}} &= \prod_{j=1}^{M}\frac{(v^{-1}q^{h_{j-1,0}^-(\eta)+h_{j+1}^+(\xi)+k_j+1+2\eta_j};q^2)_{\infty}}{(v^{-1}q^{h_{j-1,0}^-(\eta)+h_{j+1}^+(\xi)+k_j+1+2\xi_j};q^{2})_{\infty}} \\
			&= \prod_{j=1}^{M}\frac{(v^{-1}q^{h_{j,0}^-(\eta)+h_{j+1}^+(\xi)+1};q^2)_{\infty}}{(v^{-1}q^{h_{j-1,0}^-(\eta)+h_{j}^+(\xi)+1};q^{2})_{\infty}} \\
			&= \frac{(v^{-1}q^{h_{M,0}^-(\eta)+h_{M+1}^+(\xi)+1};q^2)_{\infty}}{(v^{-1}q^{h_{0,0}^-(\eta)+h_{1}^+(\xi)+1};q^{2})_{\infty}} \\
			&= \frac{(v^{-1}q^{2|\eta|+|\vec{k}|+\rho+1};q^2)_{\infty}}{(v^{-1}q^{2|\xi|+|\vec{k}|+\rho+1};q^{2})_{\infty}}\\
			&= \frac{(v^{-1}q^{|\vec{k}|+\rho+1};q^2)_{|\xi|}}{(v^{-1}q^{|\vec{k}|+\rho+1};q^{2})_{|\eta|}},
		\end{align*}
		proving \eqref{eq:limorthjac2}.
	\end{proof}
	\begin{remark}
		We mentioned that one can go from the Big $q$-Jacobi duality function $P_\mathrm{J}^v$ to the Al-Salam--Chihara one $P_\mathsf{R}^{}$ by taking a limit in the free parameter $v$. If we would take this limit in above orthogonality relation as well, we end up with the orthogonality \eqref{eq:orthoP_R} from Corollary \ref{cor:orthASC}(ii).
	\end{remark}	
	Let us now turn to the second orthogonality of Table \ref{tab:limits_orth_duality_dynamic_ASIP}, the one between $\ASIP{q^{-1},\vec{k}}$ and $\RASIPpar$, which corresponds to the limit $\lambda\to-\infty$. Duality functions for this duality can be obtained from the duality function $P^v_\mathrm{J}$ between $\ASIPpar\leftrightarrow\RASIPpar$ by replacing $q$ by $q^{-1}$. That is, we obtain Big $q^{-1}$-Jacobi polynomials $P_\mathrm{J}^v(\eta,\xi;q^{-1})$. For orthogonality, we cannot simply replace $q$ by $q^{-1}$ in Proposition \ref{prop:degenerate orthogonality qjac} because of the infinite shifted factorials. However, we can take the limit $\lambda\to-\infty$ in the orthogonality relation \eqref{eq:aworthoxiq} between $\LASIPno$ and $\RASIPno$. Before we do so, we have to make the connection between the limit $\lambda\to-\infty$ of the Askey-Wilson duality function $P^v_\mathrm{AW}$ with $q^{-1}$ and the Big $q^{-1}$-Jacobi ones. 
	\begin{lemma} \label{lem:AWbigq-1jac}
		The following limits hold,
		\begin{align*}
			\lim\limits_{\lambda\to-\infty}P^{vq^{-\lambda}}_{\mathrm{AW}}(\eta,\xi;q^{-1}) &=  P_\mathrm{J}^{v}(\eta,\xi;q^{-1}),\\
			\lim\limits_{\lambda\to-\infty} q^{-2\lambda|\eta|} P^{v^{-1}q^{\lambda}}_{\mathrm{AW}}(\eta,\xi;q^{-1}) &=v^{-2|\eta|}\frac{(v^{-1}q^{-|\vec{k}|-\rho-1};q^{-2})_{|\xi|}}{(v^{-1}q^{-|\vec{k}|-\rho-1};q^{-2})_{|\eta|}}P_\mathrm{J}^v(\eta,\xi;q^{-1}).
		\end{align*}
	\end{lemma}
	\begin{proof}
		Both limits follow from the already known limits \eqref{eq:limorthjac1} and \eqref{eq:limorthjac2}. Indeed, we can replace $q^{-\lambda}$ by $\varepsilon$ in the limits of this lemma. The limit $\lambda\to-\infty$ then becomes the limit $\varepsilon \to 0$. Thus the limits in this lemma are equivalent to
		 \begin{align*}
		 	\lim\limits_{\varepsilon\to 0}P^{v\varepsilon}_{\mathrm{AW}}(\eta,\xi;q^{-1}) \qquad \text{and}\qquad \lim\limits_{\varepsilon \to 0} \varepsilon^{2|\eta|} P^{v^{-1}\varepsilon^{-1}}_{\mathrm{AW}}(\eta,\xi;q^{-1}),
		 \end{align*}
	 	where $q^{-\lambda}$ should be replaced by $\varepsilon$ everywhere in $P_{\mathrm{AW}}^{v\varepsilon}$ as well. Similarly, if we replace $q^\lambda$ by $\varepsilon$ in the limits \eqref{eq:limorthjac1} and \eqref{eq:limorthjac2} from the previous proposition, we obtain
	 	\begin{align*}
	 		\lim\limits_{\varepsilon\to 0}P^{v\varepsilon}_{\mathrm{AW}}(\eta,\xi;q) \qquad \text{and}\qquad \lim\limits_{\varepsilon \to 0} \varepsilon^{2|\eta|} P^{v^{-1}\varepsilon^{-1}}_{\mathrm{AW}}(\eta,\xi;q).
	 	\end{align*}
 		Therefore, the limits of this lemma follow from \eqref{eq:limorthjac1} and \eqref{eq:limorthjac2} by replacing $q$ by $q^{-1}$.
	\end{proof}
	Applying this lemma, we can prove orthogonal duality relations of $P^v_{\mathrm{J}}(\eta,\xi;q^{-1})$ with respect to the reversible measure $W_\mathsf{R}$. The proof is similar to the proof of the previous proposition.
	\begin{proposition}[Orthogonality Big $q^{-1}$-Jacobi]
		let $\rho>-1$, then we have the following orthogonality relation for the Big $q^{-1}$-Jacobi duality function $P^v_\mathrm{J}(\eta,\xi;q^{-1})$ between $\ASIP{q^{-1},\vec{k}}$ and $\RASIPpar$,
		\begin{align*}
			\sum_{\xi\in X_d} P_\mathrm{J}^v(\eta,\xi;q^{-1})P_\mathrm{J}^v(\eta',\xi;q^{-1})W_\mathsf{R}(\xi) \omega^{v}_\mathrm{J'}(|\eta|,|\xi|,\rho,q) = \frac{\delta_{\eta,\eta'}}{W(\eta,q^{-1})} \qquad \text{for all } \eta,\eta'\in X_d,
		\end{align*}
		where
		\[
			\omega^{v}_\mathrm{J'}(|\eta|,|\xi|,\rho) =  v^{-2|\eta|}q^{|\eta|(1-2|\eta|)}\frac{(v^{-1}q^{-|\vec{k}|-\rho-1};q^{-2})_{|\xi|}}{(v^{-1}q^{-|\vec{k}|-\rho-1};q^{-2})_{|\eta|}}\frac{(v^{-1}q^{2|\xi|+|\vec{k}|+\rho+1};q^2)_\infty}{(v^{-1}q^{-2|\eta|-|\vec{k}|+\rho+1};q^2)_\infty}.
		\]
	\end{proposition}
		\begin{proof}
			As said, we take the limit $\lambda\to-\infty$ in \eqref{eq:aworthoxiq} with $v$ replaced by $vq^{-\lambda}$ and $\zeta$ by $\eta$,
			\begin{align*}
					\sum_{\xi\in X_d} P^{vq^{-\lambda}}_{\mathrm{AW}}(\eta,\xi;q^{-1}) P^{v^{-1}q^\lambda}_{\mathrm{AW}}(\eta',\xi;q^{-1}) W_\mathsf{R}(\xi,q) \omega^{vq^{-\lambda}}_{\mathrm{AW}}(|\xi|,|\eta|,\rho,\lambda,q)&= \frac{\delta_{\eta,\eta'}}{W_\mathsf{L}(\eta,q)}.
			\end{align*}
			Justifying interchanging the sum and limit is done similarly as in Proposition \ref{prop:degenerate orthogonality qjac}. Per term, we have the following limits,
			\begin{align}
				\lim\limits_{\lambda\to-\infty}P^{vq^{-\lambda}}_{\mathrm{AW}}(\eta,\xi;q^{-1}) &=  P_\mathrm{J}^{v}(\eta,\xi;q^{-1}),\\
				\lim\limits_{\lambda\to-\infty} q^{-2\lambda|\eta|} P^{v^{-1}q^{\lambda}}_{\mathrm{AW}}(\eta,\xi;q^{-1}) &=v^{-2|\eta|}\frac{(v^{-1}q^{-|\vec{k}|-\rho-1};q^{-2})_{|\xi|}}{(v^{-1}q^{-|\vec{k}|-\rho-1};q^{-2})_{|\eta|}}P_\mathrm{J}^v(\eta,\xi;q^{-1}),\\
				\lim\limits_{\lambda\to-\infty} \omega^{vq^{-\lambda}}_{\mathrm{AW}}(|\xi|,|\eta|,\rho,\lambda,q) &= \frac{(v^{-1}q^{2|\xi|+|\vec{k}|+\rho+1};q^2)_\infty}{(v^{-1}q^{-2|\eta|-|\vec{k}|+\rho+1};q^2)_\infty},\\
				\lim\limits_{\lambda\to-\infty} q^{2\lambda|\eta|}W_\mathsf{L}(\eta) &= q^{|\eta|(1-2|\eta|)}W(\eta,q^{-1}).
			\end{align}
			Using these four limits, proving the proposition is straightforward. The first two are the content of Lemma \ref{lem:AWbigq-1jac}, the third and fourth are straightforward calculations similar to the ones in the proof of Proposition \ref{prop:degenerate orthogonality qjac}.
		\end{proof}
		\begin{remark}
		Similarly as discussed in the previous remark, one can go from the Big $q^{-1}$-Jacobi polynomials to the $q^{-1}$-Al-Salam--Chihara polynomials by taking a limit in the free parameter $v$. When taking this limit in the free parameter $v$, the orthogonality relation of this proposition is equivalent\footnote{One has to reverse the order of sites so that $\LASIPno\leftrightarrow\RASIPno$ and $\ASIPpar\leftrightarrow\ASIP{q^{-1},\vec{k}}$.} to the second one of Corollary \ref{cor:orthASC}(i), which corresponds to the dual orthogonality of the $q^{-1}$-Al-Salam--Chihara polynomials. However, the $q^{-1}$-Al-Salam--Chihara polynomials also have their usual orthogonality relation, namely the first one of Corollary \ref{cor:orthASC}(i). Thus, something a bit unexpected happens: the duality function $P_\mathrm{J}^v(\eta,\xi;q^{-1})$ has one orthogonality relation with respect to a reversible measure, but a limit of this function has two. From the perspective of orthogonal polynomials this has to do with the following. This proposition is the dual orthogonality of the Big $q^{-1}$-Jacobi polynomials. Their usual orthogonality consist of two infinite sums. One sum has support $\Z_{\geq0}$, but the other one has not. Therefore, it cannot be a reversible measure since part of the support lies outside the statespace of the process. However, this part disappears when taking the limit in the free parameter $v$, so that only the infinite sum with support $\Z_{\geq0}$ remains, which corresponds to the first orthogonality of Corollary \ref{cor:orthASC}(i).
		\end{remark}
		Let us now turn to the orthogonality of (iii) in Table \ref{tab:limits_orth_duality_dynamic_ASIP}: the self-duality of $\ASIPpar$. Duality functions are given by the $q$-Meixner duality functions $P_\mathrm{M}^v$. Orthogonality can be obtained from the orthogonality of the Big $q$-Jacobi duality functions $P_\mathrm{J}^v$ by replacing $v$ by $vq^{-\rho}$ and letting $\rho\to-\infty$, see Proposition \ref{prop:degenerate duality}(ii). If $k_1=k_2=...=k_M$ is the same for every site, this result was already proven in \cite{CFG}.
\begin{proposition}\* 
	We have the following orthogonality relation for the $q$-Meixner self-duality functions $P_\mathrm{M}^v$ for $\ASIPpar$,
		\begin{align*}
			\sum_{\eta\in X_d} P^v_{\mathrm{M}}(\eta,\xi) P^v_{\mathrm{M}}(\eta,\xi') W(\eta)\omega_{\mathrm{M}}^{v}(|\eta|,|\xi|) = \frac{\delta_{\xi,\xi'}}{W(\xi)} \qquad \text{for all } \xi,\xi'\in X_d,
		\end{align*}
		where
		\begin{align*}
			\omega_{\mathrm{M}}^{v}(|\eta|,|\xi|)&=(-v)^{|\eta|-|\xi|}\frac{(v^{-1}q^{2|\eta|+|\vec{k}|+1};q^2)_\infty}{(v^{-1}q^{-2|\xi|-|\vec{k}|+1};q^2)_\infty}q^{|\eta|(|\eta|-|\vec{k}|-1)+|\xi|(1-|\xi|-|\vec{k}|)}.
		\end{align*}
\end{proposition}
\begin{proof}
	We use proposition \ref{prop:degenerate orthogonality qjac} with $v$ replaced by $vq^{-\rho}$,
	\begin{align*}
		\sum_{\eta\in X_d} P^{vq^{-\lambda}}_{\mathrm{J}}(\eta,\xi) P^{vq^{-\lambda}}_{\mathrm{J}}(\eta,\xi') W(\eta)\omega_{\mathrm{J}}^{vq^{-\lambda}}(|\eta|,|\xi|,\rho) = \frac{\delta_{\xi,\xi'}}{W_\mathsf{R}(\xi)},
	\end{align*}
	and then we take the limit $\lambda\to-\infty$ on both sides. We have the limits,
	\begin{align*}
		\lim\limits_{\rho\to-\infty}P_\mathrm{J}^{vq^{-\rho}}(\eta,\xi)&= P_\mathrm{M}^v(\eta,\xi), \\
		\lim\limits_{\rho\to-\infty}q^{-2\rho|\xi|}\omega_\mathrm{J}^{vq^{-\rho}}&=(-v)^{|\eta|-|\xi|}\frac{(v^{-1}q^{2|\eta|+|\vec{k}|+1};q^2)_\infty}{(v^{-1}q^{-2|\xi|-|\vec{k}|+1};q^2)_\infty}q^{|\eta|(|\eta|-|\vec{k}|-1)+|\xi|(|\xi|+|\vec{k}|)},\\
		\lim\limits_{\rho\to -\infty} q^{2\rho|\xi|}W_\mathsf{R}(\xi)&= q^{|\xi|(1-2|\xi|-2|\vec{k}|)} W(\xi).
	\end{align*}
	The first is Proposition \ref{prop:degenerate duality}(ii), the second and third are similar computations as before. From these limits, the required identity easily follows.
\end{proof}
\begin{remark}
	In this section, we have taken limits in the duality between $\LASIPpar$ and $\RASIPpar$ to obtain (orthogonal) dualities for asymmetric discrete processes. One could do something similar for the symmetric case. That is, take limits in the duality between $\LSIP{\vec{k},\lambda}$ and $\RSIP{\vec{k},\rho}$ to find (orthogonal) dualities between (dynamic) SIP and SIP. Alternatively, these dualities could also be found by taking the limit $q\to 1$ in the results of this section. 
\end{remark}
\section{Dynamic ABEP}\label{sec:dynABEP}
	In \cite{CGRSsu11}, an asymmetric version of the Brownian energy process (ABEP) was introduced. This diffusion process was shown to arise as a limit of ASIP where simultaneously the number of particles goes to infinity and the asymmetry goes to zero. Moreover, it was proven there exists a deterministic transformation `$g$' which turns BEP into ABEP. This transformation can be used to transfer some properties of BEP towards ABEP, such as reversibility, duality and existence. In this section, we will do something similar for dynamic ABEP, a process that arises as a limit of dynamic ASIP. Again, there exists a deterministic transformation `$g$' which turns BEP into dynamic ABEP and we will exploit this to prove duality and reversibility of the process.\\
	\subsection{Definition dynamic ABEP}
	Before we introduce dynamic ABEP, we need the partial energy $E_j^-(\vec{x})$ of $\vec{x}=(x_1,...,x_M)$, where $x_j\geq 0$ represents the energy on site $j$. We define this to be the sum of energies left of and including site $j$, 
	\begin{align*}
		E_j^-(\vec{x})=\sum_{i=1}^j x_i.
	\end{align*}
	Similarly, we define
	\begin{align*}
	E_j^+(\vec{x})=\sum_{i=j}^M x_i.
	\end{align*}
	 Moreover, we introduce the following two functions to ease notation later on,
	\begin{align*}
		\sinh_\sigma(x)=\sinh(\sigma x)\qquad \text{and}\qquad \cosh_\sigma(x)=\cosh(\sigma x).
	\end{align*}
	Now we define ABEP$^{}_\mathsf{L}$, the left version of the dynamic asymmetric Brownian energy process, as follows.
		\begin{Definition}[Dynamic ABEP]\label{def:dynabep}
		Let $\sigma\in\R\backslash\{0\}$ be the asymmetry parameter, $\lambda >0$, and $\vec{k}=(k_1,k_2,...,k_M)$ with each $k_j > 0$. Then ABEP$_\mathsf{L}(\sigma,\vec{k},\lambda)$ is the diffusion process on $X_c=\R_{\geq 0}^N$ with generator acting on $f\in C^2(X_c)$ by
		\[
			\big[\genLABEP f\big](\vec{x})= \sum_{j=1}^{N-1}A_j(\vec{x},\sigma,\lambda) \bigg(\frac{\partial }{\partial x_j}-\frac{\partial }{\partial x_{j+1}}\bigg)^2f(\vec{x}) + B_j(\vec{x},\sigma,\vec{k},\lambda) \bigg(\frac{\partial }{\partial x_j}-\frac{\partial }{\partial x_{j+1}}\bigg) f(\vec{x}),
		\]
		where
		\begin{align*}
			A_j=& \frac{1}{\sigma^2}\sinh_\sigma(x_{j})\sinh_\sigma(x_{j+1})\frac{\sinh_\sigma( \lambda+2E_j^--x_j)\sinh_\sigma(\lambda+2E_j^-+x_{j+1})}{\sinh_\sigma(\lambda+2E_j^-)^2},\\
			B_j=&\frac{1}{\sigma}\bigg[k_j\sinh_\sigma(x_{j+1})\frac{\sinh_\sigma(\lambda+2E_j^-+ x_{j+1})}  {\sinh_\sigma(\lambda+2E_j^-)} -k_{j+1}\sinh_\sigma( x_{j})\frac{\sinh_\sigma(\lambda+2E_j^--x_j)}   {\sinh_\sigma(\lambda+2E_j^-)} \\
			&\hspace{0.3cm} -2\sinh_\sigma (x_j)\sinh_\sigma( x_{j+1})\frac{\cosh_\sigma(\lambda+2E_j^-)\sinh_\sigma(\lambda+2E_j^--x_j)\sinh_\sigma(\lambda+2E_j^-+x_{j+1})}{\sinh_\sigma(\lambda+2E_j^-)^3}\bigg].
		\end{align*}
		\end{Definition}
		\begin{remark}\*
			\begin{itemize} 
				\item Note that if $c$ is a factor depending on the total energy $|\vec{x}|$, then 
				\begin{align*}
					\big[\genLABEP cf\big](\vec{x})=c \big[\genLABEP f\big](\vec{x}),
				\end{align*}
				since by the chain rule,
				\[
					\Big(\frac{\partial}{\partial x_i} - \frac{\partial}{\partial x_j}\Big)c(|\vec{x}|)=0.
				\]
				Therefore, duality functions multiplied by a factor depending on $|\vec{x}|$ are still duality functions. Similarly for reversible measures. 
				\item Similar to the $q\to q^{-1}$ invariance of the dynamic ASIP rates, we have that the dynamic ABEP rates are invariant under sending $\sigma\to-\sigma$.
				\item One could also write the rates without the hyperbolic function, then we have
						\begin{align*}
					A_j&= \frac{1}{4\sigma^2}\big(1-e^{2\sigma x_{j}}\big)\big(e^{-2\sigma x_{j+1}}-1\big)\frac{\big(1-e^{2\sigma \{\lambda+2E_j^--x_j\}}\big)\big(1-e^{2\sigma \{\lambda+2E_j^-+x_{j+1}\}}\big)}{\big(1-e^{2\sigma \{\lambda+2E_j^-\}}\big)^2},\\
					B_j&=\frac{1}{2\sigma}\bigg[k_j\big(1-e^{-2\sigma x_{j+1}}\big)\bigg(\frac{1-e^{2\sigma\{\lambda+2E_j^-+ x_{j+1}\}}  }{1-e^{2\sigma\{\lambda+2E_j^-\}}}\bigg) +k_{j+1}\big(1-e^{2\sigma x_{j}}\big)\bigg(\frac{1-e^{2\sigma\{\lambda+2E_j^--x_j\}}   }{1-e^{2\sigma\{\lambda+2E_j^-\}}}\bigg) \\
					&+\big(1-e^{2\sigma x_j}\big)\big(e^{-2\sigma x_{j+1}}-1\big)\frac{\big(1+e^{2\sigma\{\lambda+2E_j^-\}}\big)\big(1-e^{2\sigma\{\lambda+2E_j^--x_j\}}\big)\big(1-e^{2\sigma\{\lambda+2E_j^-+x_{j+1}\}}\big)}{\big(1-e^{2\sigma\{\lambda+2E_j^-\}}\big)^3}\bigg].
				\end{align*}
				\item Similarly as for dynamic ASIP, one can also define a right version of dynamic ABEP by reversing the order of sites. One then obtains the rates
				\begin{align*}
					A_j&= \frac{1}{4\sigma^2}\sinh_\sigma(x_{j+1})\sinh_\sigma( x_{j})\frac{\sinh_\sigma(\rho+2E_{j+1}^+-x_{j+1})\sinh_\sigma(\rho+2E_{j+1}^++x_{j})}{\sinh_\sigma(\rho+2E_{j+1}^+)^2},\\
					B_j&=\frac{1}{2\sigma}\bigg[k_{j+1}\sinh_\sigma( x_{j})\frac{\sinh_\sigma(\rho+2E_{j+1}^++ x_{j})  }{\sinh_\sigma(\rho+2E_{j+1}^+)} +k_{j}\sinh_\sigma( x_{j+1})\frac{\sinh_\sigma(\rho+2E_{j+1}^+-x_{j+1})   }{\sinh_\sigma(\rho+2E_{j+1}^+)} \\
					&\hspace{0.3cm}+\sinh_\sigma(x_{j+1})\sinh_\sigma( x_{j})\frac{\cosh_\sigma(\rho+2E_{j+1}^+)\sinh_\sigma(\rho+2E_{j+1}^+-x_{j+1})\sinh_\sigma(\rho+2E_{j+1}^++x_{j})}{\sinh_\sigma(\rho+2E_{j+1}^+)^3}\bigg].
				\end{align*}
				\item One can also define dynamic ABEP for $\lambda \leq 0$. To ensure positivity of $A_j$ in that case, it is sufficient to require
				\[
					\lambda + 2|\vec{x}| < 0.
				\] 
				We will do this by limiting our statespace $X_c$, similarly as done for dynamic ASIP. That is, if $\lambda\leq 0$, we will define dynamic ABEP with the same rates, but on the statespace
				\[
					X_{c,\lambda}=\{\vec{x}\in X_c: \lambda + 2|\vec{x}| < 0 \}.
				\]
			\end{itemize}
		\end{remark}
		\subsection{Special cases Dynamic ABEP}
		Without loss of generality, we will take $\sigma >0$ in the remaining of this section.
		\begin{itemize}
			\item $[$ABEP, $\lambda\to-\infty]$. When taking the limit $\lambda\to-\infty$, one obtains ABEP$(-\sigma,\vec{k})$ and when $\lambda\to\infty$ one gets ABEP$(\sigma,\vec{k})$. This process is a slight generalization of ABEP introduced in \cite{CGRSsu11}, where the parameter $k$ may vary per site. Its generator is given by 
			\[
			\hspace{1cm}\big[L^{\mathrm{ABEP}}_{\sigma,\vec{k}}f\big](\vec{x})= \sum_{j=1}^{N-1}C_j(\vec{x},\sigma) \bigg(\frac{\partial }{\partial x_j}-\frac{\partial }{\partial x_{j+1}}\bigg)^2f(\vec{x}) + D_j(\vec{x},\sigma,\vec{k}) \bigg(\frac{\partial }{\partial x_j}-\frac{\partial }{\partial x_{j+1}}\bigg) f(\vec{x}),
			\]
			with
			\begin{align*}
				C_j(\vec{x},\sigma)=& \frac{1}{4\sigma^2}\big(1-e^{-2\sigma x_{j}}\big)\big(e^{2\sigma x_{j+1}}-1\big),\\
				\hspace{1cm} D_j(\vec{x},\sigma,\vec{k})=&\frac{1}{2\sigma}\big(k_j\big(e^{2\sigma x_{j+1}}-1\big)+k_{j+1}\big(e^{-2\sigma x_{j}}-1\big) +\big(e^{-2\sigma x_j}-1\big)\big(e^{2\sigma x_{j+1}}-1\big)\big).
			\end{align*}
			\item $[$Dynamic BEP, $\sigma\to 0]$. By taking the limit $\sigma\to0$ from dynamic ABEP, we can obtain a dynamic version of BEP, which we well call BEP$_\mathsf{L}(\vec{k},\lambda)$. This is the diffusion process on $X_c$ with generator given by
				\[
				\hspace{1cm}\big[L^{\mathrm{BEP}_\mathsf{L}}_{\vec{k},\lambda}f\big](\vec{x})= \sum_{j=1}^{N-1}\hat{A}_j(\vec{x},\lambda) \bigg(\frac{\partial }{\partial x_j}-\frac{\partial }{\partial x_{j+1}}\bigg)^2f(\vec{x}) + \hat{B}_j(\vec{x},\vec{k},\lambda) \bigg(\frac{\partial }{\partial x_j}-\frac{\partial }{\partial x_{j+1}}\bigg) f(\vec{x}),
				\]
				where
				\begin{align*}
					\hat{A}_j(\vec{x},\lambda)=& x_jx_{j+1}\frac{(\lambda+2E_j^--x_j)(\lambda+2E_j^-+x_{j+1})}{(\lambda+2E_j^-)^2},\\
					\hat{B}_j(\vec{x},\vec{k},\lambda)=&k_jx_{j+1}\bigg(\frac{\lambda+2E_j^-+x_{j+1}}{\lambda+2E_j^-}\bigg)-k_{j+1}x_j\bigg(\frac{\lambda+2E_j^--x_j}{\lambda+2E_j^-}\bigg)\\
					&+ 2x_jx_{j+1}\frac{(\lambda+2E_j^--x_j)(\lambda+2E_j^-+x_{j+1})}{(\lambda+2E_j^-)^3}.
				\end{align*}
				\item $[$BEP, $\sigma\to0$ and $\lambda\to\pm\infty]$. Lastly, when taking the limit $\lambda\to\pm\infty$ of dynamic BEP or when taking the limit $\sigma\to 0$ of ABEP, one obtains BEP$(\vec{k})$, which has generator
				\[
				\hspace{1.3cm}\big[L^{\mathrm{BEP}}_{\vec{k}}f\big](\vec{x})= \sum_{j=1}^{N-1}x_jx_{j+1} \bigg(\frac{\partial }{\partial x_j}-\frac{\partial }{\partial x_{j+1}}\bigg)^2f(\vec{x}) + (k_jx_{j+1}-k_{j+1}x_{j}) \bigg(\frac{\partial }{\partial x_j}-\frac{\partial }{\partial x_{j+1}}\bigg) f(\vec{x}).
				\]
				For later reference, we mention that BEP$(\vec{k})$ has a reversible measure given by a product measure of gamma distributions (see \cite{Gr2019}),
				\begin{align}
					\mu^{\mathrm{BEP}}(\vec{x},\vec{k})=\prod_{j=1}^{M}\frac{x^{k_j-1}}{\Gamma(k_j)}e^{-x_j}.\label{eq:BEPrev}
				\end{align}
		\end{itemize} 
		\subsection{Dynamic ABEP as a diffusion limit of dynamic ASIP.}
		Dynamic ABEP arises as a weak asymmetric limit of dynamic ASIP, where the number of particles go to infinity,	similarly as done in \cite{CGRSsu11} when going from ASIP to ABEP. Take $\vec{x}\in X_c$ and for $\varepsilon>0$, let $q=(1-\varepsilon\sigma)$ and replace $\lambda$ by $\lambda/\varepsilon$. Take $\zeta^\varepsilon=\lfloor \vec{x}/\varepsilon\rfloor \in X_d$, where $\lfloor \vec{y}\rfloor$ is $\vec{y}$ where each component is rounded down to an integer. Then $\vec{x}^\varepsilon=\varepsilon\zeta^\varepsilon$ converges to $\vec{x}$ as $\varepsilon$ goes to $0$. Moreover, the the rates of the interacting particle system dynamic ASIP converge to the rates of the diffusion process dynamic ABEP, similarly as in \cite{CGRSsu11}. That is, the difference operator $\genLASIP$  becomes the differential operator $\genLABEP$ in the following sense:
		\begin{align*}
			\lim_{\varepsilon\to 0}\big[L^{\mathrm{ASIP}_\mathsf{L}}_{1-\varepsilon\sigma,\vec{k},\lambda/\varepsilon} f_\varepsilon\big](\lfloor \vec{x}/\varepsilon\rfloor) = \big[\genLABEP f](\vec{x}),
		\end{align*}
		where $f_\varepsilon,f\colon X_c\to \R$ are sufficiently smooth functions such that 
		\[
			\lim\limits_{\varepsilon\to 0}f_\varepsilon(\vec{x}/\varepsilon) = f(\vec{x}).
		\]
		The next theorem makes this claim rigorous.
		\begin{theorem}\label{thm:ratesdynasipdynabep}
			Let $f\in C^3(X_c)$ and suppose that there exists $\varepsilon'>0$ such that for each $0<\varepsilon<\varepsilon'$ we have a function $f_\varepsilon\in C^3(X_c)$ where each partial derivative of $f_\varepsilon(\cdot/\varepsilon)$ of order 3 is bounded in an $\varepsilon'$-neighbourhood of $\vec{x}$, uniformly\footnote{A function $g_\varepsilon$ is bounded on $X$ uniformly in $\varepsilon$ if $\sup\limits_{\vec{x}\in X}|g_\varepsilon(\vec{x})| < M$, where $M$ is independent of $\varepsilon$.} in $\varepsilon$. If for each $\vec{x}\in X_c$
			\begin{align}
				\lim\limits_{\varepsilon\to 0}f_\varepsilon(\vec{x}/\varepsilon) &= f(\vec{x}) \label{eq:dynasipabepf}
			\end{align}
			and
			\begin{align}
				\lim\limits_{\varepsilon\to 0}\left[\bigg(\frac{\partial}{\partial x_j}-\frac{\partial}{\partial x_{j+1}}\bigg)f_\varepsilon(\cdot/\varepsilon)\right](x) &= \left[\bigg(\frac{\partial}{\partial x_j}-\frac{\partial}{\partial x_{j+1}}\bigg)f\right](\vec{x}),\label{eq:dynasipabeppartialf}\\
				\lim\limits_{\varepsilon\to 0}\left[\bigg(\frac{\partial}{\partial x_j}-\frac{\partial}{\partial x_{j+1}}\bigg)^2f_\varepsilon(\cdot/\varepsilon)\right](\vec{x}) &= \left[\bigg(\frac{\partial}{\partial x_j}-\frac{\partial}{\partial x_{j+1}}\bigg)^2f\right](\vec{x})\label{eq:dynasipabeppartial2f}
			\end{align}
			for each $j=1,..,M-1$, then the difference operator $L_{1-\varepsilon\sigma,\vec{k},\lambda/\varepsilon}^{\mathrm{ASIP}_\mathsf{L}}$ applied to $f_\varepsilon(\cdot/\varepsilon)$ converges to differential operator $\genLABEP$ applied to $f$,
			\begin{align}
			\lim_{\varepsilon\to 0}\big[L^{\mathrm{ASIP}_\mathsf{L}}_{1-\varepsilon\sigma,\vec{k},\lambda/\varepsilon} f_\varepsilon\big](\lfloor \vec{x}/\varepsilon\rfloor) = \big[\genLABEP f](\vec{x}).\label{eq:limlasiplabep}
		\end{align}
		
		\end{theorem}		
		\begin{proof}
			We can use the Taylor expansion of $f_\varepsilon$ in the coordinates $j$ and $j+1$ to write
			\begin{align*}
				f_\varepsilon (\vec{y}^{j,j+1}) = f_\varepsilon(\vec{y}) - \bigg(\frac{\partial }{\partial y_{j}}-\frac{\partial }{\partial y_{j+1}} \bigg) f_\varepsilon (\vec{y}) + \half \bigg(\frac{\partial }{\partial y_{j+1}}-\frac{\partial }{\partial y_{j}} \bigg)^2  f_\varepsilon (\vec{y}) +R,
			\end{align*}
			where $R$ is the Lagrange remainder term. If we replace $\vec{y}$ by $\vec{x}/\varepsilon$, the chain rule gives us
		\begin{align*}
			\begin{split}f_\varepsilon ((\vec{x}/\varepsilon)^{j,j+1}) =& f_\varepsilon (\vec{x}/\varepsilon) - \varepsilon\left[ \bigg(\frac{\partial }{\partial x_{j}}-\frac{\partial }{\partial x_{j+1}} \bigg) f_\varepsilon (\cdot/\varepsilon)\right](\vec{x}) \\
				&\qquad + \half\varepsilon^2 \left[\bigg(\frac{\partial }{\partial x_{j+1}}-\frac{\partial }{\partial x_{j}} \bigg)^2  f_\varepsilon (\cdot/\varepsilon)\right](\vec{x}) + \mathcal{O} (\varepsilon^3),\end{split}
		\end{align*}
		where we used that each partial derivative of $f_\varepsilon$ of order 3 is bounded uniformly in $\varepsilon$. 
		Therefore, we obtain
		\begin{align}
				\begin{split}f_\varepsilon ((\vec{x}/\varepsilon)^{j,j+1})- f_\varepsilon (\vec{x}/\varepsilon) =&  - \varepsilon\left[ \bigg(\frac{\partial }{\partial x_{j}}-\frac{\partial }{\partial x_{j+1}} \bigg) f_\varepsilon (\cdot/\varepsilon)\right](\vec{x}) \\
					&\qquad + \half\varepsilon^2 \left[\bigg(\frac{\partial }{\partial x_{j+1}}-\frac{\partial }{\partial x_{j}} \bigg)^2  f_\varepsilon (\cdot/\varepsilon)\right](\vec{x}) + \mathcal{O} (\varepsilon^3).\end{split} \label{eq:taylor1}
		\end{align}
		Similarly, we get
		\begin{align}
			\begin{split}f_\varepsilon ((\vec{x}/\varepsilon)^{j+1,j})- f_\varepsilon (\vec{x}/\varepsilon) =&  \varepsilon\left[ \bigg(\frac{\partial }{\partial x_{j}}-\frac{\partial }{\partial x_{j+1}} \bigg) f_\varepsilon (\cdot/\varepsilon)\right](\vec{x}) \\
				&\qquad + \half\varepsilon^2 \left[\bigg(\frac{\partial }{\partial x_{j+1}}-\frac{\partial }{\partial x_{j}} \bigg)^2  f_\varepsilon (\cdot/\varepsilon)\right](\vec{x}) + \mathcal{O} (\varepsilon^3).\end{split} \label{eq:taylor2}
		\end{align}
		Moreover, we claim that for the rates of $\LASIP{1-\varepsilon\sigma,\vec{k},\lambda/\varepsilon}$ we have
		\begin{align}
			C^{\mathsf{L},+}_j(\vec{x}/\varepsilon) &= A_j(\vec{x},\sigma,\lambda)/\varepsilon^2 + \mathcal{O}(1/\varepsilon),\label{eq:ratesasipabep1}\\
			C^{\mathsf{L},-}_{j+1}(\vec{x}/\varepsilon)&= A_j(\vec{x},\sigma,\lambda)/\varepsilon^2 + \mathcal{O}(1/\varepsilon),\label{eq:ratesasipabep2}\\
			 C^{\mathsf{L},-}_{j+1}(\vec{x}/\varepsilon)-C^{\mathsf{L},+}_j(\vec{x}/\varepsilon)&=  B_j(\vec{x},\sigma,\vec{k},\lambda)/\varepsilon + \mathcal{O}(1),\label{eq:ratesasipabep3}
		\end{align}
		where $A$ and $B$ are the factors in the rates of dynamic ABEP from Definition \ref{def:dynabep}. Indeed, define $\zeta^\varepsilon=\lfloor \vec{x}/\varepsilon\rfloor$ and $\vec{x}^\varepsilon=\varepsilon\zeta^\varepsilon$, then from \eqref{eq:taylor1}-\eqref{eq:ratesasipabep3} we get
		\begin{align*}
				\big[L^{\mathrm{ASIP}_\mathsf{L}}_{1-\varepsilon\sigma,\vec{k},\lambda/\varepsilon} f_\varepsilon\big](\zeta^\varepsilon)=\sum_{j=1}^{M-1} &\varepsilon \left(C^{\mathsf{L},-}_{j+1}(\zeta^\varepsilon)-C^{\mathsf{L},+}_j(\zeta^\varepsilon)\right) \left[\bigg(\frac{\partial }{\partial x_{j+1}}-\frac{\partial }{\partial x_{j}} \bigg) f_\varepsilon(\cdot/\varepsilon)\right](\vec{x}^\varepsilon) \\
			+\half &\varepsilon^2\left(C^{\mathsf{L},+}_j(\zeta^\varepsilon)+C^{\mathsf{L},-}_{j+1}(\zeta^\varepsilon)\right)\left[\bigg(\frac{\partial }{\partial x_{j+1}}-\frac{\partial }{\partial x_{j}} \bigg)^2  f_\varepsilon(\cdot/\varepsilon)\right](\vec{x}^\varepsilon) +  \mathcal{O}(\varepsilon) \\
				= \sum_{j=1}^{M-1} &B_j(\vec{x}^\varepsilon,\sigma,\vec{k},\lambda) \left[\bigg(\frac{\partial }{\partial x_{j+1}}-\frac{\partial }{\partial x_{j}} \bigg) f_\varepsilon(\cdot/\varepsilon)\right](\vec{x}^\varepsilon) \\
				+&A_j(\vec{x}^\varepsilon,\sigma,\lambda)\left[\bigg(\frac{\partial }{\partial x_{j+1}}-\frac{\partial }{\partial x_{j}} \bigg)^2  f_\varepsilon(\cdot/\varepsilon)\right](\vec{x}^\varepsilon) +  \mathcal{O}(\varepsilon).
		\end{align*}
		From this and the assumptions \eqref{eq:dynasipabepf}-\eqref{eq:dynasipabeppartial2f} the desired limit \eqref{eq:limlasiplabep} follows. Proving \eqref{eq:ratesasipabep1}, \eqref{eq:ratesasipabep2}, and \eqref{eq:ratesasipabep3} comes down to calculating the Taylor series of $C_j^{\mathsf{L},+}(\vec{x}/\varepsilon)$ and $C_j^{\mathsf{R},-}(\vec{x}/\varepsilon)$. For readability, the proof can be found in Appendix \ref{app:ratesdynasipdynabep}.
		\end{proof}
		\begin{remark}
			If one has a function $f\in C^3(X_c)$, one can find a candidate $f_\varepsilon$ by defining $f_\varepsilon(\vec{x})=f(\lfloor \vec{x}/\varepsilon\rfloor)$. However, more useful is the other way around. That is, if we have functions $f_\varepsilon$ for which the pointwise limit $$\lim\limits_{\varepsilon\to0}f_\varepsilon(\vec{x}/\varepsilon)=f(\vec{x})$$ exists, does \eqref{eq:limlasiplabep} hold? Using the Theorem above, it is sufficient to show uniform boundedness in $\varepsilon$ of all partial derivatives of order 3 and showing \eqref{eq:dynasipabeppartialf} and \eqref{eq:dynasipabeppartial2f}. This basically comes down to showing that taking the limit $\varepsilon\to 0$ and the derivative can be interchanged. In the next section, we will use this strategy to show duality between dynamic ABEP and SIP from the duality between dynamic ASIP and ASIP. 
		\end{remark}
		\subsection{Duality dynamic ABEP with SIP}
		Recall from Corollary \ref{cor:dualityLASIPASIP} that $\LASIPno$ and $\ASIPno$ are in duality with $P^{}_\mathsf{L}$ as duality functions,
		\[
			\Big[\genLASIP P^{}_\mathsf{L}(\eta,\cdot)\Big](\zeta)=\Big[\genASIP P^{}_\mathsf{L}(\cdot,\zeta)\Big](\eta).
		\]
		We want to use Theorem \ref{thm:ratesdynasipdynabep} to show that this results in a duality between $\LABEPno$ and $\SIPno$. That is, we want to replace $q$, $\lambda$ and $\zeta$ by $1-\varepsilon\sigma$, $\lambda/\varepsilon$ and $\vec{x}/\varepsilon$ respectively and then take the limit $\varepsilon\to 0$. This comes down to two steps. First, we have to find the limit 
		\[
			\lim\limits_{\varepsilon\to 0} P_\mathsf{L}^{}(\vec{x}/\varepsilon,\eta;\lambda/\varepsilon,\vec{k};1-\varepsilon\sigma).
		\]
		Secondly, we have to find the limit of the generators on both sides. The right-hand side is easy since this is just the limit $q\to 1$,
		\[
			\lim\limits_{\varepsilon\to 0} \Big[L^{\mathrm{ASIP}}_{1-\varepsilon\sigma,\vec{k}} f\Big](\eta) = \Big[L^\mathrm{SIP}_{\vec{k}} f\Big](\eta).
		\]
		For the left-hand side, we want to use Theorem \ref{thm:ratesdynasipdynabep}. For the first step, we have the following result.
		\begin{proposition} \label{prop:PLtoD}
			Let $\vec{x}\in X_c$ and $\eta\in X_d$, then we have
			\begin{align}
				\lim\limits_{\varepsilon\to 0} \left(\frac{1-(1-\varepsilon\sigma)^{-2}}{2\sigma}\right)^{|\eta|}P^{}_\mathsf{L}(\vec{x}/\varepsilon,\eta;\lambda/\varepsilon,\vec{k};1-\varepsilon\sigma) = D'(\vec{x},\eta),\label{eq:limitPLtoD}
			\end{align}
			where $D'\colon X_c \times X_d \to \R$ is given by
			\[
			D'(\vec{x},\eta) = \prod_{j=1}^{L} \frac{\Gamma(k)}{\Gamma(k+\eta_j)}\bigg(\frac{\cosh_\sigma(\lambda+2E_j^-(\vec{x}))-\cosh_\sigma(\lambda + 2 E^-_{j-1}(\vec{x}))}{\sigma}\bigg)^{\eta_j}.
			\] 
		\end{proposition}
		\begin{proof}
			The factor
			\[
				\left(\frac{1-(1-\varepsilon\sigma)^{-2}}{2\sigma}\right)^{|\eta|}
			\]
			is present to make the limit convergent and have a similar form as the known duality function between ABEP and SIP \cite{CGRSsu11}. The 1-site duality function of $P^{}_\mathsf{L}$ is given by
			\[
				\begin{split}&\left(\frac{1-(1-\varepsilon\sigma)^{-2}}{2\sigma}\right)^n(1-\varepsilon\sigma)^{\frac{\lambda}{\varepsilon}+\half k +\half}\\
					&\hspace{3cm} \times  \sum_{j=0}^n \frac{((1-\varepsilon\sigma)^{2n},(1-\varepsilon\sigma)^{\frac{2x}{\varepsilon}},(1-\varepsilon\sigma)^{\frac{-2(\lambda+x)}{\varepsilon}-2k};(1-\varepsilon\sigma)^{-2})_j}{((1-\varepsilon\sigma)^{-2k},(1-\varepsilon\sigma)^{-2};(1-\varepsilon\sigma)^{-2})_j}(1-\varepsilon\sigma)^{-2j}.\end{split}
			\]
			Since
			\[
				\lim\limits_{\varepsilon\to0}(1-\varepsilon\sigma)^{a+b/\varepsilon} = e^{-b\sigma}
			\]
			and
			\[
				\lim\limits_{\varepsilon\to 0} \frac{1-(1-\varepsilon\sigma)^a}{1-(1-\varepsilon\sigma)^b}=\frac{a}{b},
			\]	
			the limit \eqref{eq:limitPLtoD} follows from a straightforward computation, using that
			\[
				(k)_n = \frac{\Gamma(k+n)}{\Gamma(k)}. \qedhere 
			\]
		\end{proof}
		\begin{remark}\*
			Multiplying $D'$ by $(-\exp(\sigma \lambda))^{|\eta|}$, letting $\lambda\to-\infty$ and reversing the order of sites, we recover a similar\footnote{Taking $k_j=k$ for all sites $j$ gives the same duality function.} duality function between ABEP and SIP as the one in from \cite[(5.6)]{CGRSsu11}, 
			\begin{align*}
				D^{\mathrm{AB}}_{\hspace{0.04cm}\mathrm{S}}(\vec{x},\eta) = \prod_{j=1}^{L} \frac{\Gamma(k)}{\Gamma(k+\eta_j)}\bigg(\frac{e^{-2\sigma E^+_{j-1}(\vec{x})}-e^{-2\sigma E^+_{j}(\vec{x})} }{2\sigma}\bigg)^{\eta_j}.
			\end{align*}
			Moreover, letting $\sigma\to0$ in the $D^{\mathrm{AB}}_{\hspace{0.04cm}\mathrm{S}}$ gives a similar duality function between BEP and SIP as in \cite{CGGRbep},
			\begin{align}
				D^{\mathrm{B}}_{\hspace{0.04cm}\mathrm{S}}(\vec{x},\eta) = \prod_{j=1}^{L} \frac{\Gamma(k_j)}{\Gamma(k_j+\eta_j)}x_j^{\eta_j}.\label{eq:dualityBEPSIP}
			\end{align}
		\end{remark}
		Above remark shows that $D'$ would fit in naturally in the hierarchy of duality functions. To make the claim that $D'$ is a duality function rigorous, we have to show that $P^{}_L(\vec{x}/\varepsilon,\eta)$ satisfies the assumptions from Theorem \ref{thm:ratesdynasipdynabep}, which will be done in the next theorem.
		\begin{theorem}\label{thm:dualityDynABPSIP}
			The function $D'(\vec{x},\eta)$ is a duality function between $\LABEPpar$ and SIP$(\vec{k})$,
			\[
			\big[\genLABEP D'(\cdot,\eta)\big](\vec{x}) = \big[L^\mathrm{SIP}_{\vec{k}} D'(\vec{x},\cdot)\big](\eta).
			\]
		\end{theorem}
		\begin{proof}
			Let
			\[
				f_\varepsilon(\vec{x})=\left(\frac{1-(1-\varepsilon\sigma)^{-2}}{2\sigma}\right)^{|\eta|}P_\mathsf{L}^{}(\vec{x}/\varepsilon,\eta;\lambda/\varepsilon,\vec{k};1-\varepsilon\sigma)
			\]
			and
			\[
				f(\vec{x})=D'(\vec{x},\eta).
			\]
			Then we have to show that $f_\varepsilon$ and $f$ satisfy the assumptions of Theorem \ref{thm:ratesdynasipdynabep}. Note that $f\in C^3(X_c)$. Now let $\varepsilon'>0$ be small enough such that $|\varepsilon'\sigma|<1$. Then $f_\varepsilon\in C^3(X_c)$ for $\varepsilon\in (0,\varepsilon')$. From Proposition \ref{prop:PLtoD} we obtain
			\[
				\lim\limits_{\varepsilon\to 0 } f_\varepsilon(\vec{x}/\varepsilon) = f(\vec{x}).
			\]
			Let us now show that taking a partial derivative of order $\leq 3$ and taking the limit $\varepsilon\to 0$ of $f_\varepsilon(\vec{x})$ can be interchanged. If that would be true, then the limit $\varepsilon\to 0$ of each partial derivative of order $\leq 3$ of $f_\varepsilon$ exists and
			\begin{align*}
				\lim\limits_{\varepsilon\to 0}\left[\bigg(\frac{\partial}{\partial x_j}-\frac{\partial}{\partial x_{j+1}}\bigg)f_\varepsilon(\cdot/\varepsilon)\right](\vec{x}) &= \left[\bigg(\frac{\partial}{\partial x_j}-\frac{\partial}{\partial x_{j+1}}\bigg)f\right](\vec{x}),\\
				\lim\limits_{\varepsilon\to 0}\left[\bigg(\frac{\partial}{\partial x_j}-\frac{\partial}{\partial x_{j+1}}\bigg)^2f_\varepsilon(\cdot/\varepsilon)\right](\vec{x}) &= \left[\bigg(\frac{\partial}{\partial x_j}-\frac{\partial}{\partial x_{j+1}}\bigg)^2f\right](\vec{x}).
			\end{align*}
			To show that we can interchange the limit of $\varepsilon\to0$ and taking a partial derivative, note that $f_\varepsilon$ is a finite sum of terms of the form
			\[
				a(\varepsilon)\prod_{j=1}^M (1-\varepsilon\sigma)^{b_j(\varepsilon) x_j/\varepsilon},
			\]
			where $a(\varepsilon)$ and $b_j(\varepsilon)$ are independent of all $x_i$ and whose limit $\varepsilon\to 0$ exist. A straightforward calculation shows that
			\begin{align*}
				\lim\limits_{\varepsilon\to0}\frac{\partial}{\partial x_j} (1-\varepsilon\sigma)^{b_j(\varepsilon) x_j/\varepsilon} &=  \lim\limits_{\varepsilon\to0} b_j(\varepsilon)(1-\varepsilon\sigma)^{b_j(\varepsilon)x_j/\varepsilon}\frac{\ln(1-\varepsilon\sigma)}{\varepsilon}\\
				&= -\sigma b_j(0)e^{-\sigma b_j(0)x_j}
			\end{align*}
			and
			\begin{align*}
				\frac{\partial}{\partial x_j}\lim\limits_{\varepsilon\to0} (1-\varepsilon\sigma)^{b_j(\varepsilon) x_j/\varepsilon} &= \frac{\partial}{\partial x_j} e^{-\sigma b_j(0)x_j}\\
				&= -\sigma b_j(0)e^{-\sigma b_j(0)x_j}.
			\end{align*}
			Therefore, the partial derivatives of $f_\varepsilon$ are also sums of terms of the form
			\[
				a(\varepsilon)\prod_{j=1}^M (1-\varepsilon\sigma)^{b_j(\varepsilon) x_j/\varepsilon}
			\]
			and taking the limit of $\varepsilon\to 0$ and a partial derivative can be interchanged.\\ \indent Lastly, we have to show that each partial derivative of $f_\varepsilon(\cdot/\varepsilon)$ is bounded in an $\varepsilon'$-neighborhood of $\vec{x}$, uniformly in $\varepsilon$. We have the estimate
			\[
				 (1-\varepsilon\sigma)^{b_j(\varepsilon) x_j/\varepsilon} \leq \max\{(1-\varepsilon|\sigma|)^{-|b_j(\varepsilon)|M_j/\varepsilon}, (1+\varepsilon|\sigma|)^{|b_j(\varepsilon)|M_j/\varepsilon}\}
			\]
			where $M_j=x_j+\varepsilon'$. Since
			\begin{align*}
				\lim\limits_{\varepsilon\to 0}(1-\varepsilon|\sigma|)^{-|b_j(\varepsilon)|M_j/\varepsilon} &= e^{|\sigma||b_j(0)|M_j},\\
				\lim\limits_{\varepsilon\to 0}(1+\varepsilon|\sigma|)^{|b_j(\varepsilon)|M_j/\varepsilon} &= e^{|\sigma||b_j(0)|M_j},
			\end{align*}
			each of the terms 
			\[
				a(\varepsilon)\prod_{j=1}^M (1-\varepsilon\sigma)^{b_j(\varepsilon) x_j/\varepsilon}
			\]
			is bounded in an $\varepsilon'$-neighborhood of $x$, uniformly in $\varepsilon$.
		\end{proof}
		\subsection{Transforming dynamic ABEP to BEP}
		Notice the similarity between the duality function $D'$ between $\LABEPno$ and SIP and the duality function $\DBS$ between BEP and SIP. If we define $g\colon X_c \to X_c$ by $g(\vec{x})=(g_j(\vec{x}))_{j=1}^N$, where each component is given by
		\begin{align*}
			g_j(\vec{x})= \frac{\cosh_\sigma(\lambda+2E_{j}^-(\vec{x}))-\cosh_\sigma(\lambda + 2E_{j-1}^-(\vec{x}))}{\sigma},
		\end{align*}
		we have the following relations between those two duality functions,
		\begin{align}
			\DBS(g(\vec{x}),\eta) = D'(\vec{x},\eta).\label{eq:LABEPBEPduality}
		\end{align}
		We can exploit this relation to show that $g$ is a non-local transformation which connects BEP and $\LABEPno$.
		\begin{theorem}\* 
			The function $g$ is a non-local transformation between $\BEP{\vec{k}}$ and $\LABEPpar$ in the following sense, for $f\in C^3(X_c)$ we have
			\begin{align}
				\big[L^\mathrm{BEP}_{\vec{k}}f\big](g(\vec{x}))= \big[\genLABEP(f\circ g)\big](\vec{x}).\label{eq:LABEPBEPtransformation}
			\end{align}
			Consequently, if $\mathcal{X}(t)$ is an instance of $\LABEPpar$ with $\mathcal{X}(0)=\vec{x}$. Then the process $\mathcal{Z}(t)$ defined by $\mathcal{Z}(t)=g(\mathcal{X}(t))$ is an instance of $\BEP{\vec{k}}$ with initial position $\mathcal{Z}(0)=g(\vec{x})$. 
		\end{theorem}
		\begin{proof}
		We will prove \eqref{eq:LABEPBEPtransformation} using the duality between SIP and $\LABEPno$ given in \eqref{eq:LABEPBEPduality}. Using the latter, we obtain
		\begin{align*}
			[\genLABEP \DBS(g(\cdot),\eta)](\vec{x})&= [\genLABEP D'(\cdot,\eta)](\vec{x}) \\
			&= [L^\mathrm{SIP}_{\vec{k}}D'(\vec{x},\cdot)](\eta)\\
			&=  [L^\mathrm{SIP}_{\vec{k}}\DBS(g(\vec{x}),\cdot)](\eta) \\
			&= [L^\mathrm{BEP}_{\vec{k}}\DBS(\cdot,\eta)](g(\vec{x})).
		\end{align*}
		Therefore, we obtain \eqref{eq:LABEPBEPtransformation} for all $f$ in the linear span of $\{\DBS(\cdot,\eta)\}_{\eta\in X_d}$, which are all polynomials since $\DBS(\vec{x},\eta)$ is just a constant times the monomial
		\[
			\prod_{j=1}^M x_j^{\eta_j}.
		\]
		Now let $f\in C^3(X_c)$ and $\vec{x}\in X_c$ be arbitrary. Let $T_2$ be the second order Taylor polynomial of $f$ around the point $g(\vec{x})$. Then $T_2$ and $f$ are equal in $g(\vec{x})$, as well as all their first and second order partial derivatives. Hence,
		\[
			\big[L^\mathrm{BEP}_{\vec{k}}(f-T_2)\big](g(\vec{x})) = 0 = \big[\genLABEP((f-T_2)\circ g)\big](\vec{x}).
		\] 
		Consequently, by linearity of the generators,
		\begin{align*}
			\big[L^\mathrm{BEP}_{\vec{k}}f\big](g(\vec{x}))&=\big[L^\mathrm{BEP}_{\vec{k}}(f-T_2\big)](g(\vec{x})) + \big[L^\mathrm{BEP}_{\vec{k}}T_2\big](g(\vec{x})) \\
			&= \big[\genLABEP((f-T_2)\circ g)\big](\vec{x}) +  \big[\genLABEP(T_2\circ g)\big](\vec{x}) \\
			&= \big[\genLABEP(f\circ g)\big](\vec{x}). \qedhere 
		\end{align*}
		\end{proof}
		\begin{remark}\*
			\begin{itemize}
				\item As said before, this transformation $g$ can be used to transfer many properties from BEP to $\LABEPno$ including dualities and reversibility. We will exploit this in the next sections.
				\item Since multiplying $D'$ by $\alpha^{|\eta|}$ also gives a duality function, multiplying each $g_j$ by $\alpha$ also gives a transformation between $\BEP{\vec{k}}$ and $\LABEPpar$. 
				\item This non-local transformation generalizes the one between ABEP and BEP given in \cite{CGRSsu11}. Indeed, multiplying each $g_j$ by $-\exp(\sigma\lambda)$, letting $\lambda\to-\infty$ and reversing the order of sites gives the desired transformation.
			\end{itemize}
		\end{remark}
		Later on, we need the following property of the function $g$.
		\begin{proposition}
			Let $\lambda \geq 0$, then the function $g:X_c\to X_c$ is bijective. The inverse $g^{-1}$ has components given by
			\begin{align*}
				g_j^{-1}(\vec{y})= \frac{\cosh^{-1}\bigg(\cosh(\sigma \lambda) + \sum_{i=1}^jy_i\bigg)  - \cosh^{-1}\bigg(\cosh(\sigma \lambda) + \sum_{i=1}^{j-1}y_i\bigg)}{2\sigma},
			\end{align*}
			where $\cosh^{-1}\colon[1,\infty)\to[0,\infty)$ is  the inverse of $\cosh\colon[0,\infty)\to [1,\infty)$.
		\end{proposition}
		\begin{proof}
			By telescoping, we have
			\[
				\sigma \sum_{i=1}^j g_i(\vec{x}) = \cosh\big(\sigma\lambda + 2\sigma E_j^-(\vec{x})\big) - \cosh(\sigma\lambda).
			\]
			Since $\cosh\colon[0,\infty)\to [1,\infty)$ is bijective, we can apply $\cosh^{-1}\colon[1,\infty)\to [0,\infty)$ to obtain
			\[
				2\sigma E_j^-(\vec{x}) = \cosh^{-1}\bigg(\cosh(\sigma\lambda)+\sigma \sum_{i=1}^j g_i(\vec{x}) \bigg)-\sigma\lambda.
			\]
			Thus
			\begin{align*}
				x_j =&  E_j^-(\vec{x}) - E_{j-1}^-(\vec{x}) \\
				=& \frac{\cosh^{-1}\big( \cosh(\sigma\lambda) + \sigma\sum_{i=1}^j g_i(\vec{x}) \big)-\cosh^{-1}\big(\cosh(\sigma\lambda)+\sigma \sum_{i=1}^{j-1} g_i(\vec{x})\big)}{2\sigma}
			\end{align*}
			and $g^{-1} (g(\vec{x}))=\vec{x}$. A straightforward computation shows that $g(g^{-1}(\vec{y}))=\vec{y}$ as well. 
		\end{proof}
		\subsection{Reversibility dynamic ABEP}
		Since $\BEP{\vec{k}}$ is reversible with measure $\mu^{\mathsf{BEP}}$ from \eqref{eq:BEPrev}, we can use the function $g$ that transforms $\LABEP{\sigma,\vec{k},\lambda}$ into $\BEP{\vec{k}}$ to show that $\LABEP{\sigma,\vec{k},\lambda}$ is reversible as well. 
		\begin{theorem}
			Let $\lambda \geq 0$, then $\LABEP{\sigma,\vec{k},\lambda}$ has a family of reversible measures given by
			\begin{align*}
				\mu^{}_\mathsf{L}(\vec{x},\sigma,\vec{k},\lambda) = \prod_{j=1}^M \sinh_\sigma\big(\lambda+2E_j^-(\vec{x})\big) \frac{\Big( \cosh_\sigma\big(\lambda + 2 E_j^-(\vec{x})\big)-\cosh_\sigma\big(\lambda + 2 E_{j-1}^-(\vec{x})\big)\Big)^{k_j-1}}{\sigma^{k_j-1}\Gamma(k_j)}.
			\end{align*}
		\end{theorem}
		\begin{proof}
			We will show $\mu^{}_\mathsf{L}$ is a reversible measure by showing that the generator of $\LABEP{\sigma,\vec{k},\lambda}$ is self-adjoint with respect to $\mu^{}_\mathsf{L}$. That is, for any $f_1,f_2\in C^2(X_c)$ with compact support we have
			\begin{align}
				\int_{X_c} \Big[\genLABEP f_1\Big](\vec{x}) f_2(\vec{x}) \mu^{}_\mathsf{L}(\vec{x})\mathrm{d}\vec{x} = \int_{X_c} f_1(\vec{x}) \Big[\genLABEP f_2\Big](\vec{x}) \mu^{}_\mathsf{L}(\vec{x})\mathrm{d}\vec{x}.\label{eq:genLABEPselfadjoint}
			\end{align}
			Since $\mu^{\mathrm{BEP}}$ is a reversible measure for $\BEP{\vec{k}}$, we have for any $F_1,F_2\in C^2(X_c)$ with compact support,
			\begin{align*}
				\int_{X_c} \Big[\genBEP F_1\Big](\vec{y}) F_2(\vec{y}) \mathrm{d}\mu^{\mathrm{BEP}}(\vec{y}) \mathrm{d}\vec{y} = \int_{X_c} F_1(\vec{y}) \Big[\genBEP F_2\Big](\vec{y})\mu^{\mathrm{BEP}}(\vec{y}) \mathrm{d}\vec{y}.
			\end{align*}
			Doing a change of variables $\vec{y}=g(\vec{x})$ gives
			\begin{align}
				\begin{split}\int_{X_c} \Big[\genBEP F_1\Big](g(\vec{x})) &F_2(g(\vec{x})) \mu^{\mathrm{BEP}}(g(\vec{x}))|J_g(\vec{x})|\mathrm{d}\vec{x} \\
				&= \int_{X_c} F_1(g(\vec{x})) \Big[\genBEP F_2\Big](g(\vec{x})) \mu^{\mathrm{BEP}}(g(\vec{x}))|J_g(\vec{x})|\mathrm{d}\vec{x},\end{split}\label{eq:genBEPselfadjoint}
			\end{align}
			where $J_g(\vec{x})$ is the Jacobian of $g$ whose entries are given by
			\begin{align*}
				\big(J_g(\vec{x})\big)_{ij} =\bigg[\frac{\partial}{\partial x_j}g_i\bigg] (\vec{x}) = \begin{cases}
					0 &\text{ if } j > i,\\
					2 \sinh_\sigma\big(\lambda + 2 E_i\big) &\text{ if } j=i,\\
					 2 \sinh_\sigma\big(\lambda + 2 E_i\big) - 2 \sinh_\sigma\big( \lambda + 2E_{i-1}\big) &\text{ if } j<i,
				\end{cases}
			\end{align*}
			and $|J_g(\vec{x})|$ is the absolute value of the determinant of $J_g(\vec{x})$. Since $J_g(\vec{x})$ is a lower triangular matrix,
			\begin{align*}
				|J_g(\vec{x})| = \prod_{j=1}^M 2 \sinh_\sigma\big(\lambda + 2 E_j^-\big).
			\end{align*}
			Using \eqref{eq:LABEPBEPtransformation}, we get that \eqref{eq:genBEPselfadjoint} is equivalent to
			\begin{align*}
				\begin{split}\int_{X_c} \Big[\genLABEP \big(F_1\circ g\big)\Big](\vec{x}) &\big(F_2\circ g\big)(\vec{x}) \mu^{\mathrm{BEP}}(g(\vec{x}))|J_g(\vec{x})|\mathrm{d}\vec{x} \\
					&= \int_{X_c} \big(F_1 \circ g\big)(\vec{x}) \Big[\genLABEP \big(F_2\circ g\big)\Big](\vec{x}) \mu^{\mathrm{BEP}}(g(\vec{x}))|J_g(\vec{x})|\mathrm{d}\vec{x},\end{split}
			\end{align*}
			Since $g$ is invertible, we can take $F_1=f_1 \circ g^{-1}$ and $F_2=f_2\circ g^{-1}$. Therefore, $\genLABEP$ is self-adjoint with respect to the measure 
			\begin{align*}
				\mu^{\mathrm{BEP}}\big(g(\vec{x})\big)|J_g(\vec{x})|\mathrm{d}\vec{x} =& \prod_{j=1}^M 2\sinh_\sigma\big(\lambda+2 E_j^-(\vec{x})\big) \frac{g_j(\vec{x})^{k_j-1}}{\Gamma(k_j)}e^{-g_j(\vec{x})}\mathrm{d}x_j \\
				=&c\prod_{j=1}^M \sinh_\sigma\big(\lambda+2 E_j^-(\vec{x})\big) \frac{g_j(\vec{x})^{k_j-1}}{\Gamma(k_j)}\mathrm{d}x_j,
			\end{align*}
			where
			\begin{align*}
				c=2^Me^{\big(\cosh_\sigma(\lambda)-\cosh_\sigma(\lambda + 2 E_N(\vec{x}))\big)/\sigma}
			\end{align*}
			is a factor depending on $|\vec{x}|$, thus proving \eqref{eq:genLABEPselfadjoint}. 
		\end{proof}
		
	\subsection{Dualities of Dynamic ABEP}
		Let us give an overview of all the dualities related to BEP, ABEP and dynamic (A)BEP. We have already proven that dynamic ABEP is dual to SIP and used this to show that there is non-local transformation $g$ which transforms the BEP into dynamic ABEP. By taking an appropriate limit, we can also show that there is non-local transformation which transforms the BEP into dynamic BEP Therefore, any duality between BEP and another process can be lifted to duality between ABEP or dynamic (A)BEP and that process. Thus, we only have to look at dualities of BEP. We know that BEP is dual to itself \cite{ReSau2018} and to SIP \cite{GKR}. The duality between BEP and SIP can be generalized to a duality between BEP and dynamic SIP.
		\begin{proposition}
			Let
			\[
				P_n^{(\alpha,\beta)}(x)=\frac{(\alpha+1)_n}{n!}\rFs{2}{1}{-n,\alpha+\beta+n+1}{\alpha+1}{\frac{1-x}{2}}
			\]
			be the Jacobi polynomials. Then $\BEP{\vec{k}}$ is dual to $\RSIP{\vec{k},\rho}$ with duality function $\hat{P}_\mathrm{J}^v(\vec{y},\xi)$ given by
			\begin{align*}
				\hat{P}_{\mathrm{J}}^v(\vec{y},\xi;\vec{k})=&\prod_{j=1}^M(v+E_{j-1}^-(\vec{y}))^{\xi_j} \frac{\xi_j!}{(k_j)_{\xi_j}}P_{\xi_j}^{(k_j-1,h_{j+1}^+(\xi))}\bigg(1+\frac{2y_j}{v+E_{j-1}^-(\vec{y})}\bigg) \\ 
				=&\prod_{j=1}^M(v+E_{j-1}^-(\vec{y}))^{\xi_j} \rFs{2}{1}{-\xi_j,\xi_j+h_{j+1}^+(\xi)+k_j}{k_j}{\frac{-y_j}{v+E_{j-1}^-(\vec{y})}}.
			\end{align*}
		\end{proposition}
		\begin{proof}
			We start by obtaining a duality function between SIP and dynamic SIP from the duality function $\hat{P}^v_\mathrm{W}$ between $\LSIP{\vec{k},\lambda}$ and $\RSIP{\vec{k},\rho}$ given in \eqref{eq:dualityequationdynamicSIP}. By replacing $v$ by $v-\half\lambda$, dividing by $\lambda^{|\zeta|}$ and letting $\lambda\to\infty$ we obtain
			\begin{align*}
				\lim\limits_{\lambda\to\infty} \lambda^{-|\zeta|}\hat{P}^{v-\half \lambda}_{\mathrm{W}}(\zeta,\xi) =& \lim\limits_{\lambda\to\infty} \prod_{j=1}^M \lambda^{-\zeta_j} p^{}_{\mathrm{W}}(\zeta_j,\xi_j;h_{j-1}^-(\zeta),h_{j+1}^+(\xi);v-\tfrac12\lambda,\vec{k})\\
				=&\prod_{j=1}^M (\tfrac12\beta+v)_{\xi_j} \rFs{3}{2}{-\zeta_j,-\xi_j,\xi_j+h_{j+1}^+(\xi)+k_j}{k_j,\tfrac12\beta+v}{1},
			\end{align*} 
			where 
			\[
				\beta=h^-_{j-1,0}(\zeta)+h^+_{j+1}(\xi)+k_j+1.
			\]
			Let us now do the diffusion limit of SIP to BEP: take $\vec{y}\in X_c$ and define $\zeta^\varepsilon=\vec{y}/\varepsilon$. If we then replace $v$ by $v/\varepsilon$, multiply above duality function between SIP and dynamic SIP by $\varepsilon^{|\xi|}$ and take the limit $\varepsilon\to 0$ we obtain
			\begin{align*}
				\prod_{j=1}^{M} (v+E_{j-1}^-(\vec{y}))^{\xi_j} \rFs{2}{1}{-\xi_j,\xi_j+h_{j+1}^+(\xi)+k_j}{k_j}{\frac{-y_j}{v+E_{j-1}^-(\vec{y})}}.
			\end{align*}		
			The proposition is now proven similarly as Theorem \ref{thm:dualityDynABPSIP}.	
		\end{proof}
		\begin{remark}
			We can relate this duality function to known duality functions as follows. If we replace $v$ by $-\rho v$, divide the duality function by $\rho^{|\xi|}$, let $\rho\to\infty$ and divide the resulting duality function by $(-v)^{|\xi|}$, we obtain Laguerre polynomials which were found to be duality functions between SIP and BEP (\cite{FrGi2019,Gr2019,ReSau2018}). In our case, we have an extra parameter $v$ and the parameter $k_j$ may vary per site,
			\begin{align}
				\prod_{j=1}^M \rFs{1}{1}{-\xi_j}{k_j}{\frac{y_j}{v}}. \label{eq:LaguerreDuality}
			\end{align}
			If we multiply \eqref{eq:LaguerreDuality} by $(-v)^{|\xi|}$ and let $v\to\infty$ we obtain the classical duality function between SIP and BEP from \cite{GiKuReVa}, where the parameter $k_j$ may now vary per site,
			\begin{align*}
				\prod_{j=1}^M \frac{y_j^{\xi_j}}{(k_j)_{\xi_j}}.
			\end{align*}
			If we replace $v$ by $v/\varepsilon$ and $\xi_j$ by $x_j/\varepsilon$ in \eqref{eq:LaguerreDuality} and let $\varepsilon\to 0$, we obtain Bessel functions of the first kind as self-duality functions for BEP,
			\begin{align*}
				\prod_{j=1}^M \rFs{0}{1}{-}{k_j}{-\frac{x_jy_j}{v}},
			\end{align*}
			similar to the ones found in \cite{Gr2019} and \cite{ReSau2018}.
		\end{remark}
\section{The quantum algebra $\U_q(\mathfrak{su}(1,1))$ and Al-Salam--Chihara polynomials}\label{sec:Uq11}
In this section, we will state the necessary properties regarding the quantum algebra $\U_q(\mathfrak{su}(1,1))$ and the Al-Salam--Chihara polynomials required for 
\begin{itemize}
	\item constructing the generator of dynamic ASIP,
	\item proving duality between ASIP and dynamic ASIP with Al-Salam--Chihara polynomials as duality functions (Theorem \ref{thm:ASCduality}),
	\item proving duality between the left and right version of dynamic ASIP with the Askey-Wilson polynomials as duality functions (Theorem \ref{thm:AWduality}).
\end{itemize}  
We will first introduce $\U_q(\mathfrak{su}(1,1))$ and give a representation of this algebra which has close connections with ASIP and dynamic ASIP. Afterwards, we will connect this algebra with the Al-Salam--Chihara polynomials $p_\mathrm{A}^{}(n,x)$. 
\subsection{The algebra $\U_q(\mathfrak{su}(1,1))$}
We introduce $\mathcal U_q := \U_q\big(\su(1,1)\big)$, the quantized universal enveloping algebra of the Lie algebra $\mathfrak{su}(1,1)$. This is the unital, associative, complex algebra generated by $K$, $K^{-1}$, $E$, and $F$, subject to the relations
\begin{align}
	K K^{-1} = 1 = K^{-1}K, \quad KE = qEK, \quad KF= q^{-1}FK, \quad EF-FE =\frac{K^2-K^{-2}}{q-q^{-1}}.\label{eq:UqRelations}
\end{align}
The Casimir element
\begin{equation} \label{eq:Casimir}
	\Om = \frac{q^{-1} K^2 +qK^{-2}-2}{(q^{-1}-q)^2} + EF= \frac{q^{-1}K^{-2}+qK^2-2}{(q^{-1}-q)^2} +FE
\end{equation}
is a central element of $\U_q$, i.e. $\Omega X=X\Omega$ for all $X\in\U_q$.
The $*$-structure on $\U_q$ is an anti-linear involution defined on the generators by
\[
K^*=K, \quad E^*=-F, \quad F^* = -E, \quad (K^{-1})^* = K^{-1}.
\]
Note that the Casimir element is self-adjoint in $\U_q$, i.e.~$\Om^*=\Om$.\\ 

\noindent The comultiplication $\De:\U_q\to\U_q\otimes\U_q$ is a $*$-algebra homomorphism defined on the generators by
\begin{equation} \label{eq:comult}
	\begin{split}
		\De(K) = K \tensor K,&\quad  \De(E)= K \tensor E + E \tensor K^{-1}, \\
		\De(K^{-1}) = K^{-1} \tensor K^{-1},&\quad  \De(F) = K \tensor F + F \tensor K^{-1}.
	\end{split}
\end{equation}
The self-adjoint element $\Delta(\Omega)$ will be the generator of our Markov processes. It follows from \eqref{eq:Casimir} and \eqref{eq:comult} that
\begin{equation} \label{eq:De(OM)}
	\begin{split}
		\De(\Om) = &\frac{1}{(q^{-1}-q)^{2}} \big[ q( K^2 \tensor K^2 )+ q^{-1} (K^{-2} \tensor K^{-2}) -2 (1\tensor 1) \big] \\
		&+ K^2 \tensor FE +KE\tensor FK^{-1} + FK \tensor K^{-1}E + FE \tensor K^{-2}.
	\end{split}
\end{equation}
Another important element in $\U_q$ is the twisted primitive element\footnote{Note the subtle difference with $Y_\rho$ from \cite{GroeneveltWagenaarDyn}. Because of the different $*$-structure on $\U_q(\mathfrak{su}(1,1))$ compared to $\U_q(\mathfrak{su}(2))$, there is a minus sign in front of $FK$ to make sure $Y_\rho$ is self-adjoint. Moreover, the constant in front of the factor $K^2-1$ is slightly different. This is due to the form of the spectrum we will encounter in the representation of $\U_q(\mathfrak{su}(1,1))$ of $Y_\rho$.}  $Y_\rho$, defined by
\[
Y_\rho = q^{\frac12} EK - q^{-\frac12}FK + \mu_\rho(K^2-1), \quad \rho\in\R,
\]
where 
\[
	\mu_\rho = \frac{q^\rho+q^{-\rho}}{q^{-1}-q}.
\]
Then $Y_\rho^*=Y_\rho$ and it satisfies the co-ideal property,
\begin{align}
	\De(Y_\rho) = K^2 \tensor Y_\rho + Y_\rho \tensor 1.\label{eq:coproductYrho}
\end{align}
In Lie algebras, the comultiplication of an element $X$ is defined by $\Delta(X)=1\otimes X + X\otimes 1$. Note that $Y_\rho$ almost satisfies this. The $K^2$ in above equation will cause the asymmetry of the process.

\subsection{A representation of $\U_q$ related to $\ASIP{q,\vec{k}}$} In \cite{CGRSsu11} it was shown that the generator of $\ASIP{q,\vec{k}}$ can be realized by $\Delta(\Omega)$ when $k_j=k$ for all sites $j$. Non-surprisingly, this is still true when the $k_j$ may vary per site, which we will show here. Define $H_j$ to be the Hilbert space of functions $f \colon \Z_{\geq 0}\to\C$ with inner product induced by the orthogonality measure \eqref{eq:1siterevASIP},
\[
\langle f,g \rangle_{H_j} = \sum_{n=0}^{\infty} f(n) \overline{g(n)} w(n;q,k_j)u_j(\vec{k})^{-2n},
\]
where we recall that $w(n;q,k_j)$ is given by
\begin{align}
	w(n;q,k)=q^{-n(k-1)}\frac{(q^{2k};q^2)_n}{(q^2;q^2)_n}=w(\eta;q^{-1},k)
\end{align} 
and the factor
\begin{align}
	u_j(\vec{k})= q^{\half k_j - \sum_{l=1}^{j} k_j}
\end{align}
is present to prevent ground-state transformations later which is done in \cite{CGRSsu11}. Note that 
\[
	\prod_{j=1}^{M}w(\eta_j;q,k_j)u_j(\vec{k})^{-2\eta_j} = W(\eta;q,\vec{k}),
\] 
where $W$ is the reversible measure of ASIP given in \eqref{eq:revmeasASIP}. Let $H$ be the Hilbert space of functions on $X_d=Z_{\geq 0}^M$ with inner product given by
\begin{align*}
	\langle f,g\rangle_H = \sum_\eta f(\eta)\overline{g(\eta)}W(\eta;q,\vec{k}),
\end{align*} 
Since $H$ is isomorphic to the closure of the algebraic $M$-fold tensor product of $H_j$,
\begin{align}
	H\simeq\overline{H_1 \otimes H_2 \otimes  \ldots \otimes H_M},
\end{align}
We can identify functions in $$H_1\otimes H_2\otimes\ldots \otimes H_M$$ as functions on $X_d=Z_{\geq 0}^M$. 
Let $F_j$ be the subspace of $H_j$ consisting of functions with compact support. Then we define the (unbounded) $*$-representation $\pi_j$ of $\U_q$ on $H_j$ with $F_j$ as dense domain by
\begin{equation} \label{eq:representation}
	\begin{split}
		[\pi_j(K)f](n ) &= q^{n+\frac12 k_j} f(n), \\
		[\pi_j(E)f](n) &= u_j(\vec{k})[n]_q f(n-1), \\
		[\pi_j(F) f](n) & = -\frac{[n+k_j]_q}{u_j(\vec{k})}  f(n+1),\\
		[\pi_j(K^{-1}) f](n) &= q^{-n-\frac12 k_j}f(n).
	\end{split}
\end{equation}
One can easily verify that this is a $*$-representation, i.e. 
\[
	\langle\pi_j(X)f,g\rangle_{H_j}=\langle f,\pi_j(X^*)g\rangle_{H_j}
\]
for all $X\in \U_q$ and $f,g\in F_j$, by checking this for the generators $K,K^{-1},E$ and $F$.\\

\noindent Denote by $\pitensor$ the tensor product representation of $\pi_j$ and $\pi_{j+1}$,
\[
\pitensor(X\otimes Y)= \pi_j(X)\otimes \pi_{j+1}(Y),\qquad X,Y\in \U_q.
\] 
A direct calculation shows that the representation $\pitensor$ of $-\Delta(\Omega)$ is the generator of $\ASIP{q,\vec{k}}$ for sites $j$ and $j+1$ minus some constant, i.e.
\begin{align}
	\begin{split}[\pitensor (-\Delta(\Omega))f](\eta) =& c^+_j [f(\eta^{j,j+1})-f(\eta)] + c^-_{j+1} [f(\eta^{j+1,j})-f(\eta)] \\
		&- \big[\tfrac12(k_j+k_{j+1}-1)\big]_q^2f(\eta).\end{split}\label{eq:ASIPCasimir}
\end{align}
where
\begin{align*}
	c_j^+ &= q^{n_j+k_j-n_{j+1}-1}[n_j]_q[k_{j+1}+n_{j+1}]_q,\\
	c_{j}^-&= q^{-(n_{j}+k_{j}-n_{j-1}-1)}[n_{j}]_q[k_{j-1}+n_{j-1}]_q.
\end{align*}
Therefore, if we add the constant and sum over $j$, we get the generator of $\ASIP{q,\vec{k}}$,
\[
\genASIP= \sum_{j=1}^{M-1} \pitensor \big( \big[\tfrac12(k_j+k_{j+1}-1)\big]_q^2 -\Delta(\Omega)\big).
\]
$\genASIP$ is symmetric on functions with finite support with respect to the measure $W$, since the $\pi_j$ are $*$-representations, $\De$ is a $*$-homomorphism and $\Omega^*=\Omega$. Therefore, $W$ is a reversible measure for $\ASIPno$. \\

\noindent In the representation $\pi_j$, the Al-Salam--Chihara polynomials $p_\mathrm{A}^{}(\cdot,x)$ are eigenfunctions of $Y_\rho$ (see e.g. \cite{Groeneveltquantumaskey}),
\begin{align}
	[\pi_j(Y_\rho)p_\mathrm{A}^{}(\cdot,x)](n)= (\mu_{2x+k+\rho}-\mu_\rho)  p_{\mathrm{A}}^{}(n,x).\label{eq:ASCeigenfunction}
\end{align}
This can be proven by using the 3-term recurrence relation for the Al-Salam--Chihara polynomials (see \cite[(14.8.4)]{KLS}) and matching this with the explicit action of $\pi_j(Y_\rho)$,
\begin{align*}
	[\pi_j(Y_\rho)f](n) = u_j(\vec{k})q^{\half k_j+n-\half}[n]_qf(n-1)+\mu_\rho(q^{k+2n}-1)f(n)-q^{\half k_j +n +\half}\frac{[k+n]_q}{u_j(\vec{k})}f(n+1).
\end{align*}
\section{Algebraic construction of dynamic ASIP}\label{sec:dynasipalgebra}
	\subsection{Constructing the generator}
	The generator of dynamic ASIP is found in the same way as the generator of generalized dynamic ASEP in \cite{GroeneveltWagenaarDyn}: we transfer the action of $\pitensor(\De(\Om))$ from the $\eta$-variable to the $\xi$-variable, where we now use the Al-Salam--Chihara polynomials instead of the $q$-Krawtchouk polynomials. We will proceed in the following three steps. 
	\begin{enumerate}[label=(\arabic*)]
		\item In their usual action, we can let the operators $\pitensor(\Delta(Y_{h^{+}_{j+2}(\xi)}))$ and $\pitensor(\De(K^{-2}))$ act on the $\eta$ variable of the duality function $P^{}_{\mathsf{R}}(\eta,\xi)$ defined in Theorem \ref{thm:ASCduality}, which is a nested product of Al-Salam--Chihara polynomials. We will show that we can transfer these actions to be exclusively depending on the $\xi$ variable. This is the content of Lemma \ref{lem:etatoxi}. Conceptually, this is similar to the three-term recurrence relation of orthogonal polynomials $\{p_n(x)\}_{n=0}^\infty$,
		\[
			xp_n(x)=a_n p_{n-1}(x)+b_n p_n(x) + c_n p_{n+1}(x).
		\]
		Here, an action in the $x$-variable is transferred to an action only depending on the $n$-variable.	
		\item Then we show that $\Omega$ can be written in terms of a polynomial of degree 3 in $Y_\rho$ and $K^{-2}$, i.e. the latter two elements are `building blocks' for the Casimir $\Omega$. Consequently, $\Delta(\Omega)$ can be written in terms of $\Delta(Y_\rho)$ and $\De(K^{-2})$. 
		\item In the last step, we explicitly compute the action of $\pitensor(\Delta(\Omega))$ on $P^{}_{\mathsf{R}}(\eta,\xi)$ in the $\xi$ variable by combining the previous two steps. This will give the generator on sites $j$ and $j+1$ of $\RASIP{q,\vec{k},\rho}$, which is summarized in Theorem \ref{Thm:dynamicASIP}.
	\end{enumerate} 
	For step (1), we will show that we can transfer the $\eta$-dependent actions 
	\begin{align*}
		[\pitensor(\De(Y_{h^{+}_{k+2}(\xi)}))P^{}_{\mathsf{R}}(\cdot,\xi)](\eta)\qquad \text{and}\qquad [\pitensor(\De(K^{-2}))P^{}_{\mathsf{R}}(\cdot,\xi)](\eta)
	\end{align*} 
	to the $\xi$ variable. For the first, we show that $P_\mathsf{R}^{}(\cdot,\xi)$ are eigenfunctions of $\De(Y_{h^{+}_{k+2}(\xi)})$, which follows from the univariate case \eqref{eq:ASCeigenfunction} and the structure \eqref{eq:coproductYrho} of the twisted primitive element $Y_\rho$ with respect to comultiplication. For the second, we show that $P_\mathsf{R}^{}(\eta,\xi)$ satisfies a $9$-term recurrence relation, which follows from $q$-difference equations of the univariate Al-Salam--Chihara polynomials $p_\mathrm{A}^{}(n,x)$.
	\begin{lemma}\label{lem:etatoxi}
		The operator $\pitensor\Big(\De\Big(Y_{h^{+}_{k+2}(\xi)}\Big)\Big)$ acts as a multiplication operator on $P^{}_{\mathsf{R}}(\eta,\xi)$,
		\begin{align}
			\Big[\pitensor\Big(\De\Big(Y_{h^{+}_{k+2}(\xi)}\Big)\Big)P^{}_{\mathsf{R}}(\cdot,\xi)\Big](\eta)=&\ \big(\mu_{h^{+}_{j}(\xi)}-  \mu_{h^{+}_{j+2}(\xi)}\big)P^{}_{\mathsf{R}}(\eta,\xi).\label{eq:lemetatoxi1}
		\end{align}
		The operator $\pitensor\big(\De(K^{-2})\big)$ is a 9-term operator for $P^{}_{\mathsf{R}}(\eta,\xi)$ in the $\xi$-variable,
		\begin{align} 
			\big[\pitensor\big(\De\big(K^{-2}\big)\big)P^{}_{\mathsf{R}}(\cdot,\xi)\big](\eta) =\sum_{i_1,i_2=-1}^1  C_j(i_1,i_2) P^{}_{\mathsf{R}}(\eta,\xi+i_1\varepsilon_j + i_2 \varepsilon_{j+1}), \label{eq:lemetatoxi2}
		\end{align}
		where $\epsilon_j$ is the standard unit vector in $\R^M$ with 1 as $j$-th element and all $C_j(i_1,i_2)$ can be explicitly computed and do not depend on $\eta$. The terms relevant for this paper are given by
		\begin{align}
			&C_{j}(-1,1)= \alpha_q(h^+_{j}(\xi)) C_{j}^{\mathsf{R},+}(\xi),\label{eq:9termrates1} \\
			&C_j(1,-1)=\alpha_q(h_{j}^+(\xi)) C_{j+1}^{\mathsf{R},-}(\xi),\label{eq:9termrates2}
		\end{align}
		where
		\begin{align*}
			&\alpha_q(\rho)=  -\frac{q^2(q+q^{-1})(q-q^{-1})^2}{(1-q^{2-2\rho})(1-q^{2\rho+2})}
		\end{align*}
		and $C_{j}^{\mathsf{R},+}$, $C_{j+1}^{\mathsf{R},-}$ are the rates from $\RASIP{q,\vec{N},\rho}$ given in \ref{eq:ratesDynASIPR}.
	\end{lemma} 
	\begin{proof}
		Both \eqref{eq:lemetatoxi1} and \eqref{eq:lemetatoxi2} can be found in \cite{Groeneveltquantumaskey}, but with different notation and normalization. Let us first reduce to two sites by taking $M=2$, since the the general case follows easily from this. Then
		\begin{align*}
			P_\mathsf{R}^{}(\eta,\xi)&=p_\mathrm{A}^{}(\eta_1,\xi_1;h_2^+(\xi))p_\mathrm{A}^{}(\eta_2,\xi_2;h_3^+(\xi))\\
			&=p_\mathrm{A}^{}(\eta_1,\xi_1;\rho+2\xi_2+k_2)p_\mathrm{A}^{}(\eta_2,\xi_2;\rho) .
		\end{align*} 
		For \eqref{eq:lemetatoxi1}, the first equation of \cite[Proposition 5.5]{Groeneveltquantumaskey} with $N=j=2$, $u=1$ and $(q^{h_1^+(\xi)},q^{h_2^+(\xi)})$ instead of $(x_1,x_2)$ reads
		\[
			[\pitensortwo(\De(Y_{\rho}))P^{}_{\mathsf{R}}(\cdot,\xi)](\eta)= \big(\mu_{h_1^+(\xi)}-  \mu_{\rho}\big)P^{}_{\mathsf{R}}(\eta,\xi),
		\]
		which proves \eqref{eq:lemetatoxi1} for $M=2$. Alternatively, the result also follows from the fact that the 1-site duality function $p^{}_{\mathrm{A}}$ is an eigenfunction of $\pi_j(Y_\rho)$ and the structure of $Y_\rho$ with respect to the coproduct $\Delta$, see Appendix D of \cite{GroeneveltWagenaarDyn} for more details.\\
		
		\noindent From \cite[Proposition 5.10]{Groeneveltquantumaskey} we obtain \eqref{eq:lemetatoxi2} by taking 
		\[
			N=j=2, u=1
		\]
		 and variables $(q^{h_1^+(\xi)},q^{h_2^+(\xi)},q^\rho)$ instead of $(x_1,x_2,x_3)$. The terms that are relevant in this paper are the terms where $\xi\to\xi^{j,j+1}$ and $\xi\to\xi^{j+1,j}$. Since $h_1^+(\xi)$ does not change if a particle moves from site $1$ to $2$ or vice versa, we have to look at the terms in \cite{Groeneveltquantumaskey} corresponding to $(q^{2\nu_1}q^{h_1^+(\xi)},q^{\nu_2}q^{h_2^+(\xi)})$ where $\nu_1=0$ and $\nu_2=\pm1$. Let us explicitly calculate 
		\begin{align*}
			C_{1}(-1,1)= A_{k_2}\big(q^{h_2^+(\xi)},q^\rho\big)D_{k_1}\big(q^{h_1^+(\xi)},q^{h_2^+(\xi)}\big),
		\end{align*}
		where $A_{k_2}$ and $C_{k_1}$ can be found in \cite[Lemma 3.5]{Groeneveltquantumaskey} and \cite[Lemma 5.9]{Groeneveltquantumaskey} respectively. This gives
		\begin{align*}
			C_{1}(-1,1)&= \frac{q^{-k_2}(1-q^{2h_2^+(\xi)-2x_2})(1-q^{2x_2+2k_2})}{(1-q^{2h_2^+(\xi)})(1-q^{2h_2^+(\xi)+2})} \frac{q^{3-k_1}(q+q^{-1})(1-q^{2h_1^+(\xi)-2x_1})(1-q^{-2x_1})}{(1-q^{2h_1^+(\xi)+2})(1-q^{2-2h_1^+(\xi)})} \\
			&= -\frac{q^2(q+q^{-1})(q-q^{-1})^2}{(1-q^{2h_1^+(\xi)+2})(1-q^{2-2h_1^+(\xi)})}\frac{[h_2^+(\xi)-x_2]_q[x_2+k_2]_q[h_1^+(\xi)-x_1]_q[x_1]_q}{[h_2^+(\xi)]_q[h_2^+(\xi)+1]_q}\\
			&= -\frac{q^2(q+q^{-1})(q-q^{-1})^2}{(1-q^{2h_1^+(\xi)+2})(1-q^{2-2h_1^+(\xi)})} C_{j}^{\mathsf{R},+}(\xi).
		\end{align*}
		Similarly we obtain
		\begin{align*}
				C_{1}(1,-1)= -\frac{q^2(q+q^{-1})(q-q^{-1})^2}{(1-q^{2h_1^+(\xi)+2})(1-q^{2-2h_1^+(\xi)})} C_{j+1}^{\mathsf{R},-}(\xi).
		\end{align*}
		To obtain the general case, note that in the above steps we could pick the parameter $\rho$ freely and $\pitensor$ acts on $\eta_j,\eta_{j+1}$, while leaving $\eta_1,...,\eta_{j-1},\eta_{j+2},...,\eta_M$ invariant. Moreover, for $m\not\in\{j,j+1\}$ we have
		\[
			h_m^+\big(\xi^{j,j+1}\big)=h_m^+\big(\xi\big).
		\] 
		Therefore, we can inductively work from right to left to obtain \eqref{eq:lemetatoxi1} where we have to adjust $\rho$ every step: we have to add $2\xi_{j+1}+k_{j+1}$ to $\rho$ when going from sites $(j,j+1)$ to $(j-1,j)$. We do this for $j=M-1$ down to $j=1$. This exactly agrees with our definition of the height function $h_j^+$, see the proof of \cite[Lemma 7.1]{GroeneveltWagenaarDyn} for more details.
	\end{proof}
	Let us now turn to step (2): writing $\Omega$ as a degree 3 polynomial in $Y_\rho$ and $K^{-2}$. One can prove that 
	\begin{align}
		\Omega = -\frac{f(Y_\rho+\mu_\rho,K^{-2})}{(q+q^{-1})(q-q^{-1})^2}+ \frac{(q+q^{-1})K^{-2}}{(q-q^{-1})^2} +\mu_\rho\frac{Y_\rho+\mu_\rho}{q+q^{-1}} -\frac{ 2}{(q-q^{-1})^2} ,\label{eq:CasimirinYrhoK-2}
	\end{align}
	where $f \colon \U_q\times \U_q \to\U_q$ is the function given by
	\begin{align}
		f(A,B)=(q^2+q^{-2})ABA - A^2B-BA^2.
	\end{align}
	This identity in $\U_q$ can be shown by either a direct calculation using the commutation relations \eqref{eq:UqRelations}, or by observing that \eqref{eq:CasimirinYrhoK-2} is actually a relation in the degenerate version of the Askey-Wilson algebra $\text{AW}(3)$ generated by $Y_\rho$ and $K^{-2}$, see e.g. \cite{GranZhed} or \cite[Theorem 2.2]{GroeneveltWagenaar}\footnote{Note that in \cite{GroeneveltWagenaar}, the Casimir differs from $\Om$ by a scaling factor and an additive constant.}. Now, taking the coproduct on both sides we obtain
	\begin{align}
		\De(\Omega) = -\frac{f(\De(Y_\rho)+\mu_\rho,\De(K^{-2}))}{(q+q^{-1})(q-q^{-1})^2}+ \frac{(q+q^{-1})\De(K^{-2})}{(q-q^{-1})^2} + \mu_\rho\frac{\De(Y_\rho)+\mu_\rho}{q+q^{-1}}  -\frac{ 2}{(q-q^{-1})^2}  ,\label{eq:DeCasimirinYrhoK-2}
	\end{align}
	completing step (2). Note that we can pick our parameter $\rho$ freely, in particular we can take $\rho=h^{+}_{j+2}(\xi)$ for all $j=1,...,M-1$. \\
	\\
	Now we have all the ingredients to transfer the action of $\pitensor(\Delta(\Omega))$ on $P^{}_\mathsf{R}$ from the $\eta$-variable to the $\xi$ variable by combining \eqref{eq:DeCasimirinYrhoK-2} with Lemma \ref{lem:etatoxi}.
	\begin{theorem}\label{Thm:dynamicASIP}
		The action of the operator $\pitensor(\De(\Omega))$ on the $\eta$-variable of $P^{}_{\mathsf{R}}(\eta,\xi)$ can be transferred to the $\xi$-variable,
		\begin{align*}
		\begin{split}[\pitensor(-\De(\Omega))P^{}_{\mathsf{R}}(\cdot,\xi)](\eta) =&  C^{\mathsf R,+}_j [P^{}_{\mathsf{R}}(\eta,\xi^{j,j+1})-P^{}_{\mathsf{R}}(\eta,\xi)] +  C^{\mathsf R,-}_{j+1}[P^{}_{\mathsf{R}}(\eta,\xi^{j+1,j})-P^{}_{\mathsf{R}}(\eta,\xi) ] \\
			&- \big[\tfrac12(k_j+k_{j+1}-1)\big]_q^2 )P^{}_{\mathsf{R}}(\eta,\xi).\end{split}
	\end{align*}
	Here, $C^{\mathsf R,+}_j$ and $C^{\mathsf R,-}_{j+1}$ are the rates from $\RASIP{q,\vec{N},\rho}$ given in \eqref{eq:ratesDynASIPR}.
\end{theorem}
\begin{remark} 
	Note that the factor $\big[\tfrac12(k_j+k_{j+1}-1)\big]_q^2$ is the same as the one appearing in \eqref{eq:ASIPCasimir} for $\ASIP{q,\vec{k}}$.
\end{remark}
\begin{proof}
	The idea is to use \eqref{eq:DeCasimirinYrhoK-2} and Lemma \ref{lem:etatoxi} to transfer the action of $\pitensor(\De(\Om))$ on $P^{}_{\mathsf{R}}(\eta,\xi)$ from the $\eta$-variable to the $\xi$-variable. Applying $\pitensor$ to \eqref{eq:DeCasimirinYrhoK-2} we obtain for any $\rho\in\C$,
	\begin{align}
		\begin{split}[\pitensor(-\Delta(\Omega))=& \frac{f\big(\pitensor(\De(Y_{\rho})+\mu_\rho,\pitensor(\De(K^{-2}))\big)-(q+q^{-1})^2\pitensor(\De(K^{-2}))}{(q+q^{-1})(q-q^{-1})^2} \\
			&- \mu_\rho\frac{\pitensor(\De(Y_{\rho})+\mu_\rho)}{q+q^{-1}}  +\frac{ \pitensor(2)}{(q-q^{-1})^2}. \end{split}\label{eq:DeOmwrittenout}
	\end{align}
	By Lemma \ref{lem:etatoxi}, we have 
	\begin{align}
		[\pitensor(\De(Y_{h^{+}_{j+2}(\xi)})+\mu_{h^{+}_{j+2}(\xi)})P^{}_\mathsf{R}(\cdot,\xi)](\eta) = \mu_{h_j^+(\xi)}P^{}_\mathsf{R}(\eta,\xi)\label{eq:piDeYrho}
	\end{align}
	and that $\pitensor(\De(K^{-2}))$ acts on $P^{}_\mathsf{R}(\eta,\xi)$ as a $9$-term operator in the $\xi$-variable. We will show that only 3 terms of this 9-term operator are nonzero in \eqref{eq:DeOmwrittenout}. These are exactly the terms corresponding to $P^{}_\mathsf{R}(\eta,\xi+i_1\varepsilon_{j}-i_1\varepsilon_{j+1})$ where $i_1\in\{-1,0,1\}$. That is, the number of particles is preserved. Writing out \eqref{eq:DeOmwrittenout}, with $\rho=h^{+}_{k+2}(\xi)$, applied to $P^{}_\mathsf{R}$ we obtain
	\begin{align}
		\begin{split}
			\big[\pitensor(-\De(\Om))P^{}_{\mathsf{R}}(\cdot,\xi)\big](\eta) =&\sum_{i_1,i_2=-1}^1  \beta_j(i_1+i_2)C_j(i_1,i_2) P^{}_{\mathsf{R}}(\eta,\xi+i_1\varepsilon_j + i_2 \varepsilon_{j+1})\\
			&+ \left(\frac{\mu_{h^{+}_{j+2}(\xi)}\mu_{h^{+}_{j}(\xi)}}{q+q^{-1}}-\frac{2}{(q-q^{-1})^2} \right)P^{}_\mathsf{R}(\eta,\xi),
		\end{split}\label{eq:proof9term}
	\end{align}
	where $C_j(i_1,i_2)$ can be found in Lemma \ref{lem:etatoxi} and
	\begin{align*}
		\beta_j(m)=& \frac{(q^2+q^{-2})\mu_{h^{+}_{j}(\xi)}\mu_{h^{+}_{j}(\xi)+2m}-\mu_{h^{+}_{j}(\xi)}^2-\mu_{h^{+}_{j}(\xi)+2m}^2 -(q+q^{-1})^2}{(q+q {-1})(q-q^{-1})^2} .
	\end{align*}
	For all $\rho\in\R$ we have the identities
	\begin{align}
		&(q^2+q^{-2})\mu_\rho\mu_{\rho+2} - \mu_\rho^2-\mu_{\rho+2}^2 - (q+q^{-1})^2 =0,\\
		&(q^2+q^{-2})\mu_\rho\mu_{\rho} - \mu_\rho^2-\mu_{}^2 - (q+q^{-1})^2 =- q^{-2}(1-q^{2-2\rho})(1-q^{2\rho+2}),
	\end{align} 
	as readily verified by a direct computation. Therefore, $\beta_j(m)=0$ if $m= \pm 1$ and
	\[
		\beta_j(0) = 1/\alpha_q(h_j^+(\xi)),
	\] 
	where $\alpha_q(h_j^+(\xi))$ can be found in Lemma \ref{lem:etatoxi}. Therefore, only the terms in \eqref{eq:proof9term} with $P^{}_\mathsf{R}(\eta,\eta,\xi+i_1\varepsilon_j - i_1 \varepsilon_{j+1})$ remain. Thus, using \eqref{eq:9termrates1} and \eqref{eq:9termrates2}, the right-hand side of \eqref{eq:proof9term} becomes
	\begin{align}
			C_j^{\mathsf{R},+}\big[P^{}_\mathsf{R}(\eta,\xi^{j,j+1})-P^{}_\mathsf{R}(\eta,\xi)\big] + C_{j+1}^{\mathsf{R},-}\big[P^{}_\mathsf{R}(\eta,\xi^{j+1,j})-P^{}_\mathsf{R}(\eta,\xi)\big]+ \gamma P^{}_\mathsf{R}(\eta,\xi),\label{eq:DeOmForm}
	\end{align}
	for some constant $\gamma$ independent of $\eta$. To find this factor, observe that $P^{}_\mathsf{R}(0,\xi)=1$ for any $\xi$ since $p_\mathrm{A}^{}(0,x)=1$ by \eqref{eq:1siteASCq}. Thus, if we take $\eta=0$ in \eqref{eq:DeOmForm}, we obtain
	\begin{align*}
		[\pitensor(-\De(\Omega))P^{}_{\mathsf{R}}(\cdot,\xi)](0) = \gamma.
	\end{align*}
	Since $\pitensor(\De(\Om))$ is related to the generator of the $\ASIP{q,\vec{k}}$ process via \eqref{eq:ASIPCasimir}, we also have that
	\begin{align*}
		[\pitensor(-\De(\Omega))P^{}_{\mathsf{R}}(\cdot,\xi)](0) = -\big[\tfrac12(k_j+k_{j+1}-1)\big]_q^2,
	\end{align*}
	proving the theorem. 
\end{proof}
\subsection{Duality ASIP and Dynamic ASIP}
	From the way we constructed the generator of dynamic ASIP, we automatically get a duality between $\ASIP{q,\vec{k}}$ and $\RASIP{q,\vec{k},\rho}$ with the multivariate Al-Salam--Chihara polynomials $P^{}_\mathsf{R}$ as duality function, which was the content of Theorem \ref{thm:ASCduality}.
	\begin{proof}[Proof of Theorem \ref{thm:ASCduality}]
		We have to show that
		\[
			\big[\genASIP P^{}_\mathsf{R}(\cdot,\xi) \big] (\eta) = \big[\genRASIP P^{}_\mathsf{R}(\eta,\cdot) \big] (\xi),
		\]
		where $P^{}_\mathsf{R}$ are the multivariate Al-Salam--Chihara polynomials. Combining \eqref{eq:ASIPCasimir} with Theorem \ref{Thm:dynamicASIP}, we have
		\begin{align*}
			\big[\genASIP P^{}_{\mathsf{R}}(\cdot,\xi)\big](\eta) &=\sum_{j=1}^{M-1} c^+_j [P^{}_{\mathsf{R}}(\eta^{j,j+1},\xi)-P^{}_{\mathsf{R}}(\eta,\xi)] + c^-_{j+1} [P^{}_{\mathsf{R}}(\eta^{j+1,j},\xi)-P^{}_{\mathsf{R}}(\eta,\xi)].\\
			&= \sum_{j=1}^{M-1}[ \pitensor\big([\tfrac12(k_j+k_{j+1}-1)]_q^2-\De(\Om)\big)P^{}_{\mathsf{R}}(\cdot,\xi)](\eta) \\
			&= \sum_{j=1}^{M-1} C^{\mathsf R,+}_j [P^{}_{\mathsf{R}}(\eta,\xi^{j,j+1})-P^{}_{\mathsf{R}}(\eta,\xi)] +  C^{\mathsf R,-}_{j+1}[P^{}_{\mathsf{R}}(\eta,\xi^{j+1,j})-P^{}_{\mathsf{R}}(\eta,\xi) ]\\
			&= \big[\genRASIP P^{}_{\mathsf{R}}(\eta,\cdot)\big](\xi). \qedhere
		\end{align*}
	\end{proof}
\subsection{Duality $\LASIP{q,\vec{k},\lambda}$ and $\RASIP{q,\vec{k},\rho}$}
	In this section we will prove Theorem \ref{thm:AWduality}, which states the duality between $\LASIPno$ and $\RASIPno$ with the multivariate Askey-Wilson polynomials $P^{v}_\mathrm{AW}$ as duality function. By Theorem \ref{thm:ASCduality}, $\RASIPno$ is dual to $\ASIPno$, and by Corollary \ref{cor:dualityLASIPASIP}, $\ASIPno$ is dual to $\LASIPno$. Since $\ASIPno$ is reversible, $\RASIPno$ is also dual to $\LASIPno$ with duality functions given by  the scalar-product approach (see \cite[Proposition 4.1]{CFGGR}), which was sketched below Theorem \ref{thm:AWduality}. 
	\begin{proposition}
		The function $P^v\colon X_d\times X_d \to \R$ defined by
		\begin{align*}
			P^v(\zeta,\xi)= \sum_{\eta\in X_d}v^{|\eta|} P^{}_\mathsf{L}(\eta,\zeta) P^{}_\mathsf{R}(\eta,\xi) W(\eta),
		\end{align*}
		is a duality function between $\LASIPpar$ and $\RASIPpar$,
		\begin{align*}
			[\genLASIP P^v(\cdot,\xi)](\zeta) = 	[\genRASIP P^v(\zeta,\cdot)](\xi).
		\end{align*}
	\end{proposition}
	\begin{proof}
	Since $v^{|\eta|}$ only depends on the total number of particles which is invariant under the action of the generator, the function
	\[
		P^v_\mathsf{L}(\eta,\zeta)=v^{|\eta|}P^{}_\mathsf{L}(\eta,\zeta)
	\]
	is also a duality function between $\ASIP{q,\vec{k}}$ and $\LASIPpar$. Therefore,
	\begin{align*}
		\big[\genLASIP P^v(\cdot,\xi)\big](\zeta) &= \sum_{\eta\in X_d} \big[\genLASIP P_\mathsf{L}^v(\eta,\cdot)\big](\zeta) P^{}_\mathsf{R}(\eta,\xi)W(\eta) \\
		&= \sum_{\eta\in X_d} \big[\genASIP P_\mathsf{L}^v(\cdot,\zeta)\big](\eta) P^{}_\mathsf{R}(\eta,\xi)W(\eta).
	\end{align*}
	Since $\genASIP$ is symmetric with respect to $W$ for the duality functions $P^{}_\mathsf{L}$ and $P^{}_\mathsf{R}$ (see Proposition \ref{prop:genASIPsym}), above expression equals
	\begin{align*}
		\sum_{\eta\in X_d} P_\mathsf{L}^v(\eta,\zeta)  \big[\genASIP P^{}_\mathsf{R}(\cdot,\xi)\big](\eta)W(\eta)&= \sum_{\eta\in X_d} P_\mathsf{L}^v(\eta,\zeta)  \big[\genRASIP P^{}_\mathsf{R}(\eta,\cdot)\big](\xi)W(\eta) \\
		&=	\big[\genRASIP P^v(\zeta,\cdot)\big](\xi),
	\end{align*}
	proving the proposition.
	\end{proof}
	Let us now show that the duality function $P^v$ is a doubly nested product of Askey-Wilson polynomials. Doubly nested in the sense that the $j$-th product depends on $\xi_{j+1},\xi_{j+2},\ldots,\xi_M$ via $h_{j+1}^+(\xi)$ and also on $\zeta_1,\zeta_2,\ldots,\zeta_{j-1}$ via $h_{j-1}^-(\zeta)$. This result is a direct corollary of the following lemma, which shows that the inner product of the 1-site duality functions $p^{}_{\mathrm{A}}$ in base $q$ and $q^{-1}$ are Askey-Wilson polynomials, where the inner product is with respect to the 1-site reversible measure $w$ of $\ASIPpar$. This 
	\begin{lemma}\label{lem:AWsumASC}
		Let $p^{}_{\mathsf{AW}}(y,x;\lambda,\rho;v,k,q)$ be the 1-site duality function from \eqref{eq:1sitedualityAW} and let $q<1$. If $$|vq|< q^{2x+k+\rho-\lambda},$$ then we have the following summation formula between the Al-Salam--Chihara polynomials and Askey-Wilson polynomials,
		\begin{align}
			\begin{split}\frac{(vq^{2y+\lambda-\rho+k+1};q^2)_\infty}{(vq^{-2x+\lambda-\rho-k+1};q^2)_\infty} &p^{}_{\mathsf{AW}}(y,x;\lambda,\rho;v,k;q) \\
				&= \sum_{n=0}^\infty v^n p^{}_{\mathrm{A}}(n,y;\lambda;k,q^{-1}) p^{}_{\mathrm{A}}(n,x;\rho;k,q)w(n;q,k).\end{split}\label{eq:AWsumASC}
		\end{align}
		Consequently, $P^v$ is a doubly nested product of Askey-Wilson polynomials if $v< q^{2|\xi|+|\vec{k}|+\lambda-\rho-1}$,
		\begin{align}
			P^v(\zeta,\xi)= \prod_{j=1}^M \frac{(vq^{2\zeta_j+h^-_{j-1}(\zeta)-h^+_{j+1}(\xi)+k_j+1};q^2)_\infty}{(vq^{-2\xi_j+h^-_{j-1}(\zeta)-h^+_{j+1}(\xi)-k_j+1};q^2)_\infty} p^{}_{\mathsf{AW}}(\zeta_j,\xi_j;h^-_{j-1}(\zeta),h^+_{j+1}(\xi);v,k_j;q). \label{eq:PvAW}
		\end{align}
	\end{lemma}
	\begin{proof}
		Writing out the right hand side of \eqref{eq:AWsumASC} in terms of $\rphisempty{3}{2}$'s gives
		\begin{align*}
			\sum_{n=0}^\infty v^n q^{n(\lambda-\rho-k+1)}\frac{(q^{2k};q^2)_n}{(q^2;q^2)_n}&\rphis{3}{2}{q^{-2n}, q^{-2x}, q^{2x+2\rho+2k}}{q^{2k},0}{q^2,q^2} \\
			&\qquad \times \rphis{3}{2}{q^{2n}, q^{2y}, q^{-2y-2\lambda-2k}}{q^{-2k},0}{q^{-2},q^{-2}}.
		\end{align*}
		Showing that this sum of Al-Salam--Chihara polynomials are Askey-Wilson polynomials comes down to \cite[Lemma 4.6]{Groeneveltquantumaskey} and a change parameters. Indeed, \cite[Lemma]{Groeneveltquantumaskey} gives 
		\begin{align}
			\begin{split}c_1&\rphis{4}{3}{q^{-2m},abcdq^{2(m-1)},a\hat{x},a\hat{x}^{-1}}{ab,ac,ad}{q^2,q^2}\\
				&=\sum_{n=0}^\infty c_2  \rphis{3}{2}{q^{-2n}, \sigma q^{k}\hat{x}, \sigma q^{k}\hat{x}^{-1}}{q^{2k},0}{q^2,q^2} 
				\rphis{3}{2}{q^{2n}, \tau q^{-k}\hat{y}, \tau q^{-k}\hat{y}^{-1}}{q^{-2k},0}{q^{-2},q^{-2}},\end{split}\label{eq:sumchiharaAW}
		\end{align}
		where
		\begin{align*}
			c_1&=q^{-m(m-1)} \frac{(acq^{2m},bcq^{2m};q^2)_\infty}{(c\hat{x},c\hat{x}^{-1};q^2)_\infty} \frac{(ac,ad;q^2)_m}{(-ad)^m},\\
			c_2&=\frac{v^n q^{n(k-1)}}{(\sigma\tau)^{n}} 
			\frac{(q^{-2k};q^{-2})_n}{(q^{-2};q^{-2})_n},
		\end{align*}
		and $(a,b,c,d,\hat{y})=(q^k\sigma,q^k\sigma^{-1},qv\tau^{-1},qv^{-1}\tau^{-1},\tau q^{-k-2m})$. Taking 
		\begin{align*}
			 m&=y,\qquad \hat{x}=q^{2x+\rho+k},\qquad \tau=q^{-\lambda},\qquad \sigma=q^{\rho},
		\end{align*}
		\eqref{eq:sumchiharaAW} becomes
		\begin{align}
			\begin{split}c_1&\rphis{4}{3}{q^{-2y},q^{2y+2\lambda+2k},q^{-2x},q^{2x+2\rho+2k}}{q^{2k},vq^{\rho+\lambda+k+1},v^{-1}q^{\rho+\lambda+k+1}}{q^2,q^2}\\
				&=\sum_{n=0}^\infty c_2  \rphis{3}{2}{q^{-2n}, q^{-2x}, q^{2x+2\rho+2k}}{q^{2k},0}{q^2,q^2} 
				\rphis{3}{2}{q^{2n}, q^{2y}, q^{-2y-2\lambda-2k}}{q^{-2k},0}{q^{-2},q^{-2}},\end{split}\label{eq:sumchiharaAW2}
		\end{align}
		where 
		\begin{align*}
			c_1&=q^{-y(y-1)} \frac{(vq^{2y+\rho+\lambda+k+1},vq^{2y-\rho+\lambda+k+1};q^2)_\infty}{(vq^{2x+\rho+\lambda+k+1} ,vq^{-2x-\rho+\lambda-k+1};q^2)_\infty} \frac{(vq^{\rho+\lambda+k+1},v^{-1}q^{\rho+\lambda+k+1};q^2)_y}{(-v^{-1}q^{\rho+\lambda+k+1})^y}\\
			&= q^{-y(y-1)} \frac{(vq^{\rho+\lambda+k+1};q^2)_x(vq^{2y-\rho+\lambda+k+1};q^2)_\infty}{(vq^{-2x-\rho+\lambda-k+1};q^2)_\infty} \frac{(v^{-1}q^{\rho+\lambda+k+1};q^2)_y}{(-v^{-1}q^{\rho+\lambda+k+1})^y},  \\
			c_2&=\frac{v^n q^{n(k-1)}}{(q^{\rho-\lambda})^{n}} 
			\frac{(q^{-2k};q^{-2})_n}{(q^{-2};q^{-2})_n} = v^n q^{n(\lambda-\rho-k+1)}\frac{(q^{2k};q^2)_n}{(q^2;q^2)_n},
		\end{align*}
		and the sum converges absolutely if $|vq| < q^{2x+k+\rho-\lambda}$ (see \cite[section 7.1]{Groeneveltquantumaskey} for details). \\
		\indent Now \eqref{eq:AWsumASC} follows from the definition \eqref{eq:1sitedualityAW} of $p^{}_\mathsf{AW}$. To show that $P^v$ is a doubly-nested product of Askey-Wilson polynomials, let us write out its definition,
		\begin{align*}
			P^v(\zeta,\xi)&= \sum_{\eta\in X_d} v^{|\eta|}P^{}_{\mathsf{L}}(\eta,\zeta)P^{}_\mathsf{R}(\eta,\xi)W(\eta) \\
			&= \sum_{\eta_1}\sum_{\eta_2} \cdots \sum_{\eta_M} \prod_{j=1}^M v^{\eta_j}p^{}_\mathrm{A}(\eta_j;\zeta_j;h^-_{j-1};k_j,q^{-1}) p^{}_{\mathrm{A}}(\eta_j\xi_j;h^+_{j+1};k_j,q)w(\eta_j,q,k_j).
		\end{align*}
		Note that $\eta_j$ only appears in the $j$-th term of the product, since the terms with $u_j(\vec{k})$ exactly cancel. Therefore, we can interchange sum and product to obtain
		\begin{align*}
			P^v(\zeta,\xi)&=  \prod_{j=1}^M \sum_{\eta_j} v^{\eta_j}p^{}_\mathrm{A}(\eta_j;\zeta_j;h^-_{j-1};k_j,q^{-1}) p^{}_{\mathrm{A}}(\eta_j\xi_j;h^+_{j+1};k_j,q)w(\eta_j,q,k_j).
		\end{align*}
		Now \eqref{eq:PvAW} follows from taking $|v|<q^{2|\xi|+|\vec{k}|+\rho-\lambda-1}$ and applying the summation formula \eqref{eq:AWsumASC} $M$-times.
	\end{proof}
	\noindent Let us now show that $P^v$ and $P^v_{\mathsf{AW}}$ from \eqref{eq:dualityAW} are the same up to a factor depending on preserved quantities of the process and the stricter conditions on $q$ and $v$ can be lifted, proving Theorem \ref{thm:AWduality}.
	\begin{proof}[Proof of Theorem \ref{thm:AWduality}]
		By Lemma \ref{lem:AWsumASC} and the definition \eqref{eq:dualityAW} of $P^v_{\mathsf{AW}}$, we have
		\begin{align*}
			P^v(\zeta,\xi)&= \prod_{j=1}^M \frac{(vq^{2\zeta_j+h^-_{j-1}(\zeta)-h^+_{j+1}(\xi)+k_j+1};q^2)_\infty}{(vq^{-2\xi_j+h^-_{j-1}(\zeta)-h^+_{j+1}(\xi)-k_j+1};q^2)_\infty} p^{}_{\mathsf{AW}}(\zeta_j,\xi_j;h^-_{j-1}(\zeta),h^+_{j+1}(\xi);v,k_j;q) \\
			&= P^v_{\mathsf{AW}}(\zeta,\xi) \prod_{j=1}^M\frac{(vq^{2\zeta_j+h^-_{j-1}(\zeta)-h^+_{j+1}(\xi)+k_j+1};q^2)_\infty}{(vq^{-2\xi_j+h^-_{j-1}(\zeta)-h^+_{j+1}(\xi)-k_j+1};q^2)_\infty}.
		\end{align*}
		Now note that
		\begin{align*}
			\prod_{j=1}^{M} \frac{(vq^{2\zeta_j-h_{j+1}^+(\xi)+h_{j-1}^-(\zeta)+k_j+1};q^2)_\infty}{(vq^{-2\xi_j-h_{j+1}^+(\xi)+h_{j-1}^-(\zeta)-k_j+1};q^2)_\infty} &= \prod_{j=1}^{M} \frac{(vq^{h_{j}^-(\zeta)-h_{j+1}^+(\xi)+1};q^2)_\infty}{(vq^{h_{j-1}^-(\zeta)-h_{j}^+(\xi)+1};q^2)_\infty} \\
			&=\frac{(vq^{h_{M}^-(\zeta)-h_{M+1}^+(\xi)+1};q^2)_\infty}{(vq^{h_{0}^-(\zeta)-h_{1}^+(\xi)+1};q^2)_\infty}  \\
			&= \frac{(vq^{2|\zeta|+|\vec{k}|+\lambda-\rho+1};q^2)_\infty}{(vq^{-2|\xi|-|\vec{k}|+\lambda-\rho+1};q^2)_\infty},
		\end{align*}
		since $h_{M+1}^+(\xi)=\rho$, $h_{1}^+(\xi)= \rho+2|\xi|+|\vec{k}|$, $h_{0}^-(\zeta)=\lambda$ and $h_{M}^-(\zeta)=\lambda + 2|\zeta|+|\vec{k}|$. Therefore, the factor above depends only on the preserved quantities $\lambda,\rho, |\xi|,|\zeta|$ and $\vec{k}$, showing that $P^v_\mathsf{AW}$ is also a duality function between $\LASIPpar$ and $\RASIPpar$. Note that both sides of the duality relation
		\begin{align*}
			[\genLASIP P^v_{\mathrm{AW}}(\cdot,\xi)](\zeta)=[\genRASIP P^v_{\mathrm{AW}}(\zeta,\cdot)](\xi)
		\end{align*}
		are analytic in $v$, so the condition on $v$ can be lifted by using analytic continuation. Similarly, both sides are meromorphic functions in $q$ and analytic for $q>1$. Therefore, the duality relation is also valid for $q>1$.
	\end{proof}
	\begin{remark}
		The reason to work with $P^v_\mathsf{AW}$ as duality function instead of $P^v$ is purely aesthetic. Working with $P^v_\mathsf{AW}$ is easier since the factor in front of the Askey-Wilson polynomial is shorter and doesn't involve infinite shifted factorials. Therefore, we can more easily send $q\to q^{-1}$ in the duality function. Moreover, when taking limits of the duality function in Section \ref{sec:AsymmetricDegenerations}, the factors in front of the $q$-hypergeometric functions are shorter as well.
	\end{remark}
	
\appendix
	\newpage 
	\section{Dual orthogonality} \label{app:dualorth}
		Let $\{p_n(x)\}_{n=0}^\infty$ be a set of orthogonal polynomials whose orthogonality has a continuous and discrete part, i.e.
		\begin{align*}
			\int_\R p_n(x)p_m(x)w(x)\mathrm{d}x + \sum_{j=0}^\alpha p_n(x_j)p_m(x_j)w(x_j) = \frac{\delta_{n,m}}{h(n)},
		\end{align*}
		where $\alpha \in \Z_{\geq0}\cup \{\infty\}$ and $w(x_j)\neq 0$ for all $j=0,1,\ldots,\alpha$. The case where there is no continuous part is allowed, i.e. $w(x)=0$ on $\R\backslash\{1,2,\ldots,\alpha\}$. Let $H$ be the Hilbert space induced by this measure, i.e. it has inner product
		\begin{align*}
			\langle f,g\rangle_H = \int_\R f(x)g(x)w(x)\mathrm{d}x + \sum_{j=0}^\alpha f(x_j)g(x_j)w(x_j).
		\end{align*}
		If $\{p_n(x)\}_{n=0}^\infty$ is an orthogonal basis for $H$, then we have the dual orthogonality relation
		\begin{align*}
			\sum_{n=0}^{\infty} p_n(x_i)p_n(x_j) h(n) = \frac{\delta_{x_i,x_j}}{w(x_j)}
		\end{align*}
		for all $j=0,1,\ldots,\alpha$. \\
		
		\noindent To prove this, fix $x_j$ and consider the delta function 
		\begin{align*}
			\delta_{x_j}(x) = \delta_{x_j,x},
		\end{align*} 
		which is in $H$. Writing out $\delta_{x_j}$ in the orthogonal basis $\{p_n\}_{n=0}^\infty$ yields
		\begin{align}
			\delta_{x_j} = \sum_{n=0}^\infty \frac{\langle p_n,\delta_{x_j}\rangle^{}_H}{\langle p_n,p_n\rangle^{}_H}p_n,\label{eq:appconvdelta}
		\end{align}
		where the convergence is in $H$. Since
		\[
					\left(\delta_{x_j}(x_i) - \sum_{n=0}^N \frac{\langle p_n,\delta_{x_j}\rangle^{}_H}{\langle p_n,p_n\rangle^{}_H}p_n(x_i)\right)^2w(x_j) \leq \Big|\Big|\delta_{x_j}(x_i) - \sum_{n=0}^N \frac{\langle p_n,\delta_{x_j}\rangle^{}_H}{\langle p_n,p_n\rangle^{}_H}p_n(x_i)\Big|\Big|_H,
		\]
		which goes to zero when $N\to\infty$, we also get the pointwise equality
		\[
			\delta_{x_j}(x_i) = \sum_{n=0}^\infty \frac{\langle p_n,\delta_{x_j}\rangle^{}_H}{\langle p_n,p_n\rangle^{}_H}p_n(x_i).
		\]
		Using
		\begin{align*}
			&\langle p_n,\delta_{x_j}\rangle^{}_H = p_n(x_j)w(x_j),\\
			&\langle p_n,p_n\rangle^{}_H = 1/h(n), 
		\end{align*}	
		the dual orthogonality follows.
		
	\section{Interchanging sum and limit in Proposition \ref{prop:degenerate orthogonality qjac}} \label{app:interchangelimitsum}
	\noindent Let us prove that
	\begin{align*}
		\lim\limits_{\lambda\to\infty}&\sum_{\eta\in X_d} P^{vq^\lambda}_{\mathrm{AW}}(\eta,\xi;q) P^{v^{-1}q^{-\lambda}}_{\mathrm{AW}}(\eta,\xi';q) W_\mathsf{L}(\eta,q) \omega^{vq^\lambda}_{\mathrm{AW}}(|\eta|,|\xi|,\lambda,\rho,q)\\
		&=\sum_{\eta\in X_d} \lim\limits_{\lambda\to\infty} P^{vq^\lambda}_{\mathrm{AW}}(\eta,\xi;q) P^{v^{-1}q^{-\lambda}}_{\mathrm{AW}}(\eta,\xi';q) W_\mathsf{L}(\eta,q) \omega^{vq^\lambda}_{\mathrm{AW}}(|\eta|,|\xi|,\lambda,\rho,q).
	\end{align*}
	We will show that on the $1$-site duality function level we have
	\begin{align}
		p^{}_{\mathrm{AW}}(n,x;vq^\lambda) p^{}_{\mathrm{AW}}(n,x';v^{-1}q^{-\lambda}) w^{}_\mathsf{\mathsf{dyn}}(n;\lambda)  \sim q^{-n(2x+2x'+2\rho+2k)},\label{eq:orderpaw}
	\end{align}
	and for the factor $\omega_\mathrm{AW}^{vq^\lambda}$,
	\[
		\omega^{vq^{\lambda}}_{\mathrm{AW}}(|\eta|,|\xi|,\lambda,\rho,q) \sim 1,
	\]
	both for large $n$ and all $\lambda> T$ for some number $T$. Since $\xi\in X_{d,\rho}$, we have $2x+k+\rho < 0$. Hence, the sum
	\[
		\sum_{n=0}^{\infty} q^{-n(2x+2x'+2\rho+2k)}
	\] 
	converges absolutely. Therefore, interchanging the limit and sum is justified using the dominated convergence theorem and summing from right to left (i.e. start with $\eta_M$, then $\eta_{M-1}$, etc.). To prove \eqref{eq:orderpaw}, let us write out
	\begin{align*}
		p^{}_{\mathrm{AW}}(n,x;vq^\lambda) =   (v&q^{\rho+2\lambda+k+1};q^2)_x(vq^{-\rho-k-1};q^{-2})_n\\ 
		&\times \rphis{4}{3}{q^{-2n},q^{2n+2\lambda+2k},q^{-2x},q^{2x+2\rho+2k}}{q^{2k},vq^{\rho+2\lambda+k+1},v^{-1}q^{\rho+k+1}}{q^2,q^2}.
	\end{align*}
	We have
	\[
		(vq^{\rho+2\lambda+k+1};q^2)_x \sim 1,
	\]
	since
	\[
		|1-vq^{\rho+2\lambda+k+1}| < 1
	\]
	for $\lambda$ large enough. Using that
	\[
		(a;q^{-2})_n=(-a)^nq^{-n(n-1)}(a^{-1};q^2)_n,
	\]	
	we get
	\[
		(vq^{-\rho-k-1};q^{-2})_n \sim v^n q^{-n(\rho+k+n)}.
	\]
	For the $\rphisempty{4}{3}$ we use that
	\[
		\frac{(q^{2n+2\lambda+k};q^2)_j}{(vq^{\rho + 2\lambda +k +1};q^2)_j} \sim 1,
	\]
	to obtain
	\[
		\rphis{4}{3}{q^{-2n},q^{2n+2\lambda+2k},q^{-2x},q^{2x+2\rho+2k}}{q^{2k},vq^{\rho+2\lambda+k+1},v^{-1}q^{\rho+k+1}}{q^2,q^2} \sim q^{-2nx}.
	\]
	Thus
	\begin{align}
		p^{}_{\mathrm{AW}}(n,x;vq^\lambda) \sim v^nq^{-n(n+2x+\rho +k)}.\label{eq:orderpaw1}
	\end{align}
	Similarly,
	\begin{align}
		q^{2\lambda n}p^{}_{\mathrm{AW}}(n,x;v^{-1}q^{-\lambda}) \sim v^{-n}q^{-n(n+2x'+\rho +k)} \label{eq:orderpaw2}
	\end{align}
	Lastly, we have
	\begin{align*}
		q^{-2\lambda n}w^{}_\mathsf{\mathsf{dyn}}(n;\lambda) \sim q^{2n^2}.
	\end{align*}
	Combining this with \eqref{eq:orderpaw1} and \eqref{eq:orderpaw2} proves \eqref{eq:orderpaw}.
	\section{Details proof Theorem \ref{thm:ratesdynasipdynabep}}\label{app:ratesdynasipdynabep}
	\noindent In this section we prove the following equations,
			\begin{align}
		C^{\mathsf{L},+}_j(\vec{x}/\varepsilon) &= A_j(\vec{x},\sigma,\lambda)/\varepsilon^2 + \mathcal{O}(1/\varepsilon),\label{eq:ratesasipabep1app}\\
		C^{\mathsf{L},-}_{j+1}(\vec{x}/\varepsilon)&= A_j(\vec{x},\sigma,\lambda)/\varepsilon^2 + \mathcal{O}(1/\varepsilon),\label{eq:ratesasipabep2app}\\
		C^{\mathsf{L},-}_{j+1}(\vec{x}/\varepsilon)-C^{\mathsf{L},+}_j(\vec{x}/\varepsilon)&=  B_j(\vec{x},\sigma,\vec{k},\lambda)/\varepsilon + \mathcal{O}(1),\label{eq:ratesasipabep3app},
	\end{align}
	where $A_j$ and $B_j$ are the factors in the rates $\LABEP{\sigma,\vec{k},\lambda}$ from Definition \ref{def:dynabep},
	\begin{align*}
		A_j=& \frac{1}{\sigma^2}\sinh_\sigma(x_{j})\sinh_\sigma(x_{j+1})\frac{\sinh_\sigma( \lambda+2E_j^--x_j)\sinh_\sigma(\lambda+2E_j^-+x_{j+1})}{\sinh_\sigma(\lambda+2E_j^-)^2},\\
		B_j=&\frac{1}{\sigma}\bigg[k_j\sinh_\sigma(x_{j+1})\frac{\sinh_\sigma(\lambda+2E_j^-+ x_{j+1})}  {\sinh_\sigma(\lambda+2E_j^-)} -k_{j+1}\sinh_\sigma( x_{j})\frac{\sinh_\sigma(\lambda+2E_j^--x_j)}   {\sinh_\sigma(\lambda+2E_j^-)} \\
		&\hspace{0.3cm} -2\sinh_\sigma (x_j)\sinh_\sigma( x_{j+1})\frac{\cosh_\sigma(\lambda+2E_j^-)\sinh_\sigma(\lambda+2E_j^--x_j)\sinh_\sigma(\lambda+2E_j^-+x_{j+1})}{\sinh_\sigma(\lambda+2E_j^-)^3}\bigg].
	\end{align*}
	Recall from \eqref{eq:ratesdynASIPrewritten}  that we can write the rates of $\LASIPpar$ as
	\begin{align}
		\begin{split}
			C_j^{\mathsf{L},+}&=[\zeta_j]_q[\zeta_{j+1}+k_{j+1}]_q\frac{[h^-_{j-1}+\zeta_{j}]_q[h^-_{j}+\zeta_{j+1}]_q}{[h^-_j]_q[h^-_j-1]_q},\\
			C_j^{\mathsf{L},-}&=[\zeta_{j}]_q[\zeta_{j-1}+k_{j-1}]_q\frac{[h^-_{j-1}-\zeta_{j-1}]_q[h^-_j-\zeta_{j}]_q}{[h^-_{j-1}]_q[h^-_{j-1}+1]_q}.
		\end{split}
	\end{align}
	Therefore, the rate $C^{\mathsf{L},+}_j(\vec{x}/\varepsilon)$ of $\LASIP{1-\varepsilon\sigma,\vec{k},\lambda/\varepsilon}$ is equal to
	\begin{align*}
		&\left[\frac{x_j}{\varepsilon}\right]_{1-\varepsilon\sigma}\left[\frac{x_{j+1}}{\varepsilon}+k_{j+1}\right]_{1-\varepsilon\sigma}\frac{\left[\frac{x_j+\lambda}{\varepsilon} +\sum_{i=1}^{j-1} 2\frac{x_{i}}{\varepsilon}+k_i\right]_{1-\varepsilon\sigma}\left[\frac{x_{j+1}+\lambda}{\varepsilon} +\sum_{i=1}^{j} 2\frac{x_{i}}{\varepsilon}+k_i\right]_{1-\varepsilon\sigma}}{\left[\frac{\lambda}{\varepsilon} +\sum_{i=1}^{j} 2\frac{x_{i}}{\varepsilon}+k_i\right]_{1-\varepsilon\sigma}\left[\frac{\lambda}{\varepsilon}-1 +\sum_{i=1}^{j} 2\frac{x_{i}}{\varepsilon}+k_i\right]_{1-\varepsilon\sigma}}\\
		&=\frac{\sinh_{\alpha}\!\big(\frac{x_j}{\varepsilon}\big)\sinh_{\alpha}\!\big(\frac{x_{j+1}}{\varepsilon}+k_{j+1}\big)\sinh_{\alpha}\!\big(\frac{x_{j}+\lambda}{\varepsilon}  + 2E_{j-1}^-\big(\frac{\vec{x}}{\varepsilon}+\vec{k}\big)\big)\sinh_{\alpha}\!\big(\frac{x_{j+1}+\lambda}{\varepsilon} + 2E_j^-\big(\frac{\vec{x}}{\varepsilon}+\vec{k}\big)  \big)}{\sinh_{\alpha}\!\big(1\big)^2\sinh_{\alpha}\!\big(\frac{\lambda}{\varepsilon} + 2E_j^-\big(\frac{\vec{x}}{\varepsilon}+\vec{k}\big) \big)\sinh_{\alpha}\!\big(\frac{\lambda}{\varepsilon}-1 + 2E_j^-\big(\frac{\vec{x}}{\varepsilon}+\vec{k}\big) \big)},
	\end{align*}
	where $\alpha=\ln(1-\varepsilon\sigma)$. By the summation formula for a geometric series, we have for $|\varepsilon|<1/|\sigma|$, 
	\begin{align*}
		2\sinh\!\big(\!\ln(1-\varepsilon\sigma)\big)=1-\varepsilon\sigma-(1-\varepsilon\sigma)^{-1}= -2\sigma\varepsilon +\mathcal{O}(\varepsilon^2).
	\end{align*}
	Therefore,
	\begin{align}
		\frac{1}{\sinh_\alpha(1)^2} = \frac{1}{\sigma^2\varepsilon^2}+\mathcal{O}(1/\varepsilon),\label{eq:1/sinh}
	\end{align}
	where we used that the meromorphic function\footnote{We use the principal branch of the logarithm.} $1/\sinh_\alpha(1)^2$ has a Laurent series around $\varepsilon=0$.
	Let us define the function $g_1$ as the rate $C^{\mathsf{L},+}_{j}(x/\varepsilon)$ without the factor $\sinh_\alpha(1)^2$ in the denominator, i.e.
	\[
	g_1(\varepsilon)=\frac{\sinh_{\alpha}\!\big(\frac{x_j}{\varepsilon}\big)\sinh_{\alpha}\!\big(\frac{x_{j+1}}{\varepsilon}+k_{j+1}\big)\sinh_{\alpha}\!\big(\frac{x_{j}+\lambda}{\varepsilon} + E_{j-1}\big(\frac{2x}{\varepsilon}+\vec{k}\big)\big)\sinh_{\alpha}\!\big(\frac{x_{j+1}+\lambda}{\varepsilon} + E_{j}\big(\frac{2x}{\varepsilon}+\vec{k}\big)  \big)}{\sinh_{\alpha}\!\big(\frac{\lambda}{\varepsilon} + E_{j}\big(\frac{2x}{\varepsilon}+\vec{k}\big) \big)\sinh_{\alpha}\!\big(\frac{\lambda}{\varepsilon}-1 + E_{j}\big(\frac{2x}{\varepsilon}+\vec{k}\big) \big)}.
	\] 
	Since
	\begin{align}
		\lim_{\varepsilon\to 0}\ \ln(1-\varepsilon\sigma\big)\big(a+b/\varepsilon\big)= -b\sigma,\label{eq:limitproofdynasipabep}
	\end{align}
	$g_1$ is analytic in a neighborhood around $0$, which has a Taylor series. With \eqref{eq:ratesasipabep1app}, \eqref{eq:ratesasipabep2app}, \eqref{eq:ratesasipabep3app}, and \eqref{eq:1/sinh} in mind, we are interested in the first two terms,
	\[
	g_1(\varepsilon)=g_1(0)+g_1'(0)\varepsilon + \mathcal{O}(\varepsilon^2).
	\]
	From \eqref{eq:limitproofdynasipabep} we obtain
	\begin{align*}
		g_1(0) &= 4\sinh_\sigma(x_j)\sinh_\sigma( x_{j+1})\frac{\sinh_\sigma(\lambda+2E_{j-1}(x) + x_j\big)\sinh_\sigma\!\big(\lambda+E_j(x)+x_{j+1}\big)}{\sinh_\sigma\!\big(\lambda+2E_j(x)\big)^2}\\
		&= \sigma^2 A_j(x,\sigma,\lambda),
	\end{align*}
	which combined with \eqref{eq:1/sinh} proves \eqref{eq:ratesasipabep1app}. Entirely similar one can show that
	\[
	C^{\mathsf{L},-}_{j+1}(x/\varepsilon) = \frac{g_2(\varepsilon)}{\sinh_\alpha(1)^2},
	\]
	where 
	\[
	g_2(\varepsilon)=\frac{\sinh_{\alpha}\!\big(\frac{x_{j+1}}{\varepsilon}\big)\sinh_{\alpha}\!\big(\frac{x_j}{\varepsilon}+k_j\big)\sinh_{\alpha}\!\big(\frac{\lambda-x_j}{\varepsilon}  + E_{j}\big(\frac{2x}{\varepsilon}+\vec{k}\big)\big)\sinh_{\alpha}\!\big(\frac{\lambda-x_{j+1}}{\varepsilon} + E_{j+1}\big(\frac{2x}{\varepsilon}+\vec{k}\big)  \big)}{\sinh_{\alpha}\!\big(\frac{\lambda}{\varepsilon} + E_{j}\big(\frac{2x}{\varepsilon}+\vec{k}\big) \big)\sinh_{\alpha}\!\big(\frac{\lambda}{\varepsilon}+1 + E_{j}\big(\frac{2x}{\varepsilon}+\vec{k}\big) \big)}.
	\] 
	Then again
	\[
	g_2(\varepsilon)=g_2(0)+g_2'(0)\varepsilon+\mathcal{O}(\varepsilon^2),
	\]
	where 
	\[
	g_2(0)=\sigma^2 A_j(x,\sigma,\lambda),
	\]
	proving \eqref{eq:ratesasipabep2app}. Let us now turn to \eqref{eq:ratesasipabep3app}. Since $g_1(0)=g_2(0)$, we have
	\begin{align}
		C^{\mathsf{L},-}_{j+1}(x/\varepsilon)-C^{\mathsf{L},+}_{j}(x/\varepsilon) = \frac{g_2(\varepsilon)-g_1(\varepsilon)}{\sinh_\alpha(1)^2} = \frac{(g'_2(0)-g'_1(0))}{4\sigma^2\varepsilon} + \mathcal{O}(1).\label{eq:Btermratedynabep}
	\end{align}
	Therefore, we have to compute $g_1'(0)$ and $g_2'(0)$. Let us define the numerator en denominator of $g_1$ and $g_2$ by
	\[
	g_j(\varepsilon)=\frac{n_j(\varepsilon)}{d_j(\varepsilon)}.
	\]
	Since $n_1(0)=n_2(0)$ and $d_1(0)=d_2(0)$, the quotient rule gives us
	\begin{align}
		g_2'(0)-g_1'(0) = \frac{n_2'(0)-n_1'(0)}{d_1(0)}+ \frac{n_1(0)(d_1'(0)-d_2'(0))}{d_1(0)^2}.\label{eq:quotientrule}
	\end{align}
	Using
	\begin{align*}
		\frac{\mathrm{d}}{\mathrm{d}\varepsilon}\sinh_{\ln(1-\varepsilon\sigma)}(a+b/\varepsilon)\Big\rvert_{\varepsilon=0} &=- \cosh_{\ln(1-\varepsilon\sigma)}(a+b/\varepsilon)\bigg(\frac{\sigma(a+b/\varepsilon)}{1-\varepsilon\sigma}+\frac{b\ln(1-\varepsilon\sigma)}{\varepsilon^2}\bigg)\Big\rvert_{\varepsilon=0}\\
		&=-\cosh_\sigma(b)(a\sigma -\tfrac12 b\sigma^2),
	\end{align*}
	we can compute $n_2'(0)-n_1'(0)$ and see that only the terms with $k_j$ and $k_{j+1}$ do not cancel,
	\begin{align*}
		n_2'(0)-n_1'(0)&= k_j\sigma\cosh_\sigma(x_j^\varepsilon)\sinh_\sigma(x_{j+1}^\varepsilon)\sinh_\sigma(\lambda+2E_j(x)-x_j^\varepsilon)\sinh_\sigma(\lambda+2E_{j+1}(x)-x_{j+1}^\varepsilon) \\
		&+k_j\sigma\cosh_\sigma(\lambda+2E_j(x)-x_j^\varepsilon)\sinh_\sigma(x_{j+1}^\varepsilon)\sinh_\sigma(x_j^\varepsilon)\sinh_\sigma(\lambda+2E_{j+1}(x)-x_{j+1}^\varepsilon)\\
		&-k_{j+1}\sigma\cosh_\sigma(x_{j+1}^\varepsilon)\sinh_\sigma(x_{j}^\varepsilon)\sinh_\sigma(\lambda+2E_{j-1}(x)+x_j^\varepsilon)\sinh_\sigma(\lambda+2E_{j}(x)+x_{j+1}^\varepsilon) \\
		&+k_{j+1}\sigma \cosh_\sigma(\lambda+2E_{j+1}(x)-x_{j+1}^\varepsilon)\sinh_\sigma(x_{j+1}^\varepsilon)\sinh_\sigma(x_j^\varepsilon)\sinh_\sigma(\lambda+2E_{j}(x)-x_{j}^\varepsilon).
	\end{align*}
	If we now use 
	\begin{align*}
		2\cosh(x)\sinh(y)=\sinh(x+y)-\sinh(x-y),
	\end{align*}
	we obtain
	\begin{align*}
		n_2'(0)-n_1'(0)=\sigma k_j \sinh_\sigma(\lambda+2E_j(x)\sinh_\sigma(x_{j+1}^\varepsilon)\sinh_\sigma(\lambda+2E_j(x)+x_{j+1}^\varepsilon)) \\
		-\sigma k_{j+1}\sinh_\sigma(\lambda+2E_j(x)\sinh_\sigma(x_{j}^\varepsilon)\sinh_\sigma(\lambda+2E_j(x)-x_{j}^\varepsilon)).
	\end{align*}
	Since 
	\[
	d_1(0)=\sinh_\sigma(\lambda+2E_j(x))^2,
	\]
	the equation \eqref{eq:ratesasipabep3app} for the terms of $B_j(x,\sigma,\vec{k},\lambda)$ with $k_j$ and $k_{j+1}$ follows from \eqref{eq:Btermratedynabep} and \eqref{eq:quotientrule}. The last term of $B_j(x,\sigma,\vec{k},\lambda)$ then follows from computing $d_1'(0)-d_2'(0)$ and $n_1(0)$ and using again \eqref{eq:Btermratedynabep} and \eqref{eq:quotientrule}. We have
	\begin{align*}
		d_1'(0)-d_2'(0) = -2\sigma \cosh_\sigma(\lambda+2E_j(x))\sinh_\sigma(\lambda+2E_j(x))
	\end{align*}
	and 
	\begin{align*}
		n_1(0)= \sinh_\sigma(x_j)\sinh_\sigma( x_{j+1})\sinh_\sigma(\lambda+2E_{j}(x) - x_j\big)\sinh_\sigma\!\big(\lambda+E_j(x)+x_{j+1}\big),
	\end{align*}
	and \eqref{eq:ratesasipabep3app} follows.

\end{document}